\documentclass[12pt]{article}
\usepackage{amssymb}
\usepackage{amsmath}
\usepackage{amsfonts}
\usepackage{amsthm}
\usepackage{xcolor}
\usepackage{float}
\usepackage{hyperref}

\date{}

\newtheorem{Theorem}{Theorem}[section]
\newtheorem{Lemma}{Lemma}[section]
\newtheorem{Remark}{Remark}[section]
\newtheorem{Corollary}{Corollary}[section]
\newtheorem{Definition}{Definition}[section]

\numberwithin{equation}{section} \theoremstyle{plain}

\def\R{{\textbf{R}}}

\def\p{\partial}

\def\R{\mathbb R}
\def\N{\mathbb N}

\def\d{\mathrm{d}}

\def\va{\varepsilon}

\title { Threshold phenomena for least energy solutions of a doubly critical Neumann problem with a critical absorption term
	 \thanks {The research was supported  by   National
Key R\&D Program of China (No. 2023YFA1010002) and the Natural Science Foundation of China  (No.12271196, No. 12101530), the Natural Science Foundation of Henan Province (No. 262300421235), Sponsored by  the Program for Science \& Technology Innovation Talents in Universities of Henan Province (No. 26HASTIT040).}
	 }

\author{Yinbin Deng$^{1,}$ $^2$, Yulin Shi $^{1,}$\thanks{Corresponding author:
 shiyulin@mails.ccnu.edu.cn} \ and Pengyan Wang$^2$    \\
 {\small
    $^1$Key Laboratory of Nonlinear Analysis and Applications,
  }  \\
  {\small School of Mathematics and Statistics,   Central China Normal University,     Wuhan, China} \\
    {\small
      $^2$ School of Mathematics and Statistics, Xinyang Normal University,  Xinyang, China    } \\
    }

\begin{document}
	\baselineskip= 15pt
	\maketitle
	
	\label{firstpage}
	\date{}
	\maketitle

		\begin{abstract}
					
		We investigate the existence and nonexistence of least energy solutions for a doubly critical Neumann problem with a critical absorption term:
		\begin{equation}\label{eq1.0.100}
			\begin{cases}
				-\Delta u + \lambda u = u^{2^* - 1}  -\alpha u^{2^\sharp-1}   &\text{in }\Omega,\\
				u>0   &\text{in }\Omega, \\
				\nabla u \cdot \nu =  u^{2^\sharp -1}  &\text{on }\p\Omega,
				\end{cases}
			\end{equation}
		where $\Omega\subset\mathbb{R}^n$ is a smooth bounded domain with $n\ge5$,
		$\lambda>0$, $\alpha\ge0$, and 
		$2^*=\frac{2n}{n-2}$, $2^\sharp=\frac{2(n-1)}{n-2}$ denote the critical Sobolev exponent and the critical trace exponent, respectively. The simultaneous presence of the interior and boundary critical exponents together with this critical absorption term creates a new interaction between competing concentration mechanisms. In particular, the absorption term has the same critical trace exponent as the boundary nonlinearity, but with the opposite sign and acting in the interior of $\Omega$. Consequently, it competes with the boundary mean curvature correction at the same asymptotic order in the energy expansion, leading to a sharp threshold phenomenon.
		
		Combining several analytic  techniques and geometric tools, we  prove the existence
			of a threshold value $\alpha_{0}=\alpha_{0}(\lambda,\Omega)\in(0,+\infty)$
			such that \eqref{eq1.0.100} admits a least energy solution if
			$\alpha<\alpha_{0}$,  and no least energy solution if
			$\alpha>\alpha_{0}$.
		  Moreover,   \eqref{eq1.0.100}  also has a least energy solution at $\alpha =\alpha_0$ provided $\alpha_0> C(n)\underset{\partial \Omega}{\max} H$, where  $C(n)$  is a positive constant depending only on $n$ and $H$ is the mean
		curvature on $\partial\Omega$.
		
		\

		
		{\bf 2000 MR Subject Classification:} \ 35Q35; 35B33; 76D05.
		\end{abstract}
	
	\bigskip
	
	\section{Introduction and main results}
	In this paper,	 we are concerned with the doubly critical Neumann problem
	\begin{equation}\label{eq:main}
		\begin{cases}
			-\Delta u+\lambda u=u^{2^{*}-1}-\alpha u^{2^{\sharp}-1}
			& \text{in }\Omega,\\
             u>0                                                  & \text{in }\Omega,\\
			\frac{\partial u}{\partial\nu}=u^{2^{\sharp}-1}                   & \text{on }\partial\Omega,
		\end{cases}
	\end{equation}
	where $\lambda>0$ is fixed, $\alpha\ge 0$ is a perturbation parameter, and
	\[
	2^{*}=\frac{2n}{n-2},\qquad
	2^{\sharp}=\frac{2(n-1)}{n-2}
	\]
	are  the critical Sobolev exponents for the embedding
	$H^{1}(\Omega)\hookrightarrow L^{2^{*}}(\Omega)$ and the critical trace
	exponent for $H^{1}(\Omega)\hookrightarrow L^{2^{\sharp}}(\partial\Omega)$  respectively. Our main result establishes a sharp threshold
	$\alpha_{0}=\alpha_{0}(\lambda,\Omega)\in(0,+\infty)$ governing existence and
	nonexistence of least energy solutions. Throughout this paper, $\Omega\subset\R^n$ denotes a smooth bounded  domain
		with  $n\ge 5$, and $\nu$ stands for the
		unit outward normal field on $\partial\Omega$.

	Problem~\eqref{eq:main} arises naturally from conformal geometry.
	A central question in conformal geometry is whether a compact
	Riemannian manifold $(M,g)$ of dimension $n\ge 3$ with boundary
	admits a conformal metric $\widetilde g=u^{4/(n-2)}g$ realizing
	constant scalar curvature in the interior and constant mean
	curvature on the boundary simultaneously --
	the Escobar-Yamabe problem~\cite{Escobar1990,Escobar1996}.
The conformal transformation rules
\begin{equation}\label{eq:conf_transform}
	R_{\widetilde g}
	=u^{-\frac{n+2}{n-2}}
	\!\left(-\frac{4(n-1)}{n-2}\Delta_{g}u+R_{g}u\right),
	\qquad
	h_{\widetilde g}
	=u^{-\frac{n}{n-2}}
	\!\left(\frac{2}{n-2}\frac{\partial u}{\partial\nu}+h_{g}u\right)
\end{equation}
show that prescribing $R_{\widetilde g}=K$ and
	$h_{\widetilde g}=H$ reduces to finding a positive solution of
\begin{equation}
	\label{eq:yamabe}
	\begin{cases}
		\displaystyle
		-\frac{4(n-1)}{n-2}\Delta_{g}u
		+R_{g}u=K\,u^{2^{*}-1}
		& \text{in }M,\\[6pt]
		\displaystyle
		\frac{2}{n-2}\frac{\partial u}{\partial\nu}+h_{g}u
		=H\,u^{2^{\sharp}-1}
		& \text{on }\partial M.
		\end{cases}
	\end{equation}
The two critical exponents arise from geometrically distinct sources:
the Sobolev exponent $2^{*}=2n/(n-2)$ is dictated by the interior
scalar-curvature equation, while the trace-Sobolev exponent
$2^{\sharp}=2(n-1)/(n-2)$ is dictated by the boundary
mean-curvature equation.

			Han--Li~\cite{HanLi2000} established existence
			for~\eqref{eq:yamabe} on locally conformally flat manifolds with
			umbilic boundary and on manifolds of positive type with at least one
			non-umbilic boundary point. In their proof, the non-umbilic condition
			is the geometric mechanism that forces the mountain-pass level
			strictly below the threshold~$S_c$.

Neumann problems are of intrinsic interest in mathematical biology,
arising as steady states of the Keller--Segel chemotactic aggregation
model~\cite{KellerSegal,LNT} and as the shadow system of various
activator--inhibitor reaction--diffusion
models~\cite{GiererMeinhardt,lN}.
The corresponding Neumann problem is
\begin{equation}
	\label{eq:neumann_proto}
	\begin{cases}
		-\Delta u+\lambda u=u^{p}  & \text{in }\Omega,\\
		u>0                         & \text{in }\Omega,\\
		\frac{\partial u}{\partial\nu}=0          & \text{on }\partial\Omega,
	\end{cases}
\end{equation}
with $1<p\le2^{*}-1$ and $\lambda>0$.

In the subcritical case ($p<2^{*}-1$), Lin--Ni~\cite{lN} and
Lin--Ni--Takagi~\cite{LNT} established existence of a nonconstant
solution for $\lambda$ large and nonexistence for $\lambda$ small.
Ni--Takagi~\cite{NT1989,NT1991} then proved that, for $\lambda$ large,
the unique maximum of every least-energy solution lies on $\partial\Omega$
and concentrates toward a boundary point as $\lambda\to+\infty$.

In the critical case ($p=2^{*}-1$), Wang~\cite{Wang1991} and
Adimurthi--Mancini~\cite{AM} proved existence for general smooth domains
and large $\lambda$.
The decisive geometric insight was provided by Adimurthi--Pacella--Yadava~\cite{APY}: for $n\ge 7$, every least energy solution $u_{\lambda}$
of~\eqref{eq:neumann_proto} with $p=2^{*}-1$ concentrates as
$\lambda\to+\infty$ at a boundary point where the mean curvature $H$
attains its maximum $H_{M}:=\max_{\partial\Omega}H$;
as a consequence, whenever $\partial\Omega$ possesses $k$ strict local maxima
of $H$ at which $H$ is positive, problem~\eqref{eq:neumann_proto} admits
at least $k$ distinct solutions for $\lambda$ large~\cite{APY}.

In short, for~\eqref{eq:neumann_proto} in the critical case, the mean
curvature of $\partial\Omega$ governs both the location of concentration
and the solvability threshold.

	The step most directly relevant to the present work was taken by
Deng--Wang--Wu~\cite{DWW}, who introduced the doubly critical
Neumann problem
\begin{equation}\label{eq:DWW}
	\begin{cases}
		-\Delta u = u^{2^{*}-1} + f(x,u) & \text{in } \Omega, \\[4pt]
		\dfrac{\partial u}{\partial \nu} = u^{2^{\sharp}-1}
		& \text{on } \partial\Omega,
	\end{cases}
\end{equation}
resolving an open problem of Br\'{e}zis on the Neumann analogue
of the Br\'{e}zis--Nirenberg problem~\cite{BN}.
Their key analytic tool is Escobar's sharp doubly critical
Sobolev inequality~\cite{Escobar1988}: for every
$u\in D^{1,2}(\mathbb{R}^{n}_{+})$ and nonnegative
	constants $a,b$ with $a+b>0$,
\begin{equation}\label{eq:escobar_ineq}
		a\,\|u\|_{L^{2^{*}}(\mathbb{R}^{n}_{+})}
		+b\,\|u\|_{L^{2^{\sharp}}(\partial\mathbb{R}^{n}_{+})}
		\le S_{a,b}\,\|\nabla u\|_{L^{2}(\mathbb{R}^{n}_{+})},
\end{equation}
where the sharp constant $S_{a,b}$ depends on
	$n$, $a$, and $b$, and is achieved by
	\[
	\psi(x)
	=\bigl(1+|x'|^{2}+|x_{n}+x_{n}^{0}|^{2}\bigr)^{-(n-2)/2}
	\]
	for some constant $x_{n}^{0}=x_{n}^{0}(a,b,n)>0$.
Using~\eqref{eq:escobar_ineq} and a mountain-pass
argument below level $A$, Deng--Wang--Wu established the existence of
a positive least-energy solution of~\eqref{eq:DWW} whenever
$f(x,u)\ge-C(u+u^{s})$ for some $C\ge 0$ and $s\in(1,2^{\sharp}-1)$; the
case $f(x,u)=-\lambda u$ with any $\lambda>0$ is also covered.
Pierotti--Terracini~\cite{PT1995,PT} subsequently revealed how the
topology and geometry of $\Omega$ generate additional multiplicity, and
a complete half-space analysis was carried out by Deng--Shi~\cite{DS}.

The closest analogue of the present work is the result of Costa--Gir\~{a}o~\cite{CG}, who studied, for $n\ge 5$, the
homogeneous Neumann problem
\begin{equation}
	\label{eq:costa}
	\begin{cases}
		-\Delta u+\lambda u=u^{2^{*}-1}-\alpha u^{2^{\sharp}-1}
		& \text{in } \Omega,\\[4pt]
		\dfrac{\partial u}{\partial\nu}=0
		& \text{on } \partial\Omega.
	\end{cases}
\end{equation}
Costa--Gir\~{a}o established that there exists a threshold
$\alpha_{0}^{\mathrm{CG}}=\alpha_{0}^{\mathrm{CG}}(\lambda,\Omega)>0$
such that problem~\eqref{eq:costa} admits a least-energy solution
when $\alpha<\alpha_{0}^{\mathrm{CG}}$, and admits no least-energy
solution when $\alpha>\alpha_{0}^{\mathrm{CG}}$.

The present paper addresses the  problem at the intersection
of Deng--Wang--Wu~\cite{DWW} and Costa--Gir\~{a}o~\cite{CG}:
the case in which both a critical Neumann boundary
$\partial_{\nu}u=u^{2^{\sharp}-1}$ and a critical interior
absorption $-\alpha u^{2^{\sharp}-1}$ are simultaneously present.

		The energy functional associated with~\eqref{eq:main} is
	\begin{equation}\label{eq1.3}
		I_{\alpha}(u)
		= \int_{\Omega}\!\Bigl(\frac{1}{2}|\nabla u|^{2}
		+\frac{\lambda}{2}|u|^{2}
		+\frac{\alpha}{2^{\sharp}}|u|^{2^{\sharp}}
		-\frac{1}{2^{*}}|u|^{2^{*}}\Bigr)\,dx
		-\frac{1}{2^{\sharp}}\int_{\partial\Omega}|u|^{2^{\sharp}}\,d\sigma,
	\end{equation}
	where $d\sigma$ is the $(n-1)$-dimensional Hausdorff measure on
	$\partial\Omega$.  It is standard that
	\(I_\alpha\in C^2(H^1(\Omega),\mathbb R)\), and that its critical points
	are precisely the weak solutions of \eqref{eq:main}, namely, $u\in H^{1}(\Omega)$ satisfying
	\begin{equation}\label{eq1.2}
		\int_{\Omega}\!\bigl(\nabla u\cdot\nabla v
		+\lambda uv
		+\alpha|u|^{2^{\sharp}-2}uv
		-|u|^{2^{*}-2}uv\bigr)\,dx
		-\int_{\partial\Omega}|u|^{2^{\sharp}-2}uv\,d\sigma=0
	\end{equation}
for every \(v\in H^1(\Omega)\).
	
	The Nehari manifold associated with $I_{\alpha}$ is defined by
	\[
	\mathcal N_{\alpha}
	:=\{u\in H^1(\Omega)\setminus\{0\}:
	\langle I'_{\alpha}(u),u\rangle=0\}.
	\]
	To characterize least-energy solutions variationally, we use the
	fibering map associated with $I_{\alpha}$. A key feature of the present
	problem is that the absorption term has exactly the critical trace exponent
	$2^\sharp$. Consequently, after a suitable change of variables, the
	corresponding fibering equation reduces to a quadratic equation.
	
	More precisely, for every $u\neq0$, there exists a unique
	$t_{\alpha}(u)>0$ such that
	$t_{\alpha}(u)u\in\mathcal N_{\alpha}$, where
	\[
	t_{\alpha}(u)
	=
	\left(
	\frac{-b(u)+\sqrt{b(u)^2+4a(u)c(u)}}{2c(u)}
	\right)^{\frac{n-2}{2}},
	\]
	with
	\[
	a(u):=\|u\|^2,\qquad
	b(u):=\|u\|_{2^\sharp,\partial\Omega}^{2^\sharp}
	-\alpha\|u\|_{2^\sharp,\Omega}^{2^\sharp},
	\qquad
	c(u):=\|u\|_{2^*,\Omega}^{2^*}.
	\]
	
	Therefore, the least-energy level can be characterized as
	\[
	S_{\alpha}
	:=
	\inf_{u\in\mathcal N_{\alpha}}I_{\alpha}(u)
	=
	\inf_{u\neq0}\sup_{t>0}I_{\alpha}(tu).
	\]
	A function $u\in\mathcal N_{\alpha}$ satisfying
	$I_{\alpha}(u)=S_{\alpha}$ is called a \emph{least energy solution}  of
	\eqref{eq:main}.
	  Since $S_{0}<A$ (see~\cite{DWW}) and
		$S_{\alpha}$ is strictly increasing and continuous in $\alpha$
		by Lemma~\ref{lem:monotone_continuous} and
		Corollary~\ref{cor:strict_increase}, we introduce the threshold
		\begin{equation}\label{eq:alpha0_def}
			\alpha_{0}:=\inf\{\alpha\ge0:S_{\alpha}=A\},
		\end{equation}
		where $A$ is given by~\eqref{eq:A_value}. By construction,
		$\alpha_{0}>0$.
		Our main result is the following.
			
		\begin{Theorem}\label{thm:main}
			Let $\Omega$ be a smooth bounded domain in $\mathbb{R}^n$  with $n \ge 5$ and  $H$ be the mean
curvature of $\partial\Omega$. Then there exists a positive real number $\alpha_0 = \alpha_0(\lambda, \Omega)\ge C(n) \underset{\partial \Omega}{\max} H$ such that
			\begin{enumerate}
				\item[\textup{(i)}]
				problem~\eqref{eq:main} admits a least energy solution  if~ $0\le\alpha<\alpha_{0}$;
				\item[\textup{(ii)}]
				problem~\eqref{eq:main} admits no least energy solution if $\alpha>\alpha_{0}$;
                \item[\textup{(iii)}]  for $\alpha=\alpha_{0}$,
				problem~\eqref{eq:main} also admits a least energy solution if $ \alpha_{0}>   C(n)  \underset{\partial \Omega}{\max} H, $	
          where
          \[
C(n)=
\frac{2^{(5-n)/2} \,(n-2)^{(n-2)/2}  }
{\sqrt{n}\,(n-3)~(n-1)^{(n-5)/2}\left[
B\left(\frac12,\frac{n-2}{2}\right)
-
B_{\frac{n}{2(n-1)}}\left(\frac12,\frac{n-2}{2}\right)
\right]},
\]
depending only  on $n$, and $B$ is the standard Beta function, while $B_{\frac{n}{2(n-1)}}$ is an incomplete Beta function.  
\end{enumerate}

		\end{Theorem}
\begin{Remark}
When $\alpha_{0}=C(n)\underset{\partial \Omega}{\max} H $, either
		there exists a least energy solution of~\eqref{eq:main}
		with $\alpha=\alpha_0$,
	or	the sequence of the least energy solutions $\{u_{\alpha_k}\}$ for $\alpha _k \rightarrow \alpha^- _0 $ as $ k \rightarrow \infty$ has a subsequence
			$\{u_{k_j}\}$ with
		$u_{k_j}\rightharpoonup 0$,\,
		$\|u_{k_j}\|_{L^{\infty}(\Omega)}\to\infty,$
		and every limit point of the maximum points
		of $u_{k_j}$
		lies in
		$\{\xi\in\partial\Omega:
		H(\xi)=\max\limits_{\partial\Omega}H\}.$
In the special case where
$\Omega=B_R(0)$, the boundary mean curvature $H=\frac{1}{R}$, and the lower bound of $\alpha _0$ becomes $\frac{C(n)}{R}$. Observe that  the lower bound of $\alpha_0$ is inversely proportional to the radius of the sphere $B_R(0)$ .
\end{Remark}
		
Compared with previous works, the simultaneous presence of the
critical Sobolev exponent $2^\ast$ and the critical trace exponent $2^\sharp$,
together with the absorption term $-\alpha u^{2^\sharp-1}$ that carries the
same exponent as the boundary nonlinearity, gives rise to several new
difficulties. One of the difficulties is the presence of two distinct critical terms: the interior Sobolev exponent $2^\ast$ and the boundary trace exponent $2^\sharp$.  Both embeddings $H^1(\Omega)\hookrightarrow L^{2^\ast}(\Omega)$ and $H^1(\Omega)\hookrightarrow L^{2^\sharp}(\partial\Omega)$ fail to be compact.
Consequently, minimizing sequences for $S_\alpha$ may lose compactness
through interior or boundary concentration, and the least energy level
$S_\alpha$ may fail to be attained. The relevant compactness threshold is therefore determined by the
doubly critical problem in the half-space. In particular, the strict
inequality $S_\alpha<A$ rules out concentration at the critical level
and restores compactness.
 Another major difficulty stems from the fact that the absorption term
 has exactly the critical trace exponent $2^\sharp$, and hence the same
 scaling order as the boundary critical nonlinearity. In the
lower-order perturbation case studied by Deng--Wang--Wu \cite{DWW}, the
perturbation contributes only at order
$O(\varepsilon^2)$ to the bubble energy. Hence it is    dominated by the boundary
mean curvature correction $-L(n)H(\xi)\varepsilon$.
In that case, the curvature term dominates the perturbation, and the
bubble energy is automatically lowered below the half-space level. Here, by contrast,
the critical exponent $2^\sharp$ makes the absorption contribution
$B(n)\alpha\varepsilon$ enter the energy expansion at the same order
$\varepsilon$ as the curvature correction, and with the opposite sign (see
Lemma~\ref{lem:expansion}).  Consequently, the least energy level
$S_\alpha$ is no longer automatically below the critical level $A$. The  existence theory now reduces to determining whether the strict inequality
$S_\alpha<A$ still holds. This competition gives rise naturally to the threshold parameter
$\alpha_0$, which separates the regime where $S_\alpha<A$ from the
critical regime where $S_\alpha=A$.

To overcome these difficulties, we combine a blow-up analysis, the explicit Li--Zhu \cite{LZ} classification of half-space bubbles, coercivity estimates for the linearized operator, refined energy expansions, monotonicity
properties on the Nehari manifold, and concentration--compactness arguments. After rescaling at the concentration scale, the absorption term becomes
asymptotically negligible, since $\alpha_k\delta_k\to0$. Therefore the limiting profile is governed by the absorption‑free half‑space problem. The refined energy expansion quantifies the competition and determines the sharp threshold $\alpha_0$.  By the monotonicity and continuity of $S_\alpha$, one has
$S_\alpha<A$ for $\alpha<\alpha_0$, which restores compactness and yields
a least energy solution.   For $\alpha>\alpha_0$, we have $S_\alpha=A$, and no least energy solution exists.

		The threshold is quantitatively controlled by the geometry of $\partial\Omega$.  Recall the  ground state of limit equation in the half-space
\begin{equation}\label{eq:V1_extension}
			V_{1}(y)=C_{n}\bigl(1+|y'|^{2}+(y_{n}+x_{n}^{0})^{2}\bigr)^{-(n-2)/2}.
		\end{equation}
        Observe that $V_1$ lacks translational symmetry in the $y_n$-direction and
attains its unique maximum at a fixed normal distance $x_n^0$ from the
boundary. Let $\{u_k\}$ be a concentrating sequence of least energy
solutions. After rescaling, $u_k$ is expected to converge to the profile
$V_1$. To approximate $u_k$ near its concentration point, we adapt $V_1$
to the curved boundary. To this end, let
        \begin{equation*}
            V_{\varepsilon,\xi}(x)=\varepsilon^{-(n-2)/2}\,V_{1}\left(\frac{T_{\xi}^{-1}(x)}{\varepsilon}\right)
        \end{equation*}
        denote
			the boundary bubble centred at $\xi\in\partial\Omega$ at scale
			$\varepsilon>0$ (see Definition~\ref{def:bubble}).
		The energy expansion of the bubble functional
		(Lemma~\ref{lem:expansion}) gives, for any $\xi\in\partial\Omega$ with
		$H(\xi)>0$,
		\begin{equation}\label{eq:expansion_intro}
			I_{\alpha}(V_{\varepsilon,\xi})
			=A-L(n)\,H(\xi)\,\varepsilon+B(n)\,\alpha\,\varepsilon+O(\varepsilon^{2}),
			\quad\text{as }\varepsilon\to 0,
		\end{equation}
		where $$H(\xi) = \dfrac{1}{n-1}\sum_{i=1}^{n-1}h_i(\xi)$$ is the mean
curvature of $\partial\Omega$ at $\xi$, with $h_1(\xi),\ldots,h_{n-1}(\xi)$
the principal curvatures of $\partial\Omega$ at $\xi$, 
 \begin{equation*}\label{eq:LB}
	\begin{aligned}
		L(n)&=\frac{n^{n/2}\,(n-2)^{(2n-1)/2}\,\pi^{n/2}}
		{(n-3)\,2^{(3n-3)/2}\,(n-1)^{(n-3)/2}\,\Gamma(\frac{n}{2}+1)},\\[6pt]
		B(n)
		&=
		\frac{(n(n-2))^{(n+1)/2}\pi^{n/2}}
		{2^{\,n+1}(n-1)\Gamma\!\left(\frac n2+1\right)}
		\left[
		B\!\left(\frac12,\frac{n-2}{2}\right)
		-
		B_{\frac{n}{2(n-1)}}\!\left(\frac12,\frac{n-2}{2}\right)
		\right].
	\end{aligned}
\end{equation*}
 In particular, $L(n)>0$ and $B(n)>0$ for all $n\geq 5$. The constant
 \[
 C(n):=\frac{L(n)}{B(n)}
 \]
 will be used throughout the paper. The energy falls below $A$ when
$\alpha<\frac{L(n)}{B(n)} H(\xi)=C(n)H(\xi)$, optimising over
$\xi\in\partial\Omega$ yields the geometric lower bound
\begin{equation}\label{eq:alpha0_lower}
	\alpha_{0}\;\ge\;C(n)\,\max_{\partial\Omega}H.
\end{equation} 
        This estimate also admits a natural geometric interpretation. When $\Omega=B_R(0)$, the boundary mean curvature is constant, i.e. $H\equiv1/R$, and the geometric lower bound reduces to 
        $$ \alpha_0\ge\frac{C(n)}{R}. $$ 
        Hence the lower bounded of threshold $\alpha_{0}$ grows as the ball's radius $R$ shrinks.
         The borderline case $$\alpha_0=C(n)\max_{\partial\Omega}H $$ is analyzed in detail in Section~\ref{sec:section5}.

	The paper is organized as follows.
	Section~\ref{sec:Preliminaries} collects preliminary facts, including bubble profiles at boundary points and some key lemmas. In Section~\ref{subsec:nehari}, we
	introduce the threshold $\alpha_0$,
	and prove the existence of a least energy solution for every
	$\alpha<\alpha_0$.
	    Section~\ref{sec:blowup} shows that any sequence of least energy solutions
	$\{u_k\}$ at parameters $\alpha_k\to+\infty$ must blow up and
	concentrate near a boundary point as a rescaled bubble, while
	Section~\ref{sec:section4} establishes the finiteness of the threshold
	parameter $\alpha_0$ through a precise energy expansion based on the
	bubble profiles and the mean curvature of $\partial\Omega$, thereby
	completing the proof of Theorem~\ref{thm:main} (i)--(ii).
	Section~\ref{sec:section5} addresses the case
	$\alpha=\alpha_0$: we establish a geometric lower bound for $\alpha_0$, provide a proof of Theorem~
    \ref{thm:main} (iii) and   show that a
	least energy solution either exists at the threshold or every
	approximating sequence concentrates at a point of maximal boundary
	mean curvature.

	\paragraph*{Notation.}
	For $1\le p\le\infty$, $\|\cdot\|_{p,\Omega}$ and
	$\|\cdot\|_{p,\partial\Omega}$ denote, respectively, the $L^{p}$-norms over
	$\Omega$ and over $\partial\Omega$; $d\sigma$ is the
	$(n-1)$-dimensional Hausdorff measure on $\partial\Omega$.
	We equip $H^{1}(\Omega)$ with the $\lambda$-weighted norm
	\[
	\|u\|^{2}:=\int_{\Omega}\bigl(|\nabla u|^{2}+\lambda u^{2}\bigr)\,dx,
	\]
	which is equivalent to the standard $H^{1}$-norm since $\lambda>0$.
	The homogeneous Sobolev space $D^{1,2}(\R_+^n)$ is endowed with
	$\|u\|_{D^{1,2}}:=\|\nabla u\|_{L^{2}(\R_+^n)}$. The pairing
	$\langle\cdot,\cdot\rangle$ denotes the duality between
	$H^{1}(\Omega)^{*}$ and $H^{1}(\Omega)$.

\section{Preliminaries}\label{sec:Preliminaries}

In this section, we first record the
sharp constants for the  limit problem on half space and two key local
compactness lemmas.  Secondly, we construct the
two-parameter family of bubble profiles localized at boundary points
of~$\Omega$.


Let
\[
\R^{n}_{+}:=\{x=(x',x_{n})\in\R^{n}:\,x_{n}>0\},
\]
and  $\partial\R^{n}_{+}=\{(x',x_{n})\in\R^{n}: x_{n}=0\}$. Let $\nu=(0,\dots,0,-1)$  denote a unit outward normal vector to
$\partial\R^{n}_{+}$.  Before studying the original problem \eqref{eq:main} on the bounded domain
$\Omega$, we first consider the associated doubly critical problem on the
half-space $\mathbb{R}^{n}_{+}$, which plays a fundamental role in the
analysis of the critical level.
 Li--Zhu \cite[Theorem~1.2]{LZ} classified all positive solutions
 $U\in D^{1,2}(\mathbb{R}^{n}_{+})$ of
 \begin{equation}\label{eq:half_space_critical}
 	\begin{cases}
 		-\Delta U=U^{2^{*}-1} & \text{in }\mathbb{R}^{n}_{+},\\
 		-\partial_{x_n}U=U^{2^{\sharp}-1}
 		& \text{on }\partial\mathbb{R}^{n}_{+},
 	\end{cases}
 \end{equation}
 and showed that, up to translations parallel to the boundary and
 dilations, every such solution is given by
 \begin{equation}\label{eq:V1}
 	V_{1}(x)
 	=
 	C_{n}
 	\left(
 	\frac{1}{1+|x'|^{2}+(x_{n}+x_{n}^{0})^{2}}
 	\right)^{(n-2)/2},
 \end{equation}
 where
 \[
 C_{n}=(n(n-2))^{(n-2)/4},
 \qquad
 x_{n}^{0}=\sqrt{\frac{n}{n-2}}.
 \]
This explicit formula allows us to construct a family of bubbles
\begin{equation}\label{eq:Vepsxi}
	V_{\varepsilon,z}(x)
	:=\varepsilon^{-(n-2)/2}V_{1}\!\Bigl(\dfrac{x-z}{\varepsilon}\Bigr),
	\qquad\varepsilon>0,\ z=(z',0)\in\partial\R^{n}_{+}.
\end{equation}

The variational level associated with \eqref{eq:half_space_critical}
is
\begin{equation}\label{eq:A_def}
	A:=\inf_{u\in D^{1,2}(\R^{n}_{+})\setminus\{0\}}\sup_{t>0}\Psi_{\R^{n}_{+}}(tu),
\end{equation}
where, for an open set $D\subseteq\R^n$ with $C^{1}$ boundary,
\begin{equation}\label{eq:Psi_inf}
	\Psi_{D}(u)
	:=\tfrac12\!\int_{D}\!|\nabla u|^{2}\,dx
	-\tfrac{1}{2^{*}}\!\int_{D}\!|u|^{2^{*}}\,dx
	-\tfrac{1}{2^{\sharp}}\!\int_{\partial D}\!|u|^{2^{\sharp}}\,d\sigma.
\end{equation}
In what follows, we focus on $D=\Omega$ and $D=\R_+^n$.

Now, we present two key lemmas.
\begin{Lemma}\cite[Lemma~2.3]{DWW}\label{lem2.1}
	The infimum~\eqref{eq:A_def} is achieved by $V_{1}$. By the invariance
	of $\Psi_{\mathbb{R}^n_+}$ under translations along $\partial\mathbb{R}^n_+$
	and under the rescaling $u(x)\mapsto\varepsilon^{-(n-2)/2}u(x/\varepsilon)$,
	it is also attained at every member of the family
	$\{V_{\varepsilon,z}:\varepsilon>0,\ z\in\partial\mathbb{R}^n_+\}$.
	Moreover,
	\begin{equation}\label{eq:A_value}
		A=\Psi_{\mathbb{R}^{n}_{+}}(V_{1})
		=(\tfrac{1}{2}-\tfrac{1}{2^*})\|V_{1}\|_{2^{*},\mathbb{R}^{n}_{+}}^{2^{*}}
		+(\tfrac{1}{2^\sharp}-\tfrac{1}{2^*})\|V_{1}\|_{2^{\sharp},\partial\mathbb{R}^{n}_{+}}^{2^{\sharp}}>0.
	\end{equation}
\end{Lemma}

\begin{Lemma}\cite[Lemma~2.2]{DSX}\label{lem:local_compactness}
	Let $\{v_{k}\}\subset H^{1}(\Omega)$ satisfy $v_{k}\rightharpoonup 0$
	weakly in $H^{1}(\Omega)$.  Then either $v_{k}\to 0$ strongly in
	$H^{1}(\Omega)$, or
	\begin{equation}\label{eq:local_compactness}
		\liminf_{k\to\infty}\sup_{t>0}\Psi_{\Omega}(t v_{k})\ge A.
	\end{equation}
\end{Lemma}

The family $V_{\varepsilon,z}$ introduced in~\eqref{eq:Vepsxi}
is defined on the model half-space $\mathbb{R}^n_+$
with the fixed normal direction $e_n$.
Following Adimurthi--Yadava~\cite{AY}, we extend this construction to
a boundary point $\xi\in\partial\Omega$.
We use an orthonormal frame adapted to $\partial\Omega$ at $\xi$,
whose last vector is the inward normal direction
$-\nu(\xi)$.
The corresponding bubble is denoted by $V_{\varepsilon,\xi}$.
The construction is independent of the choice of tangential vectors.
When $\xi=z\in\partial\mathbb{R}^n_+$ and the standard frame is used,
$V_{\varepsilon,\xi}$ coincides with $V_{\varepsilon,z}$.

For $x_{0}\in\R^{n}$ and an orthonormal frame
$f(x_{0})=\{e_{1}(x_{0}),\dots,e_{n}(x_{0})\}$ at~$x_{0}$, set
\begin{equation}\label{eq:Pi_def}
\begin{aligned}
	&\Lambda(x_{0}):=\bigl\{x\in\R^{n}:\langle x-x_{0},e_{n}(x_{0})\rangle>0\bigr\},\\
&y_{j}:=\langle x-x_{0},e_{j}(x_{0})\rangle\ \ (1\le j\le n).
\end{aligned}
\end{equation}
Then $\Lambda(x_{0})$ is the open half space on the
$e_{n}(x_{0})$ side of~$x_{0}$, and $y=(y_{1},\dots,y_{n})$ are the
intrinsic affine coordinates determined by $f(x_{0})$, in which
$\Lambda(x_{0})=\{y_{n}>0\}$. For $\varepsilon>0$, define
\begin{equation}\label{equa2.9_general}
V_{\varepsilon,f(x_{0}),x_{0}}(x):=\varepsilon^{-(n-2)/2}\,V_{1}\!\Bigl(\frac{y}{\varepsilon}\Bigr),
\end{equation}
where $V_{1}$ is the half space ground state of~\eqref{eq:half_space_critical},
\begin{equation}\label{eq:V1_extension}
V_{1}(y)=C_{n}\bigl(1+|y'|^{2}+(y_{n}+x_{n}^{0})^{2}\bigr)^{-(n-2)/2},
\qquad y\in\R^{n}.
\end{equation}
By Lemma~\ref{lem2.1} and the change of variables
	$y\mapsto y/\varepsilon$, the function $V_{\varepsilon,f(x_{0}),x_{0}}$
	is positive and $C^{\infty}$ on $\R^{n}$ and satisfies
\begin{equation}\label{eq:bubble_PDE}
	\begin{cases}
		-\Delta V_{\varepsilon,f(x_{0}),x_{0}}
		=V_{\varepsilon,f(x_{0}),x_{0}}^{\,2^{*}-1}
		& \text{in }\Lambda(x_{0}),\\[2pt]
		\dfrac{\partial V_{\varepsilon,f(x_{0}),x_{0}}}{\partial\nu}
		=V_{\varepsilon,f(x_{0}),x_{0}}^{\,2^{\sharp}-1}
		& \text{on }\partial\Lambda(x_{0}),
	\end{cases}
\end{equation}
where $\nu(x_{0})=-e_{n}(x_{0})$ is the outward unit normal to
	$\Lambda(x_{0})$.

We now state an important lemma

\begin{Lemma}\label{lem:frame_indep}
	For every $\varepsilon>0$ and $x_{0}\in\R^{n}$,
		$V_{\varepsilon,f(x_{0}),x_{0}}$ depends only on $e_{n}(x_{0})$
		and not on the tangential vectors $\{e_{1}(x_{0}),\dots,e_{n-1}(x_{0})\}$.
\end{Lemma}
\begin{proof}
	By the explicit formula~\eqref{eq:V1}, the function $V_1$ depends on
	$y=(y',y_n)$ only through $|y'|^2$ and $y_n$.
	Let
	\[
	f(x_0)=\{e_1(x_0),\ldots,e_{n-1}(x_0),e_n(x_0)\}
	\]
	and
	\[
	f'(x_0)=\{e'_1(x_0),\ldots,e'_{n-1}(x_0),e_n(x_0)\}
	\]
	be two orthonormal frames with the same normal vector.
	Then their tangential parts are related by an orthogonal matrix
	$R\in O(n-1)$, namely,
	\[
	e_j=\sum_{k=1}^{n-1}R_{jk}e'_k,
	\qquad 1\le j\le n-1 .
	\]
	The corresponding coordinates satisfy
$
	y_j=\sum_{k=1}^{n-1}R_{jk}y'_k .
$
	Hence,
		\[
	\sum_{j=1}^{n-1}y_{j}^{2}
	=\sum_{k,\ell}(R^{T}R)_{k\ell}\,y_{k}'y_{\ell}'
	=\sum_{k=1}^{n-1}(y_{k}')^{2}.
	\]
	Since the normal coordinate $y_n$ remains unchanged, we obtain
	\[
	V_{\varepsilon,f_{x_0},x_0}(x)
	=
	V_{\varepsilon,f'_{x_0},x_0}(x).
	\]
	Therefore, the definition of $V_{\varepsilon,\xi}$ is independent of
	the choice of the tangential frame.
\end{proof}

	We now specialize the construction to a boundary
point of~$\Omega$.
\begin{Definition}\label{def:bubble}\label{def:Pi}
	Let $\xi\in\partial\Omega$, let $e_{n}(\xi):=-\nu(\xi)$ be
		the inward unit normal to $\partial\Omega$ at~$\xi$, and let
		$\{e_{1}(\xi),\dots,e_{n-1}(\xi)\}$ be any orthonormal basis of
		$T_{\xi}\partial\Omega$, set $f(\xi):=\{e_{1}(\xi),\dots,e_{n}(\xi)\}$.
		For $\varepsilon>0$, the \emph{boundary bubble of
			scale~$\varepsilon$ centred at~$\xi$} is
	\begin{equation}\label{equa2.9}
		V_{\varepsilon,\xi}(x):=V_{\varepsilon,f(\xi),\xi}(x).
	\end{equation}
	By Lemma~\ref{lem:frame_indep}, $V_{\varepsilon,\xi}$ is well
		defined and depends only on $\varepsilon$ and~$\xi$.
\end{Definition}

In the special case $\xi=z=(z',0)\in\partial\mathbb{R}^{n}_{+}$
	with the standard frame $f(z)=\{e_1,\ldots,e_n\}$, the rigid motion
	$T_z$ reduces to the translation $y=x-z$, so
	\begin{equation*}
	 V_{\varepsilon,f(z),z}(x)=\varepsilon^{-(n-2)/2}V_1((x-z)/\varepsilon)
	=V_{\varepsilon,z}(x),
	\end{equation*} hence Definition~\ref{def:bubble} and
	\eqref{eq:Vepsxi} are consistent.

To make explicit the isometric identification between
	$\mathbb{R}^n_+$ and the tangent half-space $\Lambda(\xi)$ that
	underlies the construction, we record the associated rigid transformation.
The rigid transformation associated with $f(\xi)$ is
\begin{equation}\label{eq:Txi_def}
	T_{\xi}(y_{1},\dots,y_{n})
	:=\xi+\sum_{j=1}^{n}y_{j}\,e_{j}(\xi).
\end{equation}
$T_{\xi}$ is an orientation-preserving isometry with $T_{\xi}(0)=\xi$
and $T_{\xi}(\mathbb{R}^n_+)=\Lambda(\xi)$.  Its inverse,
\begin{equation}\label{eq:Txi_inv}
	T_{\xi}^{-1}(x)
	=\bigl(\langle x-\xi,e_{1}(\xi)\rangle,\dots,
	\langle x-\xi,e_{n}(\xi)\rangle\bigr),
\end{equation}
gives  $y=T_{\xi}^{-1}(x)$ in which
$V_{\varepsilon,\xi}(x)=\varepsilon^{-(n-2)/2}V_{1}(y/\varepsilon)$.

\begin{Definition}\label{def:M}
	The  bubble manifold is
	\begin{equation}\label{eq:M_def}
		\mathcal M
		:=\bigl\{C V_{\varepsilon,\xi} : C\in \R,\ \varepsilon>0,\ \xi\in\partial\Omega\,\bigr\}.
	\end{equation}
\end{Definition}


\section{Existence of least energy solutions}
\label{subsec:nehari}
In this section, we develop the variational framework based on the
Nehari manifold. We prove the existence of least energy solutions for
$\alpha\in[0,\alpha_{0})$ and characterize the threshold parameter
$\alpha_{0}$ at which the least energy level reaches the critical level
$A$.

The Nehari manifold associated with $I_{\alpha}$ is
\begin{equation}\label{eq:Nehari_def}
	\mathcal{N}_{\alpha}
	:=\{u\in H^{1}(\Omega)\setminus\{0\}\,:\,
	\langle I'_{\alpha}(u),u\rangle=0\}.
\end{equation}
Define $J_{\alpha}\colon H^{1}(\Omega)\to\mathbb{R}$ by
\begin{equation}\label{eq:nehari-functional}
	J_{\alpha}(u)
	:=\langle I'_{\alpha}(u),u\rangle
	=\|u\|^{2}+\alpha\|u\|_{L^{2^\sharp}(\Omega)}^{2^{\sharp}}
	-\|u\|_{L^{2^*}(\Omega)}^{2^{*}}
	-\|u\|_{L^{2^\sharp}(\partial\Omega)}^{2^{\sharp}},
\end{equation}
so that
$\mathcal{N}_{\alpha}
=\{u\in H^{1}(\Omega)\setminus\{0\}:J_{\alpha}(u)=0\}$.

\begin{Lemma}\label{lem:nehari-fibring}
	Let $\alpha\geq 0$.
	\begin{enumerate}
		\item[\normalfont(i)] For every $u\in H^{1}(\Omega)\setminus\{0\}$,
		there exists a unique $t_{\alpha}(u)>0$ such that
		$t_{\alpha}(u)\,u\in\mathcal{N}_{\alpha}$.
		Setting
		\begin{equation}\label{eq:fibring-coeffs}
			a(u):=\|u\|^{2},\quad
			b(u):=\|u\|_{L^{2^\sharp}(\partial\Omega)}^{2^{\sharp}}
			-\alpha\|u\|_{L^{2^\sharp}(\Omega)}^{2^{\sharp}},\quad
			c(u):=\|u\|_{L^{2^*}(\Omega)}^{2^{*}},
		\end{equation}
		the unique projector is given explicitly by
		\begin{equation}\label{eq:fibring-formula}
			t_{\alpha}(u)
			=\left(\frac{-b(u)+\sqrt{b(u)^{2}+4\,a(u)\,c(u)}}
			{2\,c(u)}\right)^{(n-2)/2}.
		\end{equation}
		Moreover, $t\mapsto I_{\alpha}(tu)$ attains its strict global
		maximum on $(0,+\infty)$ at $t=t_{\alpha}(u)$, so
		\begin{equation}\label{eq:fibring-sup}
			I_{\alpha}\bigl(t_{\alpha}(u)\,u\bigr)
			=\sup_{t>0}I_{\alpha}(tu).
		\end{equation}
		
		\item[\normalfont(ii)] $\mathcal{N}_{\alpha}$ is a $C^{1}$
		submanifold of $H^{1}(\Omega)$ of codimension one.
		The fibring map $u\mapsto t_\alpha(u)\,u$
			is a homeomorphism from the unit sphere of $H^1(\Omega)$
			onto $\mathcal{N}_\alpha$, with inverse $v\mapsto v/\|v\|$.
		
		\item[\normalfont(iii)] If $u\in\mathcal{N}_\alpha$
			is a critical point of $I_\alpha\big|_{\mathcal{N}_\alpha}$,
			then $u$ is a critical point of $I_\alpha$ on $H^1(\Omega)$,
			hence a weak solution of~\eqref{eq:main}.
	\end{enumerate}
\end{Lemma}

\begin{proof}
	(i) Fix $u\in H^{1}(\Omega)\setminus\{0\}$
	and consider $g(t):=I_{\alpha}(tu)$ for $t>0$. A direct
	computation yields
	\[
	g'(t)=t\bigl(a(u)-t^{2^{\sharp}-2}b(u)-t^{2^{*}-2}c(u)\bigr),
	\]
	with $a(u)>0$ and $c(u)>0$.  Using the algebraic identity
		$2^{*}-2=\frac{4}{n-2}=2(2^{\sharp}-2)$, we set $\tau:=t^{2^{\sharp}-2}=t^{2/(n-2)}$. For $t>0$, the equation $g'(t)=0$ can be rewritten as the quadratic equation
		\begin{equation}\label{eq:fibring-quad}
			c(u)\tau^{2}+b(u)\tau-a(u)=0.
		\end{equation}
	Since $a(u),c(u)>0$, equation \eqref{eq:fibring-quad} has a
	unique positive root
	\[
	\tau_\alpha(u)=\frac{-b(u)+\sqrt{b(u)^{2}+4a(u)c(u)}}{2c(u)},
	\]
	which gives \eqref{eq:fibring-formula} with
	$t_{\alpha}(u)=\tau_{\alpha}(u)^{(n-2)/2}$.  Hence $g$ attains its strict global
	maximum on $(0,+\infty)$ at $t_\alpha(u)$.
	
	\smallskip
	(ii)  We show that $J_{\alpha}'(u)\neq 0$ for every
		$u\in\mathcal{N}_{\alpha}$.  Assume, by contradiction, that there exists
		$u\in\mathcal{N}_{\alpha}$ such that
		\[
		\langle J'_{\alpha}(u),u\rangle=0
		\]
	A direct computation gives the algebraic identity
	\begin{equation}\label{eq:Jprime-identity}
		2^{\sharp}J_{\alpha}(u)-\langle J_{\alpha}'(u),u\rangle
		=(2^{\sharp}-2)\|u\|^{2}+(2^{*}-2^{\sharp})\|u\|_{2^{*},\Omega}^{2^{*}},
	\end{equation}
	for every $u\in H^{1}(\Omega)$. Under the assumptions
	\[
	J_{\alpha}(u)=0
	\quad\text{and}\quad
	\langle J_{\alpha}'(u),u\rangle=0,
	\]
	the left-hand side of \eqref{eq:Jprime-identity} is equal to zero.
	Since $2<2^{\sharp}<2^{*}$, both coefficients appearing on the right hand side are strictly positive. Therefore, we obtain
	$u\equiv0$, which contradicts the assumption.
	Hence
	$J'_{\alpha}(u)\neq 0$ on $\mathcal{N}_{\alpha}$. Since $0$ is  a regular value of $J_{\alpha}$ on
	$H^{1}(\Omega)\setminus\{0\}$, the regular value theorem
	ensures that $\mathcal{N}_{\alpha}$ is a $C^{1}$ submanifold
	of codimension one in $H^{1}(\Omega)$.
	
	By \eqref{eq:fibring-formula} and the continuity of the Sobolev
	and trace embeddings, the map
	$u\mapsto t_{\alpha}(u)$ is continuous on
	$H^{1}(\Omega)\setminus\{0\}$. Consequently, the map $u\mapsto t_{\alpha}(u)\,u$ is
	a homeomorphism from the unit sphere of $H^{1}(\Omega)$ onto
	$\mathcal{N}_{\alpha}$, with inverse $v\mapsto v/\|v\|$. 
	
	\smallskip
	(iii) Let $u\in\mathcal{N}_{\alpha}$ be a critical point of the restriction of
	$I_{\alpha}$ to $\mathcal{N}_{\alpha}$. Since $\mathcal{N}_{\alpha}$ is a
	$C^{1}$-submanifold of $H^{1}(\Omega)$ by (ii), the Lagrange multiplier
	theorem yields a constant $\mu\in\mathbb{R}$ such that
	\[
	I_{\alpha}'(u)=\mu J_{\alpha}'(u).
	\]
	Testing this identity
	against $u$ and using $J_{\alpha}(u)=0$, we obtain
	\[
	0=J_{\alpha}(u)=\langle I_{\alpha}'(u),u\rangle=\mu\langle J_{\alpha}'(u),u\rangle.
	\]
	Using \eqref{eq:Jprime-identity} and the identity
	$J_{\alpha}(u)=0$, we obtain
	\[
	\langle J_{\alpha}'(u),u\rangle
	=-(2^{\sharp}-2)\|u\|^{2}
	-(2^{*}-2^{\sharp})\|u\|_{2^{*},\Omega}^{2^{*}}<0,
	\]
	where the strict inequality follows from $u\neq0$ and
	$2<2^{\sharp}<2^{*}$.
	Hence $\mu=0$, and $u$ is a critical point of $I_{\alpha}$.
\end{proof}

By Lemma~\ref{lem:nehari-fibring}(i), for every
$u\in H^{1}(\Omega)\setminus\{0\}$, the point
$t_{\alpha}(u)u$ belongs to $\mathcal{N}_{\alpha}$, where
$t_{\alpha}(u)$ is the unique maximizer of the fibering map
$t\mapsto I_{\alpha}(tu)$ on $(0,\infty)$. Consequently, the least
energy level admits the characterization
\begin{equation}\label{eq:Salpha_def}
	S_{\alpha}
	:=\inf_{u\in\mathcal{N}_{\alpha}}I_{\alpha}(u)
	=\inf_{u\in H^{1}(\Omega)\setminus\{0\}}
	\sup_{t>0}I_{\alpha}(tu).
\end{equation}

\begin{Lemma}\label{lem:S_le_A}
	For all $\alpha\ge 0$, we have $S_{\alpha}\le A$.
\end{Lemma}
\begin{proof}
	Fix $\xi_{0}\in\partial\Omega$ and let
	$\eta\in C_{c}^{\infty}(\mathbb{R}^{n})$ be a smooth cut-off function
	such that $\eta\equiv1$ in a neighbourhood of $\xi_{0}$ and
	$\operatorname{supp}\eta$ is contained in the coordinate neighbourhood
	of $\xi_{0}$. For sufficiently small $\varepsilon>0$, the cut-off
	bubble $\eta V_{\varepsilon,\xi_{0}}$
	belongs to $H^{1}(\Omega)$. By Lemma~\ref{lem:nehari-fibring}(i),
	there exists a unique  $t_{\varepsilon}:=
	t_{\alpha}(\eta V_{\varepsilon,\xi_{0}})>0$
	such that
$t_{\varepsilon}\eta V_{\varepsilon,\xi_{0}}\in\mathcal{N}_{\alpha}.$
	
	By the standard scaling properties of the bubble and the estimates for
	the cut-off function, the following convergence properties hold as
	$\varepsilon\to0^{+}$:
	\begin{align*}
		\|\eta V_{\varepsilon,\xi_{0}}\|^{2}
		&\to
		\|\nabla V_{1}\|_{L^{2}(\mathbb{R}^{n}_{+})}^{2},&
		\|\eta V_{\varepsilon,\xi_{0}}\|_{L^{2^{*}}(\Omega)}^{2^{*}}
		&\to
		\|V_{1}\|_{L^{2^{*}}(\mathbb{R}^{n}_{+})}^{2^{*}},\\
		\|\eta V_{\varepsilon,\xi_{0}}\|_{L^{2^{\sharp}}(\partial\Omega)}^{2^{\sharp}}
		&\to
		\|V_{1}\|_{L^{2^{\sharp}}(\partial\mathbb{R}^{n}_{+})}^{2^{\sharp}},&
		\|\eta V_{\varepsilon,\xi_{0}}\|_{L^{2^{\sharp}}(\Omega)}^{2^{\sharp}}
		&=O(\varepsilon).
	\end{align*}
	By the continuity of the map
	$u\mapsto t_{\alpha}(u)$ and the fibering formula
	\eqref{eq:fibring-formula}, we have
	$t_{\varepsilon}\to1$. Substituting the above asymptotics into
	$I_{\alpha}$ yields
	\[
	I_{\alpha}(t_{\varepsilon}\eta V_{\varepsilon,\xi_{0}})
	=
	\Psi_{\mathbb{R}^{n}_{+}}(V_{1})
	+\frac{\alpha}{2^{\sharp}}
	\int_{\Omega}|\eta V_{\varepsilon,\xi_{0}}|^{2^{\sharp}}\,dx
	+o(1)
	=A+o(1).
	\]
	Therefore, letting $\varepsilon\to0^{+}$, we obtain	$S_{\alpha}\le A.$
\end{proof}

\begin{Lemma}\label{lem:existence-below-threshold}
	If $S_{\alpha}<A$, then $S_{\alpha}$ is achieved by some
	$u_{\alpha}\in\mathcal N_{\alpha}$, i.e., \eqref{eq:main} admits a
	least energy solution.
\end{Lemma}

\begin{proof}
	Let $u\in\mathcal{N}_{\alpha}$. By the definition of the Nehari
	manifold, the following  identity holds:
	\[
	\|u\|^{2}+\alpha\|u\|_{2^{\sharp},\Omega}^{2^{\sharp}}
	=\|u\|_{2^{*},\Omega}^{2^{*}}
	+\|u\|_{2^{\sharp},\partial\Omega}^{2^{\sharp}}.
	\]
		Applying the Sobolev and trace embeddings, we obtain
		\[
		\|u\|^{2}\leq\|u\|_{2^{*},\Omega}^{2^{*}}+\|u\|_{2^{\sharp},\partial\Omega}^{2^{\sharp}}
		\leq C\bigl(\|u\|^{2^{*}}+\|u\|^{2^{\sharp}}\bigr).
		\]
		Since $2^{*},2^{\sharp}>2$, we conclude that
		$\|u\|\geq\delta$ for some $\delta=\delta(\Omega,\lambda)>0$
		independent of $\alpha$.
		Moreover,
		\begin{equation}\label{eq:Nehari_energy}
			\begin{aligned}
				I_{\alpha}(u)
				&=I_{\alpha}(u)-\frac{1}{2^{\sharp}}
				\langle I'_{\alpha}(u),u\rangle \\
				&=\frac{1}{2(n-1)}\|u\|^{2}
				+\frac{n-2}{2n(n-1)}
				\|u\|_{2^{*},\Omega}^{2^{*}},
			\end{aligned}
		\end{equation}
		and therefore
		\begin{equation*}
			I_\alpha(u)\geq\frac{1}{2(n-1)}\delta^{2}>0
			\quad\text{for every }u\in\mathcal{N}_\alpha.
		\end{equation*}
		Thus, we obtain $S_{\alpha}>0$.
		
Let $\{u_k\}\subset\mathcal{N}_{\alpha}$ be a minimizing
sequence for $S_{\alpha}$. Since $\mathcal{N}_{\alpha}$ is a
$C^{1}$ submanifold of $H^{1}(\Omega)$
and $I_{\alpha}|_{\mathcal{N}_{\alpha}}\ge 0$, Ekeland's variational
principle yields, after passing to a subsequence, still denoted by $\{u_k\}$, satisfying
	\begin{equation}\label{eq:PS}
		I_\alpha(u_k)\to S_\alpha,\qquad
		I'_\alpha(u_k)\to 0\quad\text{in }H^{1}(\Omega)^{*}.
	\end{equation}
Applying \eqref{eq:Nehari_energy} to $u_k$, we have
\[
\tfrac{1}{2(n-1)}\|u_k\|^{2}\le I_\alpha(u_k)=S_\alpha+o(1).
\]
It follows that $\{u_k\}$ is bounded in $H^{1}(\Omega)$. Passing to a subsequence,
	\begin{equation}\label{eq:weak-conv}
		\begin{aligned}
			&u_k\rightharpoonup u\ \ \text{weakly in }H^{1}(\Omega),\\
		&u_k\to u\ \ \text{strongly in }L^{s}(\Omega)\ \ \text{for }1\le s<2^{*},\\
		&u_k\to u\ \ \text{a.e.\ in }\Omega.
		\end{aligned}
	\end{equation}
	
	Set $v_k:=u_k-u$, so $v_k\rightharpoonup 0$ in $H^{1}(\Omega)$. The
	Brezis--Lieb lemma  yields
	\begin{align}
		\|u_k\|^{2}&=\|u\|^{2}+\|v_k\|^{2}+o(1),\label{eq:BL-grad}\\
		\|u_k\|_{2^{*},\Omega}^{2^{*}}&=\|u\|_{2^{*},\Omega}^{2^{*}}
		+\|v_k\|_{2^{*},\Omega}^{2^{*}}+o(1),\label{eq:BL-2star}\\
		\|u_k\|_{2^{\sharp},\partial\Omega}^{2^{\sharp}}
		&=\|u\|_{2^{\sharp},\partial\Omega}^{2^{\sharp}}
		+\|v_k\|_{2^{\sharp},\partial\Omega}^{2^{\sharp}}+o(1).\label{eq:BL-bdy}
	\end{align}
	Since $2<2^{\sharp}<2^{*}$, the embedding
	$H^{1}(\Omega)\hookrightarrow L^{2}(\Omega)$ and $H^{1}(\Omega)\hookrightarrow L^{2^{\sharp}}(\Omega)$ is compact, we have
	\begin{equation}\label{eq:vk-interior}
		\|v_k\|_{2^{\sharp},\Omega}^{2^{\sharp}}=	\|v_k\|_{2,\Omega}^{2}=o(1).
	\end{equation}
	Combining \eqref{eq:BL-grad}--\eqref{eq:vk-interior}, we obtain
	\begin{equation}\label{eq:I-split}
		I_\alpha(u_k)=I_\alpha(u)+\Psi_\Omega(v_k)+o(1).
	\end{equation}
	By \eqref{eq:PS},  we obtain $\langle I'_\alpha(u_k),u\rangle\to 0$ as $k\to\infty$.
	Combining this with 
	$u_k\rightharpoonup u$ in $H^1(\Omega)$, we obtain
	\[
	\langle I'_\alpha(u),u\rangle=0.
	\]
	Since $I_\alpha(u)\geq0$, using \eqref{eq:I-split} and the assumption
	$S_\alpha<A$, we have
	\begin{equation}\label{eq:Psi-bound}
		\Psi_\Omega(v_k)\leq S_\alpha+o(1)<A
		\quad\text{for sufficiently large }k.
	\end{equation}
	Combining \eqref{eq:BL-grad}--\eqref{eq:vk-interior} with the identity
	$\langle I'_\alpha(u_k),u_k\rangle=0$ and using
	$\langle I'_\alpha(u),u\rangle=0$, we obtain
	\begin{equation}\label{eq:Psi-prime}
		\langle\Psi'_\Omega(v_k),v_k\rangle
		=\|\nabla v_k\|_{2,\Omega}^{2}
		-\|v_k\|_{2^{*},\Omega}^{2^{*}}
		-\|v_k\|_{2^{\sharp},\partial\Omega}^{2^{\sharp}}
		=o(1).
	\end{equation}
	
Assume, by contradiction, that
$v_k\not\to0$ in $H^{1}(\Omega)$. After passing to a subsequence, we have
$
\|\nabla v_k\|_{2,\Omega}^{2}\ge c>0.
$
Consider the fibering map
$g_k(t):=\Psi_\Omega(tv_k),$ $t>0.$
Its derivative is given by
\[
g'_k(t)
=t\Bigl(
\|\nabla v_k\|_{2,\Omega}^{2}
-t^{2^{*}-2}\|v_k\|_{2^{*},\Omega}^{2^{*}}
-t^{2^{\sharp}-2}\|v_k\|_{2^{\sharp},\partial\Omega}^{2^{\sharp}}
\Bigr).
\]
Since $2^{*},2^{\sharp}>2$, the fibering map $g_k$ admits a unique
critical point $t_k>0$, which is characterized by
\begin{equation}\label{eq:tk-eq}
	\|\nabla v_k\|_{2,\Omega}^{2}
	=t_k^{2^{*}-2}\|v_k\|_{2^{*},\Omega}^{2^{*}}
	+t_k^{2^{\sharp}-2}\|v_k\|_{2^{\sharp},\partial\Omega}^{2^{\sharp}}.
\end{equation}
Combining \eqref{eq:tk-eq} with \eqref{eq:Psi-prime}, we obtain
$t_k\to1.$
Consequently,
\[
\liminf_{k\to\infty}\sup_{t>0}\Psi_\Omega(tv_k)
=
\liminf_{k\to\infty}\Psi_\Omega(t_kv_k)
=
\liminf_{k\to\infty}\Psi_\Omega(v_k)
\leq S_\alpha<A,
\]
which contradicts Lemma~\ref{lem:local_compactness}. Therefore, $v_k\to0$ in $H^{1}(\Omega)$.
Hence $u_k\to u$ strongly in $H^{1}(\Omega)$ and
$\|u\|=\lim_{k\to\infty}\|u_k\|\geq\delta>0,$
so that $u\not\equiv0$. Together with
$\langle I'_\alpha(u),u\rangle=0$, this implies
$u\in\mathcal{N}_\alpha$. By the continuity of $I_\alpha$, we obtain
$I_\alpha(u)=S_\alpha.$
Therefore, $u_\alpha:=u$ is a minimizer.
	
\end{proof}

\begin{Lemma}\label{lem:monotone_continuous}
	The map $\alpha\mapsto S_{\alpha}$ is non-decreasing and continuous on
	$[0,+\infty)$.
\end{Lemma}

\begin{proof}
	For $0\leq\alpha_{1}<\alpha_{2}$ and any
	$u\in H^{1}(\Omega)\setminus\{0\}$, we have
	\[
	I_{\alpha_{2}}(u)
	=
	I_{\alpha_{1}}(u)
	+\frac{\alpha_{2}-\alpha_{1}}{2^{\sharp}}
	\int_{\Omega}|u|^{2^{\sharp}}\,dx
	>
	I_{\alpha_{1}}(u).
	\]
	Taking the supremum with respect to $t>0$ and applying
	Lemma~\ref{lem:nehari-fibring}(i), we obtain
	\[
	\sup_{t>0}I_{\alpha_{1}}(tu)
	\leq
	\sup_{t>0}I_{\alpha_{2}}(tu)
	\]
	for every $u\neq0$. Taking the infimum over
	$u\in H^{1}(\Omega)\setminus\{0\}$ yields
	\[
	S_{\alpha_{1}}\leq S_{\alpha_{2}}.
	\]
	
Fix $\tilde\alpha\geq0$ and assume first that
$\alpha_j\searrow\tilde\alpha$. If
$S_{\tilde\alpha}=A$, then monotonicity gives
\[
S_{\tilde\alpha}\leq S_{\alpha_j}\leq A=S_{\tilde\alpha},
\]
hence
$
S_{\alpha_j}=S_{\tilde\alpha}
$
for all $j$, and the conclusion is immediate.
	If $S_{\tilde\alpha}<A$, Lemma~\ref{lem:existence-below-threshold}
	yields a minimizer $u_{\tilde\alpha}\in\mathcal{N}_{\tilde\alpha}$.
	For each $j$,
	by Lemma~\ref{lem:nehari-fibring}\,(i), there is a unique
	$t_{j}:=t_{\alpha_{j}}(u_{\tilde\alpha})>0$ with
	$t_{j}u_{\tilde\alpha}\in\mathcal{N}_{\alpha_{j}}$, and the explicit
	formula \eqref{eq:fibring-formula} gives
	\[
	t_{j}=\left(\frac{-b_{j}+\sqrt{b_{j}^{2}+4\,a\,c}}{2c}\right)^{(n-2)/2},
	\qquad b_{j}:=\|u_{\tilde\alpha}\|_{2^{\sharp},\partial\Omega}^{2^{\sharp}}
	-\alpha_{j}\|u_{\tilde\alpha}\|_{2^{\sharp},\Omega}^{2^{\sharp}},
	\]
	with $a=\|u_{\tilde\alpha}\|^{2}$ and
	$c=\|u_{\tilde\alpha}\|_{2^{*},\Omega}^{2^{*}}$ both fixed and positive. Since $\alpha_{j} \to \tilde\alpha$,
	\[
	b_{j} \to b_{\infty}
	:= \|u_{\tilde\alpha}\|_{2^{\sharp},\partial\Omega}^{2^{\sharp}}
	- \tilde\alpha\,\|u_{\tilde\alpha}\|_{2^{\sharp},\Omega}^{2^{\sharp}},
	\]
	and the continuity of $t_j$ gives
	$t_{j} \to t_{\tilde\alpha}(u_{\tilde\alpha}) = 1$, the latter equality holds
	because $u_{\tilde\alpha} \in \mathcal{N}_{\tilde\alpha}$.
	Hence
	\[
	S_{\tilde\alpha}\le\lim_{j\to\infty}S_{\alpha_{j}}
	\le\lim_{j\to\infty}I_{\alpha_{j}}(t_{j}u_{\tilde\alpha})
	=I_{\tilde\alpha}(u_{\tilde\alpha})=S_{\tilde\alpha}.
	\]
	We conclude that $\underset{j\to\infty}{\lim} S_{\alpha_{j}}=S_{\tilde\alpha}$.
	
	Assume now that $\alpha_j\nearrow\tilde\alpha$. By the monotonicity of
	$S_\alpha$, the limit
	$L:=\displaystyle\lim_{j\to\infty}S_{\alpha_j}$ exists and satisfies $L\le S_{\tilde\alpha}$.
	We only need to show that $L=S_{\tilde\alpha}$.
	
If $L<A$, then $S_{\alpha_{j}}<A$ for all sufficiently large $j$. Lemma~\ref{lem:existence-below-threshold} therefore yields $u_{j}\in\mathcal N_{\alpha_{j}}$ with $I_{\alpha_{j}}(u_{j})=S_{\alpha_{j}}\to L$. As in the proof of Lemma~\ref{lem:existence-below-threshold}, $\{u_{j}\}$ is bounded in $H^{1}(\Omega)$,  passing to a subsequence, $u_{j}\rightharpoonup u_{*}$ weakly in $H^{1}(\Omega)$. By Lemma~\ref{lem:nehari-fibring}\,(i), for each $j$ there is a unique $s_{j}:=t_{\tilde\alpha}(u_{j})>0$ with $s_{j}u_{j}\in\mathcal{N}_{\tilde\alpha}$.
	Set
	\[
	a_{j}:=\|u_{j}\|^{2},\quad
	c_{j}:=\|u_{j}\|_{2^{*},\Omega}^{2^{*}},\quad
	b_{j}:=\|u_{j}\|_{2^{\sharp},\partial\Omega}^{2^{\sharp}}
	-\alpha_{j}\|u_{j}\|_{2^{\sharp},\Omega}^{2^{\sharp}},\quad
	\]
	and
	\[
	\tilde b_{j}:=\|u_{j}\|_{2^{\sharp},\partial\Omega}^{2^{\sharp}}
	-\tilde\alpha\|u_{j}\|_{2^{\sharp},\Omega}^{2^{\sharp}}.
	\]
	Applying the fibring formula~\eqref{eq:fibring-formula} to $u_j$ with
	$\alpha=\tilde\alpha$, we obtain
	\begin{equation}\label{eq:sj}
		s_{j}=\Bigg(\frac{-\tilde b_{j}+\sqrt{\tilde b_{j}^{2}+4a_{j}c_{j}}}{2c_{j}}\Bigg)^{(n-2)/2}.
	\end{equation}
	Applying the same formula at parameter $\alpha_{j}$ and using
	$u_{j}\in\mathcal{N}_{\alpha_{j}}$, we have
	\begin{equation}\label{eq:tj1}
		1=t_{\alpha_{j}}(u_{j})=\Bigg(\frac{-b_{j}+\sqrt{b_{j}^{2}+4a_{j}c_{j}}}{2c_{j}}\Bigg)^{(n-2)/2}.
	\end{equation}
	Moreover,
	\[
	|\tilde b_{j}-b_{j}|
	=|\tilde\alpha-\alpha_{j}|\,\|u_{j}\|_{2^{\sharp},\Omega}^{2^{\sharp}}
\to 0.
	\]
	Using \eqref{eq:sj} and \eqref{eq:tj1}, we conclude
	$s_{j}\to 1$.
	Then
	\[
	S_{\tilde\alpha}\le I_{\tilde\alpha}(s_{j}u_{j})
	=I_{\alpha_{j}}(u_{j})
	+\tfrac{\tilde\alpha-\alpha_{j}}{2^{\sharp}}\!\int_{\Omega}\!|u_{j}|^{2^{\sharp}}\d x+o(1)
	=S_{\alpha_{j}}+o(1)\to L,
	\]
	since $\|u_{j}\|_{2^{\sharp},\Omega}$ is bounded.  Combining this inequality with
	$L\leq S_{\tilde\alpha}$, we conclude that
$	L=S_{\tilde\alpha}.$
Hence,
	\[
	\lim_{j\to\infty}S_{\alpha_j}=S_{\tilde\alpha},
	\]
	which proves the left continuity of $S_\alpha$ at $\tilde\alpha$.
\end{proof}

\begin{Definition}\label{def:alpha0}
	Define
	\begin{equation}\label{eq:alpha0_def}
		\alpha_{0}:=\inf\{\alpha\ge 0\,:\,S_{\alpha}=A\},
	\end{equation}
	with the convention $\inf\emptyset=+\infty$.
\end{Definition}

\begin{Remark}\label{rmk:alpha0_positive}
	Since $S_{0}<A$ (\cite[Theorem~2.5]{DWW}), we have $\alpha_{0}>0$.
	By Lemma~\ref{lem:monotone_continuous}, $\alpha\mapsto S_{\alpha}$ is
	continuous, so $S_{\alpha_{0}}=A$ whenever $\alpha_{0}<+\infty$, and
	$S_{\alpha}<A$ for all $\alpha\in[0,\alpha_{0})$.
\end{Remark}

\begin{Corollary}\label{cor:strict_increase}
	The map $\alpha\mapsto S_{\alpha}$ is strictly increasing on
	$[0,\alpha_{0})$, and on $[0,\alpha_{0}]$ if $\alpha_{0}<+\infty$.
\end{Corollary}

\begin{proof}
Fix $0\leq\alpha_{1}<\alpha_{2}<\alpha_{0}$. By
Remark~\ref{rmk:alpha0_positive}, we have
$S_{\alpha_{2}}<A$. Hence, Lemma~\ref{lem:existence-below-threshold}
provides a minimizer
$u_{\alpha_{2}}\in\mathcal{N}_{\alpha_{2}}$. By Lemma~\ref{lem:nehari-fibring}(i), there exists a unique
$\tau:=t_{\alpha_{1}}(u_{\alpha_{2}})>0$ such that
$\tau u_{\alpha_{2}}\in\mathcal{N}_{\alpha_{1}}$. Since
$u_{\alpha_{2}}\in\mathcal{N}_{\alpha_{2}}$, we also have
$t_{\alpha_{2}}(u_{\alpha_{2}})=1$.
For every $t>0$, we have
\[
I_{\alpha_{1}}(tu_{\alpha_{2}})
=
I_{\alpha_{2}}(tu_{\alpha_{2}})
-\frac{\alpha_{2}-\alpha_{1}}{2^{\sharp}}\,t^{2^{\sharp}}\,
\|u_{\alpha_{2}}\|_{2^{\sharp},\Omega}^{2^{\sharp}}
<
I_{\alpha_{2}}(tu_{\alpha_{2}}),
\]
where the strict inequality follows from
$\alpha_{2}>\alpha_{1}$, $t>0$, and
$u_{\alpha_{2}}\not\equiv0$. Taking the supremum with respect to $t>0$ and using
\eqref{eq:fibring-sup} for both parameters $\alpha_{1}$ and
$\alpha_{2}$, we obtain
\[
S_{\alpha_{1}}
\leq I_{\alpha_{1}}(\tau u_{\alpha_{2}})
=\sup_{t>0}I_{\alpha_{1}}(tu_{\alpha_{2}})
<\sup_{t>0}I_{\alpha_{2}}(tu_{\alpha_{2}})
=I_{\alpha_{2}}(u_{\alpha_{2}})
=S_{\alpha_{2}}.
\]
	If $\alpha_{0}<+\infty$, the strict monotonicity of
	$S_\alpha$ and the continuity of $\alpha\mapsto S_\alpha$ yield
	\[
	S_{\alpha_{1}}<S_{\alpha}\leq S_{\alpha_{0}}
	\quad\text{for all }\alpha\in(\alpha_{1},\alpha_{0}),
	\]
	and therefore $S_{\alpha_{1}}<S_{\alpha_{0}}.$
\end{proof}

When $\alpha_0$ is finite, the following corollary holds.
\begin{Corollary}\label{cor:existence_nonexistence}
	Problem  \eqref{eq:main} admits a least energy solution $u_\alpha$ for all $\alpha\in(0, \alpha_0)$ and  admits no least-energy solution for all $\alpha\in(\alpha_0,+\infty)$ if $\alpha_{0} $ is finite.
\end{Corollary}
\begin{proof}
	By Remark~\ref{rmk:alpha0_positive}, we have
	$S_{\alpha}<A$ for every $\alpha\in[0,\alpha_{0})$. Hence,
	Lemma~\ref{lem:existence-below-threshold} provides a least energy
	solution
$u_{\alpha}\in\mathcal{N}_{\alpha}.$
	
	Conversely, suppose that $\alpha>\alpha_{0}$ and, by contradiction,
	that there exists $u\in\mathcal{N}_{\alpha}$ such that
	$I_{\alpha}(u)=S_{\alpha}=A$. By
	Lemma~\ref{lem:nehari-fibring}(i), there exists a unique scalar
	$\tau:=t_{\alpha_{0}}(u)>0$ such that
	$\tau u\in\mathcal{N}_{\alpha_{0}}$.
	Using the characterization \eqref{eq:fibring-sup}, we obtain
	\[
	S_{\alpha_{0}}
	\leq I_{\alpha_{0}}(\tau u)
	=\sup_{t>0}I_{\alpha_{0}}(tu)
	<\sup_{t>0}I_{\alpha}(tu)
	=I_{\alpha}(u)
	=A.
	\]
	This contradicts the equality $S_{\alpha_{0}}=A$, and the proof is complete.
\end{proof}

	\section{Blow-up analysis}\label{sec:blowup}
	To prove that the threshold $\alpha_{0}$ is finite, we argue by
	contradiction. Assume that $\alpha_{0}=+\infty$. Let
	$\alpha_{k}\to+\infty$, and let
	$u_{k}\in\mathcal{N}_{\alpha_{k}}$ be the corresponding least energy
	solutions provided by Lemma~\ref{lem:existence-below-threshold}.
	We first analyze
	the blow-up behavior of $\{u_k\}$, showing that concentration occurs at
	boundary points and that $u_k$ admits a single boundary bubble profile.

Fix $P_{0}\in\partial\Omega$,
without loss of generality, assume $P_{0}=0$. There exist $R>0$
and a smooth function $g\colon B'(2R)\to\R$ with $g(0)=0$ and
$\nabla g(0)=0$ such that
\begin{equation}\label{eq:graph_setup}
	\begin{aligned}
		\Omega\cap B(R)
		&= \bigl\{(x',x_{n})\in B(R)\colon x_{n}>g(x')\bigr\},\\
		\partial\Omega\cap B(R)
		&= \bigl\{(x',x_{n})\in B(R)\colon x_{n}=g(x')\bigr\}.
	\end{aligned}
\end{equation}
	
		Define the map $\psi = (\psi_1, \dots, \psi_n): B(R) \to \mathbb{R}^n$ by
	\begin{equation}\label{eq:psi_def}
			\begin{aligned}
			\psi_j(x) &= x_j - \frac{(g(x') - x_n)}{\left(1 + |\nabla g(x')|^2\right)} \frac{\partial g}{\partial x_j}(x') \quad 1 \leq j \leq n-1, \\
			\psi_n(x) &= x_n - g(x').
		\end{aligned}
	\end{equation}
Then clearly the determinant of the Jacobian of $\psi$ at $0$ is one and hence we can choose $R_0<R$ and $\widetilde{B}\subset B(R)$ an open neighbourhood of zero such that  $\psi$ is a
$C^{\infty}$-diffeomorphism satisfying
\[
\psi \bigl(\Omega\cap\widetilde B\bigr)
= B(R_0)^+ := \{(y', y_n) \in B(R_0) : y_n > 0\},
\qquad
\psi\!\bigl(\partial\Omega\cap\widetilde B\bigr)
= B(R_0)\cap\partial\mathbb{R}^n_+.
\]
A direct computation shows that, for $u\in H^{1}(\Omega)$ and
$v(y):=u(\psi^{-1}(y))$,
\begin{equation}\label{equ3.3}
		(\Delta u)(\psi^{-1}(y)) = \sum_{j,k} a_{jk}(y) \frac{\partial^2 v}{\partial y_j \partial y_k} + \sum_j b_j(y) \frac{\partial v}{\partial y_j},
	\end{equation}
	\begin{equation}\label{equ3.4}
\frac{\partial u}{\partial \nu} (\psi^{-1}(y)) = -c(y) \frac{\partial v}{\partial y_n} \quad \text{on } y_n = 0,
\end{equation}
with $a_{jk}(y)=\delta_{jk}+O(|y|)$, $c(y)=1+O(|y|)$, and $b_{j}$
smooth and bounded.

\begin{Lemma}\label{lem3.1}
Suppose $S_{\alpha}<A$ for all $\alpha\ge0$, and let $\alpha_{k}\to+\infty$ as $k\rightarrow \infty$.  For each $k$, let $u_k\in H^1(\Omega)$ be a minimizer of $I_{\alpha_{k}}$, so that $I_{\alpha_{k}}(u_k)=S_{\alpha_{k}}$. Then
\begin{enumerate}
		\item[\textup{(i)}] $S_{\alpha_k}\to A$ as $k\to\infty$.
		\item[\textup{(ii)}] $u_k\rightharpoonup0$ weakly in $H^1(\Omega)$ and $M_k:=\max\limits_{\overline{\Omega}}u_k\to+\infty$ as $k\to\infty$.
		\item[\textup{(iii)}]  $\lim\limits_{k\to\infty}\alpha_k\|u_k\|_{2^\sharp, \Omega}^{2^\sharp}=0$.
\end{enumerate}
\end{Lemma}
	\begin{proof}
Since $u_k\in\mathcal{N}_{\alpha_{k}}$ satisfies $I_{\alpha_k}(u_k)=S_{\alpha_k}$, we have
		\begin{equation}\label{equ2.39}
			\|u_k\|_{2^*,\Omega}^{2^*}+\|u_k\|_{2^\sharp,\partial\Omega}^{2^\sharp}=\|u_k\|^2+\alpha_k\|u_k\|_{2^\sharp,\Omega}^{2^\sharp}.
		\end{equation}
	Hence,
	\begin{equation}\label{equa3.3}
		S_{\alpha_k}=I_{\alpha_k}(u_k)
		=\Bigl(\frac12-\frac1{2^\sharp}\Bigr)\|u_k\|^2
		+\Bigl(\frac1{2^\sharp}-\frac1{2^*}\Bigr)
		\|u_k\|_{2^*,\Omega}^{2^*}
		\ge
		\Bigl(\frac12-\frac1{2^\sharp}\Bigr)\|u_k\|^2.
	\end{equation}
	It follows that $\{u_k\}$ is bounded in $H^1(\Omega)$. From \eqref{equ2.39} and $\alpha_k\to+\infty$, we conclude that $u_k\rightharpoonup0$ in $H^1(\Omega)$ as $k\to\infty$.
		
	For each $u_k\in H^1(\Omega)$ there exists a unique $t_k>0$ such that $\sup\limits_{t > 0}\Psi_{\Omega}(tu_k)=\Psi_{\Omega}(t_ku_k)$, characterized by
	\begin{equation}\label{equa3.4}
		t_k^2\|\nabla u_k\|_{2,\Omega}^2=t_k^{2^*}\|u_k\|_{2^*,\Omega}^{2^*}+t_k^{2^\sharp}\|u_k\|_{2^\sharp,\partial\Omega}^{2^\sharp}.
	\end{equation}
	Since $2^\sharp-2=\frac{1}{2}(2^*-2)$, equation \eqref{equa3.4} is quadratic in $t_k^{2^\sharp-2}$,  yielding the explicit solution
	\begin{equation*}
		t_k^{2^\sharp-2}=\frac{-\|u_k\|_{2^\sharp,\partial\Omega}^{2^\sharp}+\sqrt{\|u_k\|_{2^\sharp,\partial\Omega}^{2\cdot2^\sharp}+4\|\nabla u_k\|_{2,\Omega}^2\|u_k\|_{2^*,\Omega}^{2^*}}}{2\|u_k\|_{2^*,\Omega}^{2^*}}.
	\end{equation*}
	It follows from \eqref{equ2.39} that $t_k\le 1$. Using \eqref{equa3.4} and the expression for $I_{\alpha_{k}}$, we have
	\begin{equation*}
		\Psi_{\Omega}(t_ku_k)=(\frac{1}{2}-\frac{1}{2^\sharp})t_k^2\|\nabla u_k\|_{2,\Omega}^2+(\frac{1}{2^\sharp}-\frac{1}{2^*})t_k^{2^*}\|u_k\|_{2^*,\Omega}^{2^*}\le I_{\alpha_k}(u_k).
	\end{equation*}
By Lemma \ref{lem:local_compactness},	since $u_k\rightharpoonup0$,  we conclude that
		\begin{equation}\label{equ2.41}
			A\ge\lim_{k \to \infty}S_{\alpha_k}=\lim_{k \to \infty}I_{\alpha_k}(u_k)\ge\liminf_{k \to \infty}\sup_{t > 0}\Psi_{\Omega}(tu_k)\ge A.
		\end{equation}
		Hence $S_{\alpha_{k}}\to A$, and as a consequence $t_k\to 1$.
		
 Evaluating equation \eqref{eq:main}, the value $M_k$ satisfies
		\begin{equation}\label{equ2.42}
			\lambda +\alpha_kM_k^{2^\sharp-2}\le M_k^{2^*-2}.
		\end{equation}
		Since $\alpha_k\to+\infty$ and $2^*>2^\sharp$, we must have $M_k\to+\infty$ as $k\to\infty$.
		
		Moreover, since $I_{\alpha_k}(u_k)\to A$, by \eqref{equ2.41},
		we have
		\begin{equation}\label{eq:alpha_decay}
			\lim_{k\to\infty}\alpha_k\int_\Omega |u_k|^{2^\sharp}\d x=0,
			\qquad
			\lim_{k\to\infty}\int_\Omega |u_k|^2\d x=0.
		\end{equation}
	\end{proof}
	
	\begin{Lemma}\label{lem:concentration_locus}
	Let $\alpha_k \to \alpha_0 \in (0,+\infty]$  as $k\rightarrow \infty$  , and let $u_k \in H^1(\Omega)$ be a
	minimizer for $I_{\alpha_k}$ with $S_{\alpha_k} < A$ and
	$S_{\alpha_k} \to A$  as $k\rightarrow \infty$  .  When $\alpha_0 < +\infty$, suppose in addition that
	$u_k \rightharpoonup 0$ in $H^1(\Omega)$.
	Let $P_k \in \overline{\Omega}$ satisfy $u_k(P_k) = \max\limits_{\overline{\Omega}} u_k =: M_k$,
	and define
	\begin{equation}\label{eq:delta-def-new}
		\delta_k := \kappa_n^{2/(n-2)}\,M_k^{-2/(n-2)},
		\qquad
		\kappa_n := V_1(0) = C_n\!\left(\tfrac{n-2}{2(n-1)}\right)^{(n-2)/2}.
	\end{equation}
	Then we have
  \begin{equation}\label{eq:alpha_delta}
  	\displaystyle\lim_{k\to\infty}\alpha_k\delta_k = 0,
  \end{equation}
  \begin{equation*}
  	\displaystyle\lim_{k\to\infty}\frac{\mathrm{dist}(P_k,\partial\Omega)}{\delta_k}=0,
  \end{equation*}
  and $P_k \in \partial\Omega$,  for all sufficiently large $k$.
\end{Lemma}

\noindent Note: When $\alpha_{0}=+\infty$, Lemma \ref{lem3.1} guarantees that $S_{\alpha_k}\to A$, $u_k\rightharpoonup0$ and $M_k\to+\infty$ both hold automatically. When $\alpha_{0}<+\infty$, the assumption $u_k\rightharpoonup0$ forces $M_k\to+\infty$.  Indeed, if $M_k$ were bounded, by the interpolation inequality,
$
\|u_k\|_{L^{2^*}(\Omega)}
\le M_k^{\,1-\theta}\|u_k\|_{L^{s}(\Omega)}^{\theta}
$
for some $s<2^*$ and $\theta\in(0,1)$, we would have
$\|u_k\|_{2^*,\Omega}\to0$. Since
$\|u_k\|_{2^\sharp,\partial\Omega}\to0$, it follows that
$S_{\alpha_k}=I_{\alpha_k}(u_k)\to0$, which contradicts
$S_{\alpha_k}\to A>0$.

\begin{proof}
	Let $P_k \in \overline{\Omega}$ satisfy $M_k  = u_k(P_k)$ and define
\( v_k(x) := \delta_k^{\frac{n-2}{2}} u_k(\delta_k x + P_k) \) for \( x \in \Omega_k := \frac{\Omega - P_k}{\delta_k} \). Then $v_k$ satisfies
\begin{equation}\label{equ3.20}
	\begin{cases}
		-\Delta v_k + \lambda\delta_k^2 v_k + \alpha_k \delta_k v_k^{2^\#-1} = v_k^{2^*-1} & \text{in } \Omega_k, \\
		0 < v_k \leq v_k(0) = \kappa_n & \text{in } \Omega_k, \\
		\frac{\partial v_k}{\partial \nu} = v_k^{2^\sharp-1} & \text{on } \partial\Omega_k.
	\end{cases}
\end{equation}
Rewriting \eqref{equ2.42} in terms of $\delta_k$, we obtain
\begin{equation}\label{eq:bdd}
	\lambda\delta_k^2+\alpha_k\kappa_n^{2/(n-2)}\delta_k
	\le\kappa_n^{4/(n-2)}.
\end{equation}
Hence, the sequences $\{\alpha_k\delta_k\}$ and
$\{\lambda\delta_k^2\}$ are bounded. Up to a subsequence,  we may assume that
\[
P_k\to P_0,\qquad
\alpha_k\delta_k\to a,\qquad
\Omega_k:=\frac{\Omega-\{P_k\}}{\delta_k}\to\Omega_\infty.
\]
Since $\lambda\delta_k^2\to0$, applying the elliptic estimates
of~\cite{ADN} to~\eqref{equ3.20}, we obtain
\[
v_k\to v
\quad\text{in }C_{\mathrm{loc}}^2(\Omega_\infty),
\]
where $v\in D^{1,2}(\Omega_\infty)$ satisfies
\[
\begin{cases}
	-\Delta v+a v^{2^\sharp-1}=v^{2^*-1}
	&\text{in }\Omega_\infty,\\
	0<v\le v(0)=\kappa_n
	&\text{in }\Omega_\infty,\\
	\dfrac{\partial v}{\partial\nu}=v^{2^\sharp-1}
	&\text{on }\partial\Omega_\infty.
\end{cases}
\]
 By weak lower semicontinuity of the norm, $v\in L^{2^*}(\Omega_\infty)$ and $\nabla v\in L^2(\Omega_\infty)$.
		Since $v(0)=\kappa_{n}>0$ and $v$ is continuous and positive, we may fix $R>0$ such that
	\[
	\int_{\Omega_\infty\cap B_R} |v|^{2^\sharp}\d x \ge c_R>0.
	\]
	Since  $v_k\to v$ in $L^{2^\sharp}(\Omega_\infty\cap B_R)$ and
	$v\not\equiv0$ in $\Omega_\infty\cap B_R$, we have
		\[ 
	\int_{\Omega_k\cap B_R} |v_k|^{2^\sharp}\d x \ge \frac{c_R}{2}>0
	\]
	 for all sufficiently large $k$.
	 By Lemma~\ref{lem3.1}\textup{(iii)}, we have
	 \[
	 \alpha_k\int_{\Omega}|u_k|^{2^\sharp}\,\mathrm{d}x\to0\quad \text{as } k\to\infty,
	 \]
	 and the rescaling identity gives
	 \[
	 \alpha_k\delta_k
	 \int_{\Omega_k}|v_k|^{2^\sharp}\,\mathrm{d}x
	 =
	 \alpha_k\int_{\Omega}|u_k|^{2^\sharp}\,\mathrm{d}x .
	 \]
	 Moreover, since $ \int_{\Omega_k}|v_k|^{2^\sharp}\,\mathrm{d}x
	 \ge \frac{c_R}{2}$
	 for all sufficiently large $k$, we obtain
	 \begin{equation}\label{equ2.47}
	 0\le\alpha_k\delta_k\le \frac{2}{c_R}\alpha_k\delta_k\,\int_{\Omega_k} |v_k|^{2^\sharp}\d x
	 =\frac{2}{c_R}\alpha_k\int_{\Omega} |u_k|^{2^\sharp}\d x \to 0\quad \text{as } k\to\infty.
	 \end{equation}
	 Hence,
	 \[
	 \alpha_k\delta_k\to0\quad \text{as } k\to\infty.
	 \]
	 Together with $\lambda\delta_k^2\to0$, we deduce that the limit
	 $v$ satisfies
	 \begin{equation}\label{eq:limit_critical}
	 	\begin{cases}
	 		-\Delta v=v^{2^*-1}
	 		&\text{in }\Omega_\infty,\\[2pt]
	 		-\partial_\nu v=v^{2^\sharp-1}
	 		&\text{on }\partial\Omega_\infty,\\[2pt]
	 		0<v\le\kappa_n,\qquad v(0)=\kappa_n.
	 	\end{cases}
	 \end{equation}
	 
	 Set
	 \[
	 L:=\lim_{k\to\infty}
	 \frac{\operatorname{dist}(P_k,\partial\Omega)}{\delta_k}.
	 \]
	Arguing by contradiction, suppose that $L=+\infty$.
	Then $\Omega_\infty=\mathbb{R}^n$, and the limiting function $v$
	satisfies 	$-\Delta v=v^{2^*-1}$ in $\mathbb{R}^n$.
	 By the classification theorem of Caffarelli--Gidas--Spruck~\cite{CGS},
	 $v$ is an Aubin--Talenti bubble with Sobolev energy $\frac{1}{n}S^{\frac{n}{2}}$, 
	 where $S$ denotes the sharp Sobolev constant in $\mathbb{R}^n$.
	 By the weak lower semicontinuity, we deduce that
	 \[
	 A=\lim_{k\to\infty}I_{\alpha_k}(u_k)
	 \ge
	 \frac1nS^{\frac n2}.
	 \]
	 On the other hand, let $U_\varepsilon$ denote the Aubin--Talenti bubble
	 centred at the origin. Since $U_\varepsilon$ is symmetric with respect to
	 the hyperplane $\{x_n=0\}$,	we have
	 \[
	 \int_{\mathbb{R}^n_+}|\nabla U_\varepsilon|^2\,\mathrm{d}x
	 =
	 \int_{\mathbb{R}^n_+}|U_\varepsilon|^{2^*}\,\mathrm{d}x
	 =
	 \frac12S^{\frac n2}.
	 \]
	 A direct computation yields
	 \[
	 \sup_{t>0}\left(
	 \frac{t^2}{2}
	 \|\nabla U_\varepsilon\|_{L^2(\mathbb R^n_+)}^2
	 -\frac{t^{2^*}}{2^*}
	 \|U_\varepsilon\|_{L^{2^*}(\mathbb R^n_+)}^{2^*}
	 \right)
	 =
	 \frac{S^{\frac n2}}{2n}.
	 \]
	Using $U_\varepsilon$ as a test function in the variational characterization
	of $A$, we obtain
	\[
	A\le
	\sup_{t>0}\Psi_{\mathbb R^n_+}(tU_\varepsilon).
	\]
	Since $U_\varepsilon|_{\partial\mathbb R^n_+}\not\equiv0$, the boundary term is
	strictly negative for every $t>0$. Therefore,
	\[
	A<
	\frac{S^{\frac n2}}{2n}
	<
	\frac{S^{\frac n2}}{n},
	\]
	which is impossible. Hence $L<+\infty$, and consequently
	$P_k\to P_0\in\partial\Omega$.
	
	Without loss of generality assume
 \( P_0 = 0 \). Applying the flattening diffeomorphism \( \psi \) defined in \eqref{eq:psi_def}. Let  \( q_k:=\psi(P_k) \in B(R_0)^+\) and define  	 \[
 \tilde{v}_k(x)
 := \delta_k^{\frac{n-2}{2}}\,v_k(\delta_k x + q_k),
 \qquad
 x \in B_k := \frac{B(R_0)^+ - q_k}{\delta_k}.
 \]
   Setting $q_{nk}=(\psi(P_k))_n$, the assumption $L<+\infty$ gives $q_{nk}/\delta_k\to L$. The function $\tilde{v}_k$ satisfies
	\begin{equation*}
		\begin{cases}
			-\sum_{ij} \left(\delta_{ij} + O(|\delta_k x + q_k|)\right) \frac{\partial^2 \tilde{v}_k}{\partial x_i \partial x_j} +O (\delta_k) \nabla \tilde{v}_k+ \lambda\delta_k^2 \tilde{v}_k+\alpha_k\delta_k \tilde{v}_k^{2^\sharp-1}=\tilde{v}_k^{2^*-1} \quad \text{in } B_k,\\
			-\left(1+O(|\delta_k x + q_k|)\right)\frac{\partial \tilde{v}_k}{\partial x_n} = \tilde{v}_k^{2^\sharp-1} \quad \text{on } B_k \cap \left\{x_n = -\frac{q_{nk}}{\delta_k}\right\},
		\end{cases}
	\end{equation*}
	with $\tilde{v}_k(0) = \kappa_n$, $0 \leq \tilde{v}_k \leq \kappa_n.$ By the elliptic regularity estimates, we obtain
	\[
	\tilde v_k\to\omega_1
	\quad\text{in }C^2_{\mathrm{loc}}
	(\overline{B_\infty}),
	\]
	where
	\[
	B_\infty:=\{x\in\mathbb R^n:x_n>-L\},
	\]
	and the limit function
	$\omega_1\in H^1(B_\infty)$ satisfies
	\begin{equation}\label{equa3.20}
		\begin{cases}
			-\Delta\omega_1=\omega_1^{2^*-1}
			&\text{in }B_\infty,\\[2pt]
			-\partial_{x_n}\omega_1=\omega_1^{2^\sharp-1}
			&\text{on }\partial B_\infty,\\[2pt]
			0<\omega_1\le\kappa_n
			&\text{in }B_\infty,
		\end{cases}
	\end{equation}
	with $\omega_1(0)=\kappa_n$.  Applying the classification theorem of Li--Zhu~\cite{LZ} to the
	half-space $B_\infty=\{x_n>-L\}$, we have
	\[
	\omega_1(x)
	=C_n
	\left(
	\frac{1}
	{1+|x'|^2+|x_n+L+x_n^0|^2}
	\right)^{\frac{n-2}{2}},
	\]
	where $x_n^0=\sqrt{\frac{n}{n-2}}>0.$ By direct computation, we deduce that
	\begin{equation}\label{eq:monotonicity}
		\partial_{x_n}\omega_1
		=
		-\frac{(n-2)C_n(x_n+L+x_n^0)}
		{\left(1+|x'|^2+|x_n+L+x_n^0|^2\right)^{n/2}}
		<0,
		\quad
		\forall x\in\overline{B_\infty}.
	\end{equation}
	So $\omega_1$ attains its maximum on $\partial B_\infty=\{x_n=-L\}$. Since $\omega_1(0)=\kappa_n=V_1(0)$ is the maximum, we must have $L=0$. Therefore $\omega_1=V_1$ and $P_k\to P_0\in\partial\Omega$ with $\frac{\mathrm{dist}(P_k,\partial\Omega)}{\delta_k}\to0.$
	
	It remains to prove that $P_k\in\partial\Omega$ for all sufficiently
	large $k$. Assume by contradiction that there exists a subsequence such that $P_k\in\Omega$ for all $k$.  Since $P_k\in\Omega$ is an interior maximum point of $u_k$, we have
	\[
	\nabla u_k(P_k)=0.
	\]
	By the definition of the rescaling, this implies
	\[
	\nabla\widetilde v_k(0)=0.
	\]
	Passing to the limit by the $C^2_{\mathrm{loc}}$-convergence
	$\widetilde v_k\to\omega_1$, we obtain
	\[
	\nabla\omega_1(0)=0.
	\]
	Since $L=0$, the point $0$ lies on
	$\partial B_\infty=\{x_n=0\}$. The boundary condition in
	\eqref{equa3.20} then yields
	\[
	-\partial_{x_n}\omega_1(0)
	=\omega_1^{2^\sharp-1}(0)
	=\kappa_n^{2^\sharp-1},
	\]
	and hence
	\[
	\partial_{x_n}\omega_1(0)
	=-\kappa_n^{2^\sharp-1}\neq0,
	\]
	which contradicts $\nabla\omega_1(0)=0$.
	Therefore, $P_k\in\partial\Omega$ for all sufficiently large $k$.
\end{proof}
	\begin{Lemma}\label{lem:bubble_profile}
	Under the hypotheses  of Lemma~\ref{lem:concentration_locus},  we have
	\[
	\lim_{k\to\infty}\|\nabla(u_k - V_{\delta_k,P_k})\|_{2,\Omega} = 0.
	\]
\end{Lemma}
\begin{proof}
	Since $P_k\in\partial\Omega$ for all sufficiently large $k$, the outward
	unit normal $\nu(P_k)$ is well defined, and the rigid motion
	$T_{P_k}$ introduced in~\eqref{eq:Txi_def} is available.
	Fix an orthonormal tangential frame
	$\{e_1(P_k),\dots,e_{n-1}(P_k)\}$ at $P_k$.
	By Lemma~\ref{lem:frame_indep}, the function
	$V_{\delta_k,P_k}$ is independent of this choice. Define
	\begin{equation}
		\hat v_k(y):=\delta_k^{(n-2)/2}u_k\!\bigl(T_{P_k}(\delta_k y)\bigr),
		\qquad y\in\hat\Omega_k:=\frac{T_{P_k}^{-1}(\Omega)}{\delta_k}.
		\label{eq:hatv-def}
	\end{equation}
	Because $T_{P_k}$ is an isometry, a change of variables gives
	\begin{equation}\label{eq:isom-grad-new}
		\|\nabla\hat{v}_k\|_{2,\hat\Omega_k}^2 = \|\nabla u_k\|_{2,\Omega}^2.
	\end{equation}
	From the definition of $V_{\delta_k,P_k}$ in \eqref{equa2.9} and the same
	change of variables, we have
	\begin{equation}\label{eq:isom-bubble-new}
		\delta_k^{(n-2)/2}\,V_{\delta_k,P_k}\!\left(T_{P_k}(\delta_k y)\right) = V_1(y).
	\end{equation}
	Combining \eqref{eq:isom-grad-new}--\eqref{eq:isom-bubble-new} yields
	\begin{equation}\label{eq:red-iii-new}
		\|\nabla(u_k - V_{\delta_k,P_k})\|_{2,\Omega}^2
		= \|\nabla(\hat{v}_k - V_1)\|_{2,\hat\Omega_k}^2.
	\end{equation}
	So it suffices to show $\|\nabla(\hat{v}_k - V_1)\|_{2,\hat\Omega_k} \to 0$.
	
	Transferring \eqref{eq:main} through the isometry $T_{P_k}$ and rescaling
	by $\delta_k$, the function $\hat v_k$ satisfies
	\begin{equation}
		\begin{cases}
			-\Delta\hat v_k+\lambda\delta_k^2\,\hat v_k
			+\alpha_k\delta_k\,\hat v_k^{2^\sharp-1}=\hat v_k^{2^*-1}
			& \text{in }\hat\Omega_k,\\[2pt]
			\frac{\partial\hat v_k}{\partial\nu}=\hat v_k^{2^\sharp-1}
			& \text{on }\partial\hat\Omega_k,\\[2pt]
			0<\hat v_k\le \hat v_k(0)=\kappa_n.
		\end{cases}
		\label{eq:hatv-eq}
	\end{equation}
	Since $T_{P_k}^{-1}$ maps $\nu(P_k)$ to $-e_n$ and
	$\operatorname{dist}(P_k,\partial\Omega)/\delta_k\to 0$, a standard
	$C^1$-flattening argument applied to $\partial\hat\Omega_k$ near the origin
	gives, as $k\to\infty$,
	\[
	\hat\Omega_k\nearrow\mathbb R^n_+,\qquad
	\partial\hat\Omega_k\to\partial\mathbb R^n_+\quad\text{in }C^1_{\rm loc}.
	\]
Applying the elliptic estimates in~\cite{ADN} to~\eqref{eq:hatv-eq},
we obtain
\[
\hat v_k\to\hat v
\quad\text{in }C^2_{\mathrm{loc}}
(\overline{\mathbb R^n_+}),
\]
where the limit function
$\hat v\in D^{1,2}(\mathbb R^n_+)$ satisfies
\begin{equation}
	\begin{cases}
		-\Delta\hat v=\hat v^{2^*-1}
		&\text{in }\mathbb R^n_+,\\[2pt]
		-\partial_{y_n}\hat v=\hat v^{2^\sharp-1}
		&\text{on }\partial\mathbb R^n_+,\\[2pt]
		0<\hat v\le \hat v(0)=\kappa_n.
	\end{cases}
	\label{eq:limit-pde}
\end{equation}
	By the classification theorem of Li--Zhu~\cite{LZ}, we have $\hat v=V_{\varepsilon,z}$
	for some $\varepsilon>0$ and $z\in\partial\mathbb R^n_+$. Since $V_{\varepsilon,z}$ attains its maximum at $z$ and
	$\hat v=V_{\varepsilon,z}$ attains its maximum at $0$, it follows that
	$z=0$. Since
	\[
	\hat v(0)=\kappa_n=V_1(0),
	\]
	we must have
	\[
	\varepsilon^{-(n-2)/2}V_1(0)=V_1(0),
	\]
	and hence $\varepsilon=1$. It follows that
	\begin{equation}\label{eq:hatv-is-V1}
		\hat v=V_1.
	\end{equation}
	
	Set $\hat w_k:=\hat v_k-V_1\cdot\mathbf{1}_{\hat\Omega_k}$. By~\eqref{eq:isom-grad-new} and 
	$S_{\alpha_k}\to A$ as $k\to\infty$, we deduce that
	$\{\hat v_k\}$ is bounded in $D^{1,2}(\mathbb{R}^n_+)$.  It follows from \eqref{eq:hatv-is-V1} that
	\begin{equation}
		\hat w_k\rightharpoonup 0\quad\text{in }D^{1,2}(\mathbb R^n_+).
		\label{eq:wk-weak}
	\end{equation}
	Since $u_k\in\mathcal{N}_{\alpha_k}$ and $\lambda\delta_k^2\to 0$,
	$\alpha_k\delta_k\to 0$, the rescaled Nehari identity gives
	\begin{equation}\label{eq:rescaled-nehari}
		\int_{\hat\Omega_k}\!|\nabla\hat v_k|^2\,\mathrm{d}y
		=\int_{\hat\Omega_k}\!\hat v_k^{2^*}\,\mathrm{d}y
		+\int_{\partial\hat\Omega_k}\!\hat v_k^{2^\sharp}\,\mathrm{d}\sigma_y
		+o(1).
	\end{equation}
	Testing \eqref{eq:limit-pde} against $V_1$ gives
	\begin{equation}	\label{eq:V1-nehari}
		\int_{\mathbb R^n_+}|\nabla V_1|^2\,\mathrm{d}y
		=\int_{\mathbb R^n_+}V_1^{2^*}\,\mathrm{d}y
		+\int_{\partial\mathbb R^n_+}V_1^{2^\sharp}\,\mathrm{d}\sigma_y.
	\end{equation}
Subtracting \eqref{eq:V1-nehari} from~\eqref{eq:rescaled-nehari}
and applying the Br\'ezis--Lieb lemma, we obtain
	\begin{equation}	\label{eq:wk-nehari}
		\int_{\hat\Omega_k}|\nabla\hat w_k|^2\,\mathrm{d}y
		=\int_{\hat\Omega_k}|\hat w_k|^{2^*}\,\mathrm{d}y
		+\int_{\partial\hat\Omega_k}|\hat w_k|^{2^\sharp}\,\mathrm{d}\sigma_y
		+o(1).
	\end{equation}
	Since $I_{\alpha_k}(u_k)=S_{\alpha_k}\to A$
as $k\to\infty$, it follows from the rescaling identity~\eqref{eq:isom-grad-new} that
	\begin{equation}\label{eq:energy-comp-new}
		\Psi_{\hat\Omega_k}(\hat{v}_k)
		= I_{\alpha_k}(u_k) - \tfrac{1}{2}\lambda\delta_k^2\|\hat{v}_k\|_{2,\hat\Omega_k}^2
		+ \tfrac{1}{2^\sharp}\alpha_k\delta_k\|\hat{v}_k\|_{2^\sharp,\hat\Omega_k}^{2^\sharp}
		\to A = \Psi_{\R^n_+}(V_1).
	\end{equation}
	Combining~\eqref{eq:energy-comp-new}
	with~\eqref{eq:rescaled-nehari}--\eqref{eq:wk-nehari},
	we obtain
	\begin{equation}\label{eq:Psi-wk-new}
		\Psi_{\hat\Omega_k}(\hat{w}_k)
		= \Psi_{\hat\Omega_k}(\hat{v}_k) - \Psi_{\R^n_+}(V_1) + o(1) \to 0.
	\end{equation}
	Combining the definition of
	$\Psi_{\hat\Omega_k}(\hat w_k)$ with
	\eqref{eq:wk-nehari}, we obtain
	\[
	\Psi_{\hat\Omega_k}(\hat w_k)
	=\frac{1}{n}\int_{\hat\Omega_k}|\hat w_k|^{2^*}\,\mathrm{d}y
	+\frac{1}{2(n-1)}\int_{\partial\hat\Omega_k}|\hat w_k|^{2^\sharp}
	\,\mathrm{d}\sigma_y+o(1).
	\]
	Since both terms on the right-hand side are nonnegative, we have
	\[
	\int_{\hat\Omega_k}|\hat w_k|^{2^*}\,\mathrm{d}y\to 0,\qquad
	\int_{\partial\hat\Omega_k}|\hat w_k|^{2^\sharp}\,\mathrm{d}\sigma_y\to 0.
	\]
	Substituting these limits into~\eqref{eq:wk-nehari}, we conclude that
	\[
	\int_{\hat\Omega_k}|\nabla\hat w_k|^2\,\mathrm{d}y\to 0.
	\]
	Together with~\eqref{eq:red-iii-new}, this yields
	$\|\nabla(u_k-V_{\delta_k,P_k})\|_{2,\Omega}\to 0$.
\end{proof}

		\section{Energy expansion and  nonexistence}\label{sec:section4}
		
		This section is devoted to proving   Theorem~\ref{thm:main} (i)--(ii).  It is sufficient to prove that $\alpha_0$ is finite. The argument proceeds by contradiction:
			assuming $\alpha_{0}=+\infty$,  Lemma~\ref{lem3.1} and
		Lemma~\ref{lem:concentration_locus} provide a sequence of least energy
		solutions $u_{k}$ with $\alpha_{k}\to+\infty$ that weakly converge
			to zero and concentrate at boundary points $P_{k}\in\partial\Omega$.
		We can derive a precise asymptotic expansion of $I_{\alpha_{k}}(u_{k})$ around the
		concentration point and  show that  the energy
		strictly exceeds $A$ for $k$ sufficiently large. This contradict the fact that
		$I_{\alpha_{k}}(u_{k})=S_{\alpha_{k}}<A$.
		
		\subsection{Coercivity of the linearized operator}
		\label{subsec:spectral}
		
		Recall that the bubble manifold $\mathcal{M}\setminus\{0\}$
		is a smooth $(n+1)$-dimensional submanifold of $D^{1,2}(\Omega)$, and
		its tangent space at any point $V_{\varepsilon,\xi}\in\mathcal{M}$ is
		\begin{equation}\label{eq:tangent_decomp}
			T_{1,V_{\varepsilon,\xi}}(\mathcal{M})
			= \mathrm{span}\Bigl\{V_{\varepsilon,\xi},\,
			\frac{\partial V_{\varepsilon,\xi}}{\partial\varepsilon},\,
			\frac{\partial V_{\varepsilon,\xi}}{\partial\tau_i},\;1\le i \le n-1\Bigr\},
		\end{equation}
		where $\{\tau_{1},\dots,\tau_{n-1}\}=\{e_{1}(\xi),\dots,e_{n-1}(\xi)\}$
		is the tangential frame of $T_{\xi}\partial\Omega$.
		Throughout this section, the orthogonality condition
		$\phi\perp T_{1,V_{\varepsilon,\xi}}(\mathcal{M})$ refers to the
		gradient inner product in $D^{1,2}(\Omega)$:
		\begin{equation}\label{eq:ortho_def}
			\int_{\Omega}\nabla\phi\cdot\nabla V_{\varepsilon,\xi}\,\d x=0,
			\quad
			\int_{\Omega}\nabla\phi\cdot\nabla\bigl(\partial_{\varepsilon}V_{\varepsilon,\xi}\bigr)\,\d x=0,
			\quad
			\int_{\Omega}\nabla\phi\cdot\nabla\bigl(\partial_{\tau_{i}}V_{\varepsilon,\xi}\bigr)\,\d x=0
		\end{equation}
		for $1\le i\le n-1$.
		
		\begin{Lemma}[{\cite[Theorem 2.1]{CPV}}]\label{lem:kernel_Ji}
			The kernel of the linearized problem at $V_{1}$,
			\begin{equation}\label{eq:linearised}
				\begin{cases}
					-\Delta v=\dfrac{n+2}{n-2}\,V_{1}^{\,4/(n-2)}\,v
					& \text{in }\R^{n}_{+},\\[1mm]
					\dfrac{\partial v}{\partial\nu}=\dfrac{n}{n-2}\,V_{1}^{\,2/(n-2)}\,v
					& \text{on }\partial\R^{n}_{+},
				\end{cases}
			\end{equation}
			in $D^{1,2}(\mathbb{R}^n_+)$ is $n$-dimensional, spanned by
			\[
			J_{i}(x)=\frac{\partial V_{1,\xi}}{\partial x_{i}}\Big|_{\xi=0}
			=\frac{(2-n)C_{n}\,x_{i}}{\bigl(1+|x'|^{2}+(x_{n}+x_{n}^{0})^{2}\bigr)^{n/2}},
			\qquad i=1,\dots,n-1,
			\]
			and
			\[
			J_{n}(x)=\frac{\partial V_{\varepsilon,0}}{\partial\varepsilon}\Big|_{\varepsilon=1}
			=\frac{(n-2)C_{n}}{2}\cdot
			\frac{|x|^{2}-|x_{n}^{0}|^{2}-1}{\bigl(1+|x'|^{2}+(x_{n}+x_{n}^{0})^{2}\bigr)^{n/2}}.
			\]
		\end{Lemma}
			Let $\mathcal{V} \subset D^{1,2}(\mathbb{R}^n_+)$ denote the
		$(n+1)$-dimensional subspace
		\begin{equation}\label{eq:V_space}
			\mathcal{V} := \mathrm{span}\{V_1, J_1, \ldots, J_n\}.
		\end{equation}
		Define the limit quadratic form
		\begin{equation}\label{eq:Q_infty}
			Q_{\infty}(\psi)
			:=\int_{\R^{n}_{+}}\!|\nabla\psi|^{2}\,\d x
			-(2^{*}-1)\!\int_{\R^{n}_{+}}\!V_1^{\,2^{*}-2}\psi^{2}\,\d x
			-(2^{\sharp}-1)\!\int_{\partial\R^{n}_{+}}\!V_1^{\,2^{\sharp}-2}\psi^{2}\,\d x'.
		\end{equation}
		
		Set $T_{c}:=-\sqrt{n/(n-2)}$,
		$\rho(z):=\frac{2}{1+|z'|^{2}+(z_{n}-T_{c})^{2}}$, and
		$\widetilde C_{n}:=2^{-(n-2)/2}C_{n}$, and let
		$\pi_{s}\colon\R^{n}_{+}\to\Sigma\subset S^{n}$ denote the
		stereographic projection centred at $(0,\dots,0,T_{c})$
		\textup{(see~\cite[(3.1)]{HanLi2000})}. The associated conformal
		transformation $\mathcal T\colon H^{1}(\Sigma)\to D^{1,2}(\R^{n}_{+})$ is
		\begin{equation}\label{eq:adjusted_transform}
			\bigl(\mathcal T\Phi\bigr)(z)
			:=\widetilde C_{n}\,\rho(z)^{(n-2)/2}\,\Phi\bigl(\pi_{s}^{-1}(z)\bigr)
			=V_{1}(z)\,\Phi\bigl(\pi_{s}^{-1}(z)\bigr),\qquad z\in\R^{n}_{+}.
		\end{equation}
		The associated quadratic form on $H^{1}(\Sigma)$ is
		\begin{equation}\label{eq:Q2_def_lemma}
			Q_{2}(\Phi,\Phi):=\tfrac{1}{2}\!\int_{\Sigma}\!\bigl(|\nabla_{g}\Phi|^{2}-n\Phi^{2}\bigr)\,\d\sigma_{g}
			+\tfrac{T_{c}}{2}\!\int_{\partial\Sigma}\!\Phi^{2}\,\d\tau.
		\end{equation}
		
		\begin{Lemma}\label{lem:conformal_equiv}
			The map $\mathcal T$ in~\eqref{eq:adjusted_transform} is a
			linear isomorphism between $H^{1}(\Sigma)$ and
			$D^{1,2}(\R^{n}_{+})$ satisfying the following properties.
			\begin{enumerate}
				\item[\textup{(i)}] For every $\Phi\in H^{1}(\Sigma)$,
				\begin{equation}\label{eq:Q_infty_Q2}
					Q_{\infty}(\mathcal T\Phi)=2\widetilde C_{n}^{2}\,Q_{2}(\Phi,\Phi).
				\end{equation}
				\item[\textup{(ii)}] The map $\mathcal T$ restricts to a bijection
				\begin{equation}\label{eq:kernel_lemma}
					\mathrm{Ker}\,Q_{2}=\mathrm{span}\{1,\zeta_{1},\dots,\zeta_{n}\}
					\;\xrightarrow{\;\mathcal T\;}\;
					\mathcal{V}=\mathrm{span}\{V_{1},J_{1},\dots,J_{n}\},
				\end{equation}
				with $\mathcal T 1=V_{1}$ and $\mathcal T(\zeta_{l}|_{\Sigma})=J_{l}$
				for $l=1,\dots,n$.
				\item[\textup{(iii)}] There exist constants
				$0<K_{*}(n)\le K^{*}(n)<\infty$ such that
				\begin{equation}\label{eq:norm_equivalence}
					\frac{1}{K_{*}}\,\|\Phi\|_{H^{1}(\Sigma)}^{2}
					\le\|\nabla(\mathcal T\Phi)\|_{L^{2}(\R^{n}_{+})}^{2}
					\le K^{*}\,\|\Phi\|_{H^{1}(\Sigma)}^{2}
					\end{equation}
for every  $\Phi\in H^{1}(\Sigma)$.
			\end{enumerate}
		\end{Lemma}
		
		\begin{proof}
			Write $\psi=\mathcal T\Phi$ for brevity. From the
			expression $V_{1}=C_{n}(\rho/2)^{(n-2)/2}$ and the identity
			$C_{n}^{2^{*}-2}=C_{n}^{4/(n-2)}=n(n-2)$, one obtains
			\begin{equation}\label{eq:V1_rho_identities}
				V_{1}^{\,2^{*}-2}=\tfrac{n(n-2)}{4}\rho^{2},\qquad
				V_{1}^{\,2^{\sharp}-2}=\tfrac{\sqrt{n(n-2)}}{2}\rho.
			\end{equation}
			Han--Li's conformal transfer identity
			\cite[(3.3)]{HanLi2000} applied to $\psi=\mathcal T\Phi$ gives
			\begin{equation}\label{eq:grad_transfer_lemma}
				\int_{\R^{n}_{+}}\!|\nabla\psi|^{2}\,\d z
				=\widetilde C_{n}^{2}\!\left[\int_{\Sigma}\!|\nabla_{g}\Phi|^{2}\,\d\sigma_{g}
				+\tfrac{n(n-2)}{4}\!\int_{\Sigma}\!\Phi^{2}\,\d\sigma_{g}
				-\tfrac{n-2}{2}T_{c}\!\int_{\partial\Sigma}\!\Phi^{2}\,\d\tau\right]\!,
			\end{equation}
			and $\d\sigma_{g}=\rho^{n}\,\d z$,
			$\d\tau=\rho^{n-1}\,\d z'|_{z_{n}=0}$. Combined with
			\eqref{eq:V1_rho_identities}, yield
			\begin{equation}\label{eq:L2_transfer_lemma}
				\begin{aligned}
					\int_{\R^{n}_{+}}\!V_{1}^{2^{*}-2}\psi^{2}\,\d z
					&=\widetilde C_{n}^{2}\,\tfrac{n(n-2)}{4}\!\int_{\Sigma}\!\Phi^{2}\,\d\sigma_{g},\\
					\int_{\partial\R^{n}_{+}}\!V_{1}^{2^{\sharp}-2}\psi^{2}\,\d\sigma
					&=\widetilde C_{n}^{2}\,\tfrac{\sqrt{n(n-2)}}{2}\!\int_{\partial\Sigma}\!\Phi^{2}\,\d\tau.
				\end{aligned}
			\end{equation}
			Substituting \eqref{eq:grad_transfer_lemma}--\eqref{eq:L2_transfer_lemma}
			into definition~\eqref{eq:Q_infty} of $Q_{\infty}(\psi)$ gives
			\begin{align*}
				Q_{\infty}(\psi)
				&=\widetilde C_{n}^{2}\!\left[\int_{\Sigma}\!|\nabla_{g}\Phi|^{2}\,\d\sigma_{g}
				+\tfrac{n(n-2)}{4}\!\int_{\Sigma}\!\Phi^{2}\,\d\sigma_{g}
				-\tfrac{n-2}{2}T_{c}\!\int_{\partial\Sigma}\!\Phi^{2}\,\d\tau\right]\\
				&\quad
				-(2^{*}-1)\widetilde C_{n}^{2}\,\tfrac{n(n-2)}{4}\!\int_{\Sigma}\!\Phi^{2}\,\d\sigma_{g}
				-(2^{\sharp}-1)\widetilde C_{n}^{2}\,\tfrac{\sqrt{n(n-2)}}{2}\!\int_{\partial\Sigma}\!\Phi^{2}\,\d\tau.
			\end{align*}
			By \eqref{eq:Q2_def_lemma}, we have
			$Q_{\infty}(\psi)=2\widetilde C_{n}^{2}Q_{2}(\Phi,\Phi)$.
			
			As for (ii), it follows from Han and Li~\cite[Proposition~3.2]{HanLi2000} that the kernel of $Q_{2}$ is given by
			\[
			\mathrm{Ker}\, Q_{2} = \mathrm{span}\{1,\zeta_{1},\dots,\zeta_{n}\}.
			\]
			Substituting $\Phi=1$ in~\eqref{eq:adjusted_transform} gives
			$\mathcal T 1=\widetilde C_{n}\rho^{(n-2)/2}=V_{1}$; substituting
			$\Phi=\zeta_{l}|_{\Sigma}$ gives $\mathcal T(\zeta_{l}|_{\Sigma})=J_{l}$
			by direct comparison with the explicit formula in Lemma~\ref{lem:kernel_Ji}.
			
			Finally, we prove item (iii). Since $-\tfrac{n-2}{2}T_{c}=\tfrac{\sqrt{n(n-2)}}{2}>0$,
			\eqref{eq:grad_transfer_lemma} gives
			\begin{equation*}
				\begin{aligned}
				\|\nabla\psi\|_{2,\R^{n}_{+}}^{2}
				&=\widetilde C_{n}^{2}\!\left[\int_{\Sigma}|\nabla_{g}\Phi|^{2}\,\d\sigma_{g}
				+\tfrac{n(n-2)}{4}\!\int_{\Sigma}\Phi^{2}\,\d\sigma_{g}
				+\tfrac{\sqrt{n(n-2)}}{2}\!\int_{\partial\Sigma}\Phi^{2}\,\d\tau\right]\\
                &\le K^*\|\Phi\|_{H^1(\Sigma)}^2,
			\end{aligned} 
		\end{equation*}
	where
	\[
	K^{*}:=\widetilde C_{n}^{2}\left(\max\{1,n(n-2)/4\}
	+\frac{\sqrt{n(n-2)}}{2}C_{\mathrm{tr}}\right),
	\]
	and $C_{\mathrm{tr}}$ denotes the constant appearing in the trace inequality
	\[
	\|\Phi\|_{L^{2}(\partial\Sigma)}^{2}
	\le C_{\mathrm{tr}}\|\Phi\|_{H^{1}(\Sigma)}^{2}.
	\] For the lower bound, the Sobolev embedding
			and H\"older's inequality give
			\begin{equation*}
				\int_{\R^{n}_{+}}V_{1}^{2^{*}-2}\psi^{2}\,\d z
				\le\|V_{1}^{2^{*}-2}\|_{L^{n/2}(\R^{n}_{+})}\,\|\psi\|_{L^{2^{*}}(\R^{n}_{+})}^{2}
				\le C_{S}\|V_{1}^{2^{*}-2}\|_{L^{n/2}(\R^{n}_{+})}\,\|\nabla\psi\|_{L^{2}(\R^{n}_{+})}^{2}.
			\end{equation*}
			From \eqref{eq:grad_transfer_lemma}, we have
			\begin{equation*}
				\|\Phi\|_{H^1(\Sigma)}^2=\widetilde C_{n}^{-2}\left[\|\nabla \psi\|_{L^2(\R_+^n)}-\tfrac{n(n-2)}{4}\!\int_{\Sigma}\!\Phi^{2}\,\d\sigma_{g}+\tfrac{\sqrt{n(n-2)}}{2}\!\int_{\partial\Sigma}\!\Phi^{2}\,\d\tau\right].
			\end{equation*}
		Using	\eqref{eq:L2_transfer_lemma} to express
			 $\int_{\Sigma}\!\Phi^{2}\,\d\sigma_{g}$, we conclude that $	\|\Phi\|_{H^1(\Sigma)}^2\le K_*\|\nabla \psi\|_{L^2(\R_+^n)}$
			with $K_{*}$ depending only on $n$.
		\end{proof}
		
		\begin{Lemma}\label{lem:half_space_coercivity}
			There exists $c_{0}=c_{0}(n)>0$ such that
			\begin{equation}\label{eq:spectral_gap_final}
				Q_{\infty}(\psi)\;\ge\;c_{0}\,\|\nabla\psi\|_{2,\R_+^n}^{2}
				\qquad\text{for every }\psi\in\mathcal{V}^{\perp}.
			\end{equation}
		\end{Lemma}
		
		\begin{proof}
			Let $\Phi\in H^{1}(\Sigma)$, and define $\psi$ by~\eqref{eq:adjusted_transform}, namely
			\[
			\psi=\widetilde C_{n}\rho^{(n-2)/2}\Phi.
			\]
			By Lemma~\ref{lem:conformal_equiv}\,(ii) and the norm equivalence~\eqref{eq:norm_equivalence}, the orthogonality condition
			$\psi\in\mathcal V^{\perp}$ in $D^{1,2}(\mathbb{R}^{n}_{+})$
			is equivalent to the condition that $\Phi$ is orthogonal to $\mathrm{Ker}\,Q_{2}$ in $H^{1}(\Sigma)$.
			Consequently, we may apply Han--Li~\cite[Proposition~3.4]{HanLi2000}, which yields the existence of a constant $\lambda_{*}=\lambda_{*}(n)>0$ such that
			\[
			Q_{2}(\Phi,\Phi)\ge \lambda_{*}\|\Phi\|_{H^{1}(\Sigma)}^{2}
			\qquad \text{for all } \Phi\in(\mathrm{Ker}\,Q_{2})^{\perp}.
			\]
			Combining the identity $Q_{\infty}(\psi)=2\widetilde C_{n}^{2}Q_{2}(\Phi,\Phi)$ with the estimate
			\[
			\|\nabla\psi\|_{L^{2}(\mathbb{R}^{n}_{+})}^{2}
			\le K^{*}\|\Phi\|_{H^{1}(\Sigma)}^{2},
			\]
			which follows from~\eqref{eq:norm_equivalence}, we deduce that
			\begin{equation*}
				Q_{\infty}(\psi)
				\ge 2\widetilde C_{n}^{2}Q_{2}(\Phi,\Phi)
				\ge 2\widetilde C_{n}^{2}\lambda_{*}\|\Phi\|_{H^{1}(\Sigma)}^{2}
				\ge \frac{2\widetilde C_{n}^{2}\lambda_{*}}{K^{*}}\|\nabla\psi\|_{2,\mathbb{R}^{n}_{+}}^{2}.
			\end{equation*}
We can complete the proof by setting $c_{0}:=2\widetilde C_{n}^{2}\lambda_{*}/K^{*}>0$.
		\end{proof}
		
		Let $\{(\va_k,\xi_k,a_{\va_k})\}$ satisfy $\va_k\to 0$, $a_{\va_k}\to\infty$
		with $\va_k a_{\va_k}\to 0$. For $\phi\in H^{1}(\Omega)$, define
		\begin{equation}\label{eq:Q_eps}
			\begin{aligned}
				Q_{\va_k}(\phi)
				&:=\int_{\Omega}\!|\nabla\phi|^{2}\,\d x
				+a_{\va_k}\!\int_{\Omega}\!V_{\va_k,\xi_k}^{\,2^{\sharp}-2}\phi^{2}\,\d x\\
				&\quad
				-(2^{*}-1)\!\int_{\Omega}\!V_{\va_k,\xi_k}^{\,2^{*}-2}\phi^{2}\,\d x
				-(2^{\sharp}-1)\!\int_{\partial\Omega}\!V_{\va_k,\xi_k}^{\,2^{\sharp}-2}\phi^{2}\,\d\sigma.
			\end{aligned}
		\end{equation}
		
		\begin{Lemma}\label{lem:spectral_gap}
			There exist $\gamma>0$ and $k_{0}\in\N$ such that, for every $k\ge k_{0}$,
			\begin{equation}\label{eq:coercivity}
				Q_{\varepsilon_{k}}(\phi)\ge\gamma\,\|\nabla\phi\|_{2,\Omega}^{2}
			\end{equation}
			where
			$\phi\in H^1(\Omega)$ orthogonal to $T_{1,V_{\varepsilon_{k},\xi_{k}}}(\mathcal M)$.
		\end{Lemma}
		
		\begin{proof}
		Arguing by contradiction, we suppose that there exist sequences
		$\varepsilon_k\to 0^{+}$, $\xi_k\to P_0\in\partial\Omega$, and
		$\phi_k\in H^1(\Omega)$ such that
			\begin{equation}\label{eq:H_assumptions}
				\|\nabla\phi_{k}\|_{2,\Omega}=1,
				\quad
				\phi_{k}\perp T_{1,V_{\varepsilon_{k},\xi_{k}}}(\mathcal M),
				\quad
				Q_{\varepsilon_{k}}(\phi_{k})\to 0.
			\end{equation}
			Define the rescaled function
			\begin{equation}\label{eq:blow_up_def}
				\tilde{\phi}_k(z)
				:= \varepsilon_k^{(n-2)/2}\phi_k\!\left(\varepsilon_k T_{\xi_k}z\right),
				\quad z \in \Omega_k:=\frac{T^{-1}_{\xi_k}(\Omega)}{\varepsilon_k}.
			\end{equation}
			Then $\|\nabla\widetilde\phi_{k}\|_{2,\Omega_{k}}=1$ by
			scale invariance, and $\Omega_{k}\nearrow\R^{n}_{+}$ as $k\to\infty$.
			By Definition~\ref{def:bubble} and~\eqref{equa2.9},
			$\varepsilon_{k}^{(n-2)/2}V_{\varepsilon_{k},\xi_{k}}(\varepsilon_{k}T_{\xi_{k}}z)=V_{1}(z)$.
			So the change of variables
			$\d x=\varepsilon_{k}^{n}\,\d z$,
			$\d\sigma_{x}=\varepsilon_{k}^{n-1}\,\d\sigma_{z}$, and
			$\phi_{k}^{2}(x)=\varepsilon_{k}^{-(n-2)}\widetilde\phi_{k}^{2}(z)$ yields
			\begin{equation}\label{eq:scaling_identities}
				\int_\Omega V_{\varepsilon_k,\xi_k}^{2^*-2}\phi_k^2\,\d x
				= \int_{\Omega_k}V_1^{2^*-2}\tilde{\phi}_k^2\,\d z,
				\qquad
				\int_{\partial\Omega}V_{\varepsilon_k,\xi_k}^{2^\sharp-2}\phi_k^2\,\d\sigma
				= \int_{\partial\Omega_k}V_1^{2^\sharp-2}\tilde{\phi}_k^2\,\d\sigma_z,
			\end{equation}
			while the perturbation term satisfies
			\begin{equation}\label{eq:perturb_est}
				\left|a_{\varepsilon_k}\int_\Omega V_{\varepsilon_k,\xi_k}^{2^\sharp-2}\phi_k^2\,\d x\right|
				= a_{\varepsilon_k}\varepsilon_k\int_{\Omega_k}V_1^{2^\sharp-2}\tilde{\phi}_k^2\,\d z
				\le C(\varepsilon_k a_{\varepsilon_k})\|\nabla\tilde{\phi}_k\|^2 \to  0,
			\end{equation}
			since $\varepsilon_k a_{\varepsilon_k}\to 0$. Combining these identities gives
			\begin{equation}\label{eq:scaling_identity}
				Q_{\varepsilon_k}(\phi_k) = Q_{\Omega_{k}}(\tilde{\phi}_k) + o(1),
			\end{equation}
			where
			\begin{equation}\label{eq:Q_tilde}
				Q_{\Omega_{k}}(\tilde{\phi}_k)
				:= \int_{\Omega_k}|\nabla\tilde{\phi}_k|^2\,\d z
				-(2^*-1)\int_{\Omega_k}V_1^{2^*-2}\tilde{\phi}_k^2\,\d z
				-(2^\sharp-1)\int_{\partial\Omega_k}V_1^{2^\sharp-2}\tilde{\phi}_k^2\,\d\sigma_{z}.
			\end{equation}
		Up to a subsequence, there exists $\widetilde\phi \in D^{1,2}(\mathbb{R}^{n}_{+})$ such that
		$\widetilde\phi_{k} \rightharpoonup \widetilde\phi$ weakly in $D^{1,2}(\mathbb{R}^{n}_{+})$
		and $\widetilde\phi_{k} \to \widetilde\phi$ strongly in $L^{2}_{\mathrm{loc}}(\overline{\mathbb{R}^{n}_{+}})$.
			By weak lower semicontinuity, for any compact set $K \subset \mathbb{R}^n_+$
			there exists $k_0(K)$ such that $K \subset \Omega_k$ for all $k \geq k_0(K)$, and hence
			\begin{equation}\label{eq:grad_lsc}
				\int_K |\nabla\tilde{\phi}|^2
				\leq \liminf_{k\to\infty}\int_K |\nabla\tilde{\phi}_k|^2
				\leq \liminf_{k\to\infty}\int_{\Omega_k}|\nabla\tilde{\phi}_k|^2.
			\end{equation}
			Taking the supremum over compact sets $K \nearrow \mathbb{R}^n_+$ and
			combining with~\eqref{eq:scaling_identity},
			\begin{equation}\label{eq:Q_infty_leq_0}
				Q_\infty(\tilde{\phi})
				\leq \liminf_{k\to\infty}Q_{\Omega_k}(\tilde{\phi}_k)
				= \liminf_{k\to\infty}Q_{\varepsilon_k}(\phi_k) = 0.
			\end{equation}
			It follows from the orthogonality condition~\eqref{eq:ortho_def} and the assumption
			$\phi_{k}\perp T_{1,V_{\varepsilon_{k},\xi_{k}}}(\mathcal M)$ that
			\begin{equation}\label{eq:ortho_Ji_pre1}
				\int_{\Omega}\nabla\phi_{k}\cdot\nabla V_{\varepsilon_{k},\xi_{k}}\,\d x=0,
				\quad
				\int_{\Omega}\nabla\phi_{k}\cdot\nabla\bigl(\partial_{\varepsilon}V_{\varepsilon_{k},\xi_{k}}\bigr)\,\d x=0,
			\end{equation}
			and
			\begin{equation}\label{eq:ortho_Ji_pre2}
				\int_{\Omega}\nabla\phi_{k}\cdot\nabla\bigl(\partial_{\tau_{i}}V_{\varepsilon_{k},\xi_{k}}\bigr)\,\d x=0
			\end{equation}
			for $1\le i\le n-1$. Under the change of variables $z=T_{\xi_k}^{-1}(x)/\varepsilon_k$,
			the bubble $V_{\varepsilon_k,\xi_k}$ rescales to $V_1$, while
			$\partial_{\varepsilon}V_{\varepsilon_k,\xi_k}$ and
			$\partial_{\tau_i}V_{\varepsilon_k,\xi_k}$ ($1\le i\le n-1$)
			rescale to $J_n$ and $J_i$ respectively.
			Hence \eqref{eq:ortho_Ji_pre1} and~\eqref{eq:ortho_Ji_pre2}
			transform to
			\begin{equation}\label{eq:ortho_Ji}
				\int_{\Omega_k}\nabla\widetilde\phi_k\cdot\nabla V_1\,\d z=0,
				\qquad
				\int_{\Omega_k}\nabla\widetilde\phi_k\cdot\nabla J_i\,\d z=0,
				\quad i=1,\dots,n.
			\end{equation}
			Since $\widetilde\phi_k\to\widetilde\phi$ in $C^1_{\rm loc}(\overline{\mathbb{R}^n_+})$ and
			$\Omega_k\nearrow\mathbb{R}^n_+$,
			as $k\to\infty$, we have
			\begin{equation}\label{eq:ortho_V1}
				\int_{\mathbb{R}^n_+}\nabla\widetilde\phi\cdot\nabla V_1\,\d z=0,
				\qquad
				\int_{\mathbb{R}^n_+}\nabla\widetilde\phi\cdot\nabla J_i\,\d z=0,
				\quad i=1,\dots,n,
			\end{equation}
			which gives that $\widetilde\phi\in\mathcal{V}^{\perp}$.
			Combining~\eqref{eq:ortho_V1} with~\eqref{eq:spectral_gap_final}
			and~\eqref{eq:Q_infty_leq_0}, we deduce that
			\begin{equation*}
				c_{0}\,\|\nabla\widetilde\phi\|_{2,\R^{n}_{+}}^{2}
				\le Q_{\infty}(\widetilde\phi)\le 0,
			\end{equation*}
			whence $\widetilde\phi\equiv 0$. Returning to~\eqref{eq:scaling_identities}, we obtain that
			\begin{equation}\label{eq:pot_vanish}
				\int_{\Omega_{k}}V_{1}^{2^{*}-2}\widetilde\phi_{k}^{2}\,\d z\to 0,
				\qquad
				\int_{\partial\Omega_{k}}V_{1}^{2^{\sharp}-2}\widetilde\phi_{k}^{2}\,\d \sigma_z\to 0.
			\end{equation}
			Substituting~\eqref{eq:pot_vanish} into~\eqref{eq:scaling_identity}, we deduce that
			$\|\nabla\phi_{k}\|_{2,\Omega}^{2}\to 0$.
			 It contradicts
			$\|\nabla\phi_{k}\|_{2,\Omega}=1$.
		\end{proof}
		
	\subsection{Proof of Theorem~\ref{thm:main}(i)--(ii) }
		
		For $\xi\in\partial\Omega$ and $\varepsilon>0$, $V_{\varepsilon,\xi}$
		denotes the boundary bubble given by Definition~\ref{def:bubble} and
		$\mathcal M$ is the bubble manifold given of Definition~\ref{def:M}.
		For $u\in H^{1}(\Omega)$, set
		\begin{equation}\label{eq:min_psi}
			d(u,\mathcal{M})
			:= \inf\bigl\{\|\nabla(u-\psi)\|_{2,\Omega}^2 : \psi\in\mathcal{M}\bigr\}.
		\end{equation}
		Before proving Theorem ~\ref{thm:main} (i)--(ii), we first present five important lemmas.
		
	\begin{Lemma}\label{lem:best_projection}
		Let $\delta>0$ and let $\{u_k\}\subset H^1(\Omega)$ satisfy
		\begin{align}
			u_k&\rightharpoonup0
			\quad\text{weakly in }H^1(\Omega)
			\quad\text{as }k\to\infty,\label{eq:bp1}\\
			d(u_k,\mathcal M)
			&\le
			\|\nabla u_k\|_{2,\Omega}^{2}-2\delta.
			\label{eq:bp2}
		\end{align}
		Then there exists $k_0\in\mathbb N$ such that, for every $k\ge k_0$,
		the distance $d(u_k,\mathcal M)$ is attained at some
		$b_kV_{\varepsilon_k,\xi_k}\in\mathcal M$.
		Writing
		\begin{equation}\label{eq:bp_decomp}
			u_k=b_kV_{\varepsilon_k,\xi_k}+\phi_k,
		\end{equation}
		we have
		\[
		\phi_k\perp
		T_{1,V_{\varepsilon_k,\xi_k}}(\mathcal M),
		\qquad
		k\ge k_0,
		\]
		where the orthogonality is understood in the sense of
		\eqref{eq:ortho_def}. Moreover, up to a subsequence,
		\[
		\varepsilon_k\to0,
		\qquad
		b_k\to b_0\neq0,
		\quad\text{as }k\to\infty.
		\]
	\end{Lemma}

		\begin{proof}
			For each fixed $k$, we consider the minimization problem over the scalar parameter:
			\[
			\widetilde F_k(\varepsilon,\xi)
			:=\min_{C\in\mathbb{R}}\|\nabla(u_k-CV_{\varepsilon,\xi})\|_{2,\Omega}^2
			=\|\nabla u_k\|_{2,\Omega}^2
			-\frac{\bigl(\int_\Omega\nabla u_k\cdot\nabla V_{\varepsilon,\xi}\,\d x\bigr)^2}
			{\|\nabla V_{\varepsilon,\xi}\|_{2,\Omega}^2}.
			\]
			The minimum is achieved at
			\[
			b(\varepsilon,\xi)
			:=\frac{\int_\Omega\nabla u_k\cdot\nabla V_{\varepsilon,\xi}\,\d x}
			{\|\nabla V_{\varepsilon,\xi}\|_{2,\Omega}^2}.
			\]
			Let $\{(\varepsilon_{k,\ell},\xi_{k,\ell})\}_{\ell\ge 1}$ be a minimizing
			sequence for $\widetilde F_k$ and set $b_{k,\ell}:=b(\varepsilon_{k,\ell},\xi_{k,\ell})$.
			By~\eqref{eq:bp2}, for all $\ell$ sufficiently large,
			\begin{equation}\label{eq:key}
				b_{k,\ell}^2\|\nabla V_{\varepsilon_{k,\ell},\xi_{k,\ell}}\|_{2,\Omega}^2
				-2b_{k,\ell}\int_\Omega\nabla u_k\cdot\nabla V_{\varepsilon_{k,\ell},\xi_{k,\ell}}\,\d x
				\;\le\;-\delta,
			\end{equation}
			which implies that $b_{k,\ell}\neq 0$.
			
			For $\xi\in\partial\Omega$, by definition~\eqref{equa2.9},
			$V_{\varepsilon,\xi}=\varepsilon^{-\frac{n-2}{2}}V_1(T^{-1}_\xi(x)/\varepsilon)$, where
			\[
			V_1(z)=C_n\bigl(1+|z'|^2+|z_n+x_n^0|^2\bigr)^{-\frac{n-2}{2}},
			\quad  x_n^0>0.
			\]
			Define
			$a_{k,\ell}:=\frac{|b_{k,\ell}|}{\varepsilon_{k,\ell}^{n/2}}$ and set
			\begin{align*}
				I_{k,\ell}
				&:=\int_\Omega
				\frac{\bigl|T^{-1}_{\xi_{k,\ell}}(x)/\varepsilon_{k,\ell}
					+(0,x_n^0)\bigr|^2}
				{\Bigl(1+\bigl|(T^{-1}_\xi(x))/\varepsilon_{k,\ell}+(0,x_n^0)\bigr|^2
					\Bigr)^n}
				\,\d x,\\
				J_{k,\ell}
				&:=\int_\Omega
				\frac{\nabla u_k(x)\cdot Q_{\xi_{k,\ell}}
					\!\Bigl(T^{-1}_{\xi_{k\ell}}(x)/\varepsilon_{k,\ell}+(0,x_n^0)\Bigr)}
				{\Bigl(1+\bigl|{T^{-1}_{\xi_{k,\ell}(x)}}/{\varepsilon_{k,\ell}}
					+(0,x_n^0)\bigr|^2\Bigr)^{n/2}}
				\,\d x,
			\end{align*}
			where $Q_{\xi_{k,\ell}}=(DT_{\xi_{k,\ell}}^{-1})^{T}$.
			A direct calculation, using
			$|\nabla_x V_{\varepsilon,\xi}|^2=\varepsilon^{-n}\left|\nabla V_1\left({T^{-1}_\xi(x)}/{\varepsilon}\right)\right|^2$,
			rewrites~\eqref{eq:key} as
			\begin{equation}\label{eq:star}
				(n-2)^2C_n^2\,a_{k,\ell}^2\,I_{k,\ell}
				\;\mp\;
				2(n-2)C_n\,a_{k,\ell}\,J_{k,\ell}
				\;\le\;-\delta,
			\end{equation}
			where the sign $\mp$ matches the sign of $b_{k,\ell}$.

			Suppose, for a subsequence, $\varepsilon_{k,\ell}\to+\infty$ as $\ell\to\infty$.
			For every $x\in\Omega$,  we have $\frac{T^{-1}_{\xi_{k,\ell}}(x)}{\varepsilon_{k,\ell}}\to 0$ as $\ell\to\infty$,
			so
			\begin{equation*}
				 I_{k,\ell}\to
			D:=	\frac{|x_n^0|^2}{(1+|x_n^0|^2)^n}|\Omega|>0\quad \text{as }  \ell\to\infty.
			\end{equation*}
			Similarly,
			\begin{equation*}
				J_{k,\ell}\to E_k:=(n-2)C_n\int_\Omega\frac{\nabla u_k\cdot Q_{\xi_\infty}(0,x_n^0)}{(1+|x_n^0|^2)^{n/2}}\,\d x<\infty\quad \text{as }  \ell\to\infty.
			\end{equation*}
			Letting $a_k:=\lim\limits_{\ell\to\infty} a_{k,\ell}$,
			inequality~\eqref{eq:star} gives
			\begin{equation}\label{eq:dagger}
				(n-2)^2C_n^2\,a_k^2\,D \;\mp\; 2(n-2)C_n\,a_k\,E_k\;\le\;-\delta.
			\end{equation}
			Thus $\{a_k\}$ is bounded and $a_k\to a_0\in[0,\infty)$.
			Letting $k\to\infty$ and using $u_k\rightharpoonup 0$
			in $H^1(\Omega)$, we obtain $E_k\to 0$, substituting this into~\eqref{eq:dagger},  we obtain
			\begin{equation*}
				(n-2)^2C_n^2\,a_0^2\,D\le-\delta<0,
			\end{equation*}
			which is impossible. Therefore, $\lim\limits_{\ell\to\infty}\varepsilon_{k,\ell}=\varepsilon_k<\infty$.
			
			Suppose, for a subsequence, $\varepsilon_{k,\ell}\to 0$ as $\ell\to\infty$. Under the change of variables
			\[
			z=\frac{T^{-1}_{\xi_{k,\ell}}(x)}{\varepsilon_{k,\ell}},
			\]
			we obtain
			\[
			I_{k,\ell}
			=\int_{\Omega_{k,\ell}}
			\frac{|z'|^2+|z_n+\varepsilon_{k,\ell}x_n^0|^2}
			{\left(1+|z'|^2+|z_n+\varepsilon_{k,\ell}x_n^0|^2\right)^n}\,\d z,
			\]
			where
			\[
			\Omega_{k,\ell}:=\frac{T^{-1}_{\xi_{k,\ell}}(\Omega)}{\varepsilon_{k,\ell}}.
			\]
			As $\varepsilon_{k,\ell}\to 0$, $\Omega_{k,\ell}\nearrow\mathbb{R}^n_+$
			and $\varepsilon_{k,\ell}x_n^0\to 0$ as $\ell\to\infty$, by monotone convergence theorem we deduce
			\begin{equation*}
				I_{k,\ell}\to M:=\int_{\mathbb{R}^n_+}{|z|^2}/{(1+|z|^2)^n}\,\d z>0\quad \text{as }  \ell\to\infty.
			\end{equation*}
		As $\varepsilon_{k,\ell}\to 0$, we have $V_{\varepsilon_{k,\ell},\xi_{k,\ell}}\rightharpoonup 0$
			weakly in $H^1(\Omega)$ as $\ell\to\infty$,
			\begin{equation*}
				J_{k,\ell}=\int_\Omega\nabla u_k\cdot\nabla V_{\varepsilon_{k,\ell},\xi_{k,\ell}}\,\d x\to 0\quad \text{as }  \ell\to\infty.
			\end{equation*}
			Inequality~\eqref{eq:star} then gives
			$a_k^2 M\le-\delta<0$, which is a contradiction.
			Hence $\varepsilon_{k,\ell}\to\varepsilon_k>0$ as $\ell\to\infty$.
			
			Therefore, for each fixed $k$, the sequence
			$\{\varepsilon_{k,\ell}\}_{\ell\ge1}$ is bounded and bounded away from
			zero. Passing to a subsequence if necessary,
			$(\varepsilon_{k,\ell},\xi_{k,\ell})\to(\varepsilon_k,\xi_k)$ with
			$\varepsilon_k\in(0,\infty)$ and
			$\xi_k\in\partial\Omega$. By the continuity of
			$\widetilde F_k$ with respect to $(\varepsilon,\xi)$,
			\[
			d(u_k,\mathcal M)
			=
			\widetilde F_k(\varepsilon_k,\xi_k),
			\]
			and hence
			\[
			d(u_k,\mathcal M)
			=
			\|\nabla(u_k-b_kV_{\varepsilon_k,\xi_k})\|_{2,\Omega},
			\]
			where $b_k=b(\varepsilon_k,\xi_k)$.
			
		Moreover, since $b_kV_{\varepsilon_k,\xi_k}$ realizes the distance
		from $u_k$ to $\mathcal M$, the remainder
		\[
		\phi_k:=u_k-b_kV_{\varepsilon_k,\xi_k}
		\]
		lies in the normal space of $\mathcal M$ at
		$b_kV_{\varepsilon_k,\xi_k}$. Equivalently,
		\[
		\phi_k\perp T_{1,V_{\varepsilon_k,\xi_k}}(\mathcal M),
		\]
		where the orthogonality is understood in the sense of
		\eqref{eq:ortho_def}.
		
			From~\eqref{eq:bp2}, we deduce
			\begin{equation*}
				b_{k}^2\int_{\Omega}|\nabla V_{\varepsilon_{k},\xi_{k}}|^2\,\d x
				-2b_{k}\int_\Omega\nabla u_k\cdot\nabla V_{\varepsilon_{k},\xi_{k}}\,\d x
				\;\le\;-\delta.
			\end{equation*}
			Since $u_k\rightharpoonup0$ in $H^1(\Omega)$ as $k\to\infty$,
			it follows that $\varepsilon_k\to 0$ and $b_k\to b_0\neq 0$ as $k\to\infty$. 
		\end{proof}

		\begin{Lemma}\label{lem:expansion}
			Let $n\ge5$, $\alpha\ge0$, and let
			$\xi\in\partial\Omega$ satisfy $H(\xi)>0$.
			Then, as $\varepsilon\to0^+$,
			\[
			I_\alpha(V_{\varepsilon,\xi})
			=
			A-L(n)H(\xi)\varepsilon
			+B(n)\alpha\varepsilon
			+O(\varepsilon^2)
			+O(\alpha\varepsilon^2),
			\]
			where
			\[
			L(n)
			=
			\frac{n^{n/2}(n-2)^{(2n-1)/2}\pi^{n/2}}
			{(n-3)2^{(3n-3)/2}(n-1)^{(n-3)/2}\Gamma(\frac n2+1)},
			\]
			and
			\[
			B(n)=
			\frac{(n(n-2))^{(n+1)/2}\pi^{n/2}}
			{2^{\,n+1}(n-1)\Gamma(\frac n2+1)}
			\left[
			B\!\left(\frac12,\frac{n-2}{2}\right)
			-
			B_{\frac{n}{2(n-1)}}
			\!\left(\frac12,\frac{n-2}{2}\right)
			\right].
			\]
			In particular, $L(n)>0$ and $B(n)>0$.
		\end{Lemma}
		\begin{proof}
			We expand $I_{\alpha}(V_{\varepsilon,\xi})$ according to its
			definition~\eqref{eq1.3}:
			\[
			\begin{aligned}
				I_{\alpha}(V_{\va,\xi})
				&= \underbrace{\frac{1}{2} \int_{\Omega} |\nabla V_{\va,\xi}|^2\,\d x}_{(E_1)}
				+ \underbrace{\frac{\lambda}{2} \int_{\Omega} V_{\va,\xi}^2\,\d x
					+\frac{\alpha}{2^\sharp}\int_{\Omega}V_{\va,\xi}^{2^\sharp}\,\d x}_{(E_2)}\\
				&\quad
				- \underbrace{\frac{1}{2^*} \int_{\Omega} V_{\va,\xi}^{2^*}\,\d x}_{(E_3)}
				- \underbrace{\frac{1}{2^\sharp} \int_{\partial\Omega} V_{\va,\xi}^{2^\sharp}\,\d\sigma}_{(E_4)}.
			\end{aligned}
			\]
			Without loss of generality, assume $\xi=0$. By~\eqref{eq:graph_setup},
			$\partial\Omega$ near the origin is the graph of $x_{n}=g(x')$ with
			$g(0)=0$, $\nabla g(0)=0$, and
			\begin{equation}\label{eq:phi_taylor}
				g(x')\;=\;\frac{1}{2}\sum_{i=1}^{n-1}h_{i}\,x_{i}^{2}
				\;+\;O(|x'|^{3}),
			\end{equation}
			where $h_{1},\ldots,h_{n-1}$ are the principal curvatures of $\partial\Omega$ at $0$.
			Recall
			\begin{equation}\label{eq:H_curvature}
				\sum_{i=1}^{n-1}h_{i}\;=\;(n-1)\,H(0).
			\end{equation}
			Set
			\begin{equation}\label{eq:Sigma_def}
				B(R)^{+}\;:=\;B(R)\cap\R_+^n,
				\qquad
				\Sigma\;:=\;\bigl\{(x',x_{n})\in B(R)\;:\;0<x_{n}<g(x')\bigr\}.
			\end{equation}
			By definition~\eqref{equa2.9},
			$V_\varepsilon:=V_{\va,0}(x)={\varepsilon^{-(n-2)/2}}V_1({x}/{\varepsilon})$
			with $V_{1}(y)=C_{n}(1+|y'|^{2}+(y_{n}+x_{n}^{0})^{2})^{-(n-2)/2}$,
			$x_{n}^{0}=\sqrt{n/(n-2)}$. We also introduce
			\begin{equation}\label{eq:beta_def}
				\beta_{n}^{(p)}\;:=\;
				\int_{\R^{n-1}}\!\!\frac{|y'|^{p}\,\d y'}
				{(1+|y'|^{2}+(x_{n}^{0})^{2})^{n}},
				\qquad p\in\{2,4\},
			\end{equation}
			which are finite for $n\ge 5$.
				
Let us address $(E_3)$ first.
			\[
			\int_{\Omega} V_{\va}^{2^*}\,\d x
			= \int_{B(R)^+}V_{\va}^{2^*}\,\d x
			- \int_{\Sigma} V_{\va}^{2^*}\,\d x
			+\int_{B(R)^c\cap \Omega}V_{\va}^{2^*}\,\d x.
			\]
			For the  integral over $B(R)^c\cap\Omega$, we have
		\begin{equation*}
				\int_{B(R)^c\cap\Omega}V_{\va}^{2^*}\,\d x\le\int_{B(R)^c}V_{\va}^{2^*}\,\d x=O(\varepsilon^n).
		\end{equation*}
			For the integral
			over $B(R)^+$, the rescaling $y=\frac{x}{\varepsilon}$ gives
			\begin{equation*}
				\int_{B(R)^+} V_{\va}^{2^*}\,\d x
			= \int_{\mathbb{R}_+^n} V_1(y)^{2^*}\,\d y + O(\va^n).
			\end{equation*}
			For the integral over $\Sigma$, the same rescaling and the Taylor
			expansion $g(\varepsilon y')/\varepsilon=(\varepsilon/2)\sum_i h_i y_i^2+O(\varepsilon^2|y'|^3)$ yield
			\[
			\begin{aligned}
				\int_{\Sigma} V_{\va}^{2^*}\,\d x
				&= C_n^{2^*} \int_{B'(\rho)}
				\left( \int_{0}^{g(x')}
				\frac{\va^n}{\left(\va^2|x'|^2+(x_n+\va x_n^0)^2\right)^n}
				\,\d x_n \right)\d x' \\
				&= C_n^{2^*} \int_{B'(R/\va)}
				\left( \int_{0}^{g(\va y')/\va}
				\frac{1}{\left(1+|y'|^2+(y_n+x_n^0)^2\right)^n}
				\,\d y_n \right)\d y'\\
				&=C_n^{2^*}\int_{B'(R/\va)}
				\frac{\frac{1}{2}\va\sum_{i} h_i y_i^2}
				{\left(1+|y'|^2+(x_n^0)^2\right)^n}
				\,\d y'+ O(\va^2)\\
				&= \va C_n^{2^*} \tfrac{H(0)}{2}\,\beta_n^{(2)} + O(\va^2).
			\end{aligned}
			\]
		Finally, combining these three contributions, we obtain that
			\begin{equation}\label{eq:E3_expansion}
				(E_3) = \frac{1}{2^*} \int_{\mathbb{R}_+^n} V_1^{2^*}\,\d x
				-\frac{n-2}{4n}\,C_n^{2^*}\beta_n^{(2)}\,H(0)\,\va + O(\va^2).
			\end{equation}
			
	As for $(E_4)$, we have
	\[
	\begin{aligned}
		\int_{\partial\Omega} V_{\va}^{2^\sharp}\d\sigma &= \int_{B(R) \cap \partial\Omega} V_{\va}^{2^\sharp} \d\sigma+ \int_{B(R)^c \cap \partial\Omega} V_{\va}^{2^\sharp}\d\sigma \\
		&= \int_{B'(R)} \frac{C_n^{2^\sharp} \va^{n-1}\sqrt{1+|\nabla g(x')|^2}}{\left(\va^2+ |x'|^2 + (g(x') + \va x_n^0)^2  \right)^{n-1}} \d x'+O(\va^{n-1}) \\
		&=C_n^{2^\sharp} \int_{B'(R)}  \frac{ \va^{n-1}(1+\frac{1}{2}|\nabla g(x')|^2+O(\nabla g|^4)|)}{\left(\va^2+ |x'|^2 + (g(x') + \va x_n^0)^2  \right)^{n-1}} \d x'+O(\va^{n-1})\\
		&=C_n^{2^\sharp} \int_{B'(\frac{R}{\va})} \frac{1+O(\va^2|y'|^2)}{\left(1+ y'|^2 + \left( \frac{g(\va y')}{\va} + x_n^0 \right)^2 \right)^{n-1}} \d y'+O(\va^{n-1})\\
		&= \int_{B'(\frac{R}{\va})} \frac{C_n^{2^\sharp}}{\left(1+ |y'|^2 + (x_n^0)^2  \right)^{n-1}} \d y'\\
		&\quad -\va(n-1)x_n^0C_n^{2^\sharp} \int_{B'(\frac{\rho}{\va})} \frac{\sum h_i y_i^2}{\left( 1+|y'|^2 + (x_n^0)^2  \right)^n} \d y'  + O\left( \va^2 \right) \\
		&= \int_{\partial\mathbb{R}_+^n} V_1(y', 0)^{2^\sharp} \d y'- \va(n-1)x_n^0C_n^{2^\sharp} \beta_n^{(2)} H(0) + O\left( \va^2 \right).
	\end{aligned}
	\]
			Hence,
			\begin{equation}\label{eq:E4_expansion}
				(E_4) = \frac{1}{2^\sharp} \int_{\partial\mathbb{R}_+^n} V_1^{2^\sharp}\,\d x'
				- \frac{n-2}{2}\,x_n^0\,C_n^{2^\sharp}\,\beta_n^{(2)}\,H(0)\,\va + O(\va^2).
			\end{equation}

			Now we are going to	estimate $(E_1)$.  A straightforward calculation gives
			 \[
			 \int_{\Omega} |\nabla V_{\va}|^2 \d x= \int_{\mathbb{R}_+^n} |\nabla V_1|^2 \d x- \int_{\Sigma} |\nabla V_{\va}|^2 \d x+ O\left( \va^2 \right).
			 \]
			 By the same reasoning as in the estimate of $(E_3)$, 	we have  
			 \[
			 \begin{aligned}
			 	\int_{\Sigma} |\nabla V_{\va}|^2 \d x&= C_n^2 (n-2)^2 \int_{B'(\frac{R}{\va})} \int_{0}^{\frac{g(\va y')}{\va}} \frac{|y'|^2 + (y_n + x_n^0)^2}{\left(1+ |y'|^2 + (y_n + x_n^0)^2 \right)^n} \d y_n \d y' \\
			 	&= C_n^2 (n-2)^2 \va \int_{B'(\frac{R}{\va})} \frac{(|y'|^2 + (x_n^0)^2) \frac{1}{2}\sum_{i} h_i y_i^2}{\left(1+ |y'|^2 +(x_n^0)^2  \right)^n} \d y' + O(\va^2) \\
			 	&= \frac{C_n^2 (n-2)^2}{2} \va H(0) \int_{B'(\frac{R}{\va})} \frac{(|y'|^2 + (x_n^0)^2) |y'|^2}{\left(1+ |y'|^2 + (x_n^0)^2  \right)^n} \d y' + O(\va^2) \\
			 	&= \frac{C_n^2 (n-2)^2}{2} \va H(0) \left( (x_n^0)^2 \beta_n^{(2)} + \int_{B'(\frac{R}{\varepsilon})} \frac{|y'|^4}{\left(1+ |y'|^2 + (x_n^0)^2  \right)^n} \d y' \right) + O(\va^2) \\
			 	&=\frac{1}{2} C_n^2 (n-2)^2 H(0) \left( (x_n^0)^2 \beta_n^{(2)} + \beta_n^{(4)} \right) \va + O(\va^2).
			 \end{aligned}
			 \]
			 Hence, for $n\ge5$, we deduce that
			 \begin{equation}\label{eq:E1_expansion}
			 	(E_1) = \int_{\mathbb{R}_+^n} |\nabla V_1|^2\d x - \frac{(n-2)^2}{4} C_n^2 H(0) \left( (x_n^0)^2 \beta_n^{(2)}+ \beta_n^{(4)} \right) \va + O(\va^2).
			 \end{equation}

			Finally, we  estimate of $(E_2)$. A direct computation gives
			 \begin{equation*}
			 	\begin{aligned}
			 		\int_{\Omega} V_{\va}^2\d x &=\int_{B(R)^+} V_{\va}^2\d x-\va^2 \int_{B'(\frac{R}{\va})} \int_{0}^{\frac{g(\va y')}{\va}} \frac{C_n^2}{\left(1+ |y'|^2 + (y_n + x_n^0)^2  \right)^{n-2}} \d y_n \d y'+O(\va^2)\\
			 		&= \va^2\int_{\R_+^n}V_1^{2}\d x+O(\va^2),\\
			 	\end{aligned}
			 \end{equation*}
			 so that
			 \begin{equation*}
			 	\frac{\lambda}{2}	\int_{\Omega} V_{\va}^2\d x=\frac{\lambda\va^2}{2} \int_{\R_+^n}V_1^{2}\d x+O(\va^2).
			 \end{equation*}
			 Similarly,
			 \begin{equation*}
			 	\begin{aligned}
			 		\frac{\alpha}{2^\sharp}\int_{\Omega}  V_{\va}^{2^\sharp}\d x&= \frac{\alpha\va}{2^\sharp} \int_{B(\frac{R}{\va})^+} \frac{C_n^{2^\sharp}  }{\left(1+ |y'|^2 + \left( \frac{\varphi(\va y')}{\va} + x_n^0 \right)^2 \right)^{n-1}} \d y\\
			 		&\quad-\frac{\alpha}{2^\sharp}\frac{C_n^{2^\sharp}}{2}H(0)\va^2\int_{B'(\frac{R}{\va})} \frac{|y'|^2 }{\left(1+ |\bar{y}|^2 + |x_n^0|^2 \right)^{n-1}} \d y'+O(\alpha\va^2)\\
			 		&= \frac{1}{2^\sharp}\alpha\va \int_{\R_+^n}V_1^{2^\sharp}\d x+o(\alpha\va).
			 	\end{aligned}
			 \end{equation*}
			 Thus,
			 \begin{equation}\label{eq:E2_expansion}
			 	(E_2)=	\frac{\alpha}{2^\sharp}\int_{\Omega}  V_{\va}^{2^\sharp} \d x+\frac{\lambda}{2}	\int_{\Omega} V_{\va}^2\d x
			 	= \frac{\alpha\va}{2^\sharp}\int_{\R_+^n}V_1^{2^\sharp}\d x+O(\va^2)+o(\alpha\va).
			 \end{equation}
			Combining \eqref{eq:E1_expansion}--\eqref{eq:E4_expansion} and using
			 \[
			 \frac{1}{2}\int_{\R^{n}_{+}}|\nabla V_1|^2\d x-\frac{1}{2^*}\int_{\R^{n}_{+}}V_1^{2^*}\d x-\frac{1}{2^\sharp}\int_{\partial\R^{n}_{+}}V_1^{2^\sharp}\d x' = A,
			 \]
			 we obtain
			 \begin{equation}
			 	\label{eq:I_alpha_assembled}
			 	\begin{aligned}
			 		I_{\alpha}(V_{\varepsilon})
			 		&= A + B(n)\alpha\va - L(n)H(0)\va
			 		+O(\va^{2})+o(\alpha\va),
			 	\end{aligned}
			 \end{equation}
			 where
			 \begin{align}
			 	\label{eq:Ln_explicit}
			 	L(n)
			 	&:=\frac{(n-2)^{2}}{4}\,C_{n}^{2}\bigl((x_{n}^{0})^{2}\beta_{n}^{(2)}+\beta_{n}^{(4)}\bigr)
			 	-\frac{n-2}{4n}\,C_{n}^{2^*}\,\beta_{n}^{(2)}
			 	-\frac{n-2}{2}\,x_{n}^{0}\,C_{n}^{2^\sharp}\,\beta_{n}^{(2)},\\
			 	\label{eq:Bn_explicit}
			 	B(n)
			 	&:=\frac{1}{2^\sharp}\!\int_{\R_+^n}\!V_{1}^{2^\sharp}\d y= \frac{n-2}{2(n-1)}
			 	\int_{\R_+^n}
			 	\frac{C_n^{2^\sharp}}
			 	{\bigl(1+|y'|^2+(y_n+x_n^0)^2\bigr)^{n-1}}\,dy.
			 \end{align}
			 Substituting $C_{n}=(n(n-2))^{\frac{n-2}{4}}$ and $(x_{n}^{0})^{2}=\frac{n}{n-2}$ into  \eqref{eq:Ln_explicit},
			 we have
			 \begin{equation*}
			 	L(n)=\frac{\left(n(n-2)\right)^{\frac{n}{2}}}{4n}\big((n-2)\beta_n^{(4)}-2(n-1)\beta_n^{(2)}\big).
			 \end{equation*}
			 We 	now use the standard polar coordinate identity
			 \begin{equation}
			 	\label{eq:beta_polar}
			 	\beta_{n}^{(p)}
			 	\;=\;\omega_{n-2}\!\int_{0}^{\infty}\!\!\frac{r^{n-2+p}\d r}
			 	{(1+(x_{n}^{0})^{2}+r^{2})^{n}},
			 	\qquad p\in\{2,4\},
			 \end{equation}
			 where $\omega_{n-2}=\frac{2\pi^{(n-1)/2}}{\Gamma((n-1)/2)}$ is the
			 $(n-2)$-sphere measure.  By the Beta-function identity
			 \begin{equation}\label{eq:Beta}
			 	\int_{0}^{\infty}r^{s-1}(c^{2}+r^{2})^{-t}\d r
			 	=\tfrac{1}{2}c^{s-2t}B(\frac{s}{2},t-\frac{s}{2})
			 \end{equation}with
			 $c^{2}=1+(x_{n}^{0})^{2}=\tfrac{2(n-1)}{n-2}$, we have
			 \begin{equation*}
			 	\begin{aligned}
			 		 (n-2)\beta_n^{(4)}-2(n-1)\beta_n^{(2)}&=(n-2)\pi^{\frac{n-1}{2}}\left(\frac{n-2}{2(n-1)}\right)^{\frac{n-3}{2}}\left(\frac{n+1}{n-3}\right)\frac{\Gamma(\frac{n+1}{2})}{\Gamma(n)}\\
			 		&\quad-2(n-1)\pi^{\frac{n-1}{2}}\left(\frac{n-2}{2(n-1)}\right)^{\frac{n-1}{2}}\frac{\Gamma(\frac{n+1}{2})}{\Gamma(n)}\\
			 		 &=2\pi^{\frac{n-1}{2}}\left(\frac{n-2}{2(n-1)}\right)^{\frac{n-1}{2}}\frac{\Gamma(\frac{n+1}{2})}{\Gamma(n)}\frac{4(n-1)}{n-3}>0.
			 	\end{aligned}
			 \end{equation*}
			 Applying the Legendre duplication formula
			 $\Gamma(n)=\frac{2^{n-1}}{\sqrt{\pi}}\Gamma(\frac{n}{2})\Gamma(\frac{n+1}{2})$
			 together with $\Gamma(\frac{n}{2})=\frac{2}{n}\Gamma(\frac{n}{2}+1)$, so that
			 $\frac{\Gamma(\frac{n+1}{2})}{\Gamma(n)}=\frac{n\sqrt{\pi}}{2^{n}\Gamma(\frac{n}{2}+1)}$,
			 we obtain
			 \begin{equation}\label{eq:Ln_final}
			 	L(n)=\frac{n^{n/2}\,(n-2)^{(2n-1)/2}\,\pi^{n/2}}
			 	{(n-3)\,2^{(3n-3)/2}\,(n-1)^{(n-3)/2}\,\Gamma(\frac{n}{2}+1)}>0.
			 \end{equation}
             It remains to evaluate $B(n)$. By \eqref{eq:Bn_explicit} and
			 $C_n^{2^\sharp}=(n(n-2))^{\frac{n-1}{2}}$, we have
			 \begin{equation}\label{eq:Bn_start}
			 	B(n)=\frac{n-2}{2(n-1)}\,(n(n-2))^{\frac{n-1}{2}}
			 	\int_{\R^n_+}\frac{\d y}{\bigl(1+|y'|^2+(y_n+x_n^0)^2\bigr)^{n-1}} .
			 \end{equation}
			 We integrate first in the tangential variable $y'\in\R^{n-1}$ and then in
			 the normal variable $y_n>0$. For fixed $y_n$, put $s=y_n+x_n^0$, by the
			 Beta-function identity \eqref{eq:Beta}
			 with $t=n-1$ and $c^2=1+s^2$ gives
			 \begin{equation}\label{eq:Bn_tangential}
			 	\int_{\R^{n-1}}\frac{\d y'}{\bigl(1+s^2+|y'|^2\bigr)^{n-1}}
			 	=\pi^{(n-1)/2}\frac{\Gamma(\frac{n-1}{2})}{\Gamma(n-1)}\,
			 	(1+s^2)^{-(n-1)/2}.
			 \end{equation}
			 Integrating \eqref{eq:Bn_tangential} over $y_n\in(0,\infty)$ and setting
			 \begin{equation}\label{eq:In_def}
			 	I_n:=\int_{x_n^0}^{\infty}\frac{\d s}{(1+s^2)^{(n-1)/2}}
			 	=\int_{\theta_n}^{\pi/2}\cos^{\,n-3}\!\phi\,\d\phi,
			 	\qquad \cos^2\theta_n=\frac{n-2}{2(n-1)}
			 \end{equation}
			 the second equality follows from $s=\tan\phi$, we obtain
			 \begin{equation}\label{eq:Bn_radial}
			 	\int_{\R^n_+}\frac{\d y}{\bigl(1+|y'|^2+(y_n+x_n^0)^2\bigr)^{n-1}}
			 	=\pi^{(n-1)/2}\frac{\Gamma(\frac{n-1}{2})}{\Gamma(n-1)}\,I_n .
			 \end{equation}
			 Substituting \eqref{eq:Bn_radial} into \eqref{eq:Bn_start} and applying the
			 Legendre duplication formula
			 $\Gamma(n-1)=\frac{2^{n-2}}{\sqrt{\pi}}\Gamma(\frac{n-1}{2})\Gamma(\frac{n}{2})$
			 together with $\Gamma(\frac{n}{2})=\frac{2}{n}\Gamma(\frac{n}{2}+1)$, we obtain
			 \[
			 \frac{\pi^{(n-1)/2}\Gamma(\frac{n-1}{2})}{\Gamma(n-1)}
			 =
			 \frac{n\pi^{n/2}}
			 {2^{\,n-1}\Gamma(\frac n2+1)}.
			 \]
			 Hence
			 \begin{equation}\label{eq:Bn_final}
			 	B(n)
			 	=
			 	\frac{(n(n-2))^{(n+1)/2}\pi^{n/2}}
			 	{2^{\,n}(n-1)\Gamma(\frac n2+1)}
			 	\,I_n,
			 \end{equation}
			 where
			 \[
			 I_n
			 =
			 \int_{\sqrt{n/(n-2)}}^\infty
			 \frac{ds}{(1+s^2)^{(n-1)/2}}
			 =
			 \frac12
			 \left[
			 B\!\left(\frac12,\frac{n-2}{2}\right)
			 -
			 B_{\frac{n}{2(n-1)}}
			 \!\left(\frac12,\frac{n-2}{2}\right)
			 \right].
			 \]
			  Indeed, the identity for \(I_n\) follows from the change of variables
			 \(t=(1+s^2)^{-1}\) and the relation
			 \[
			 B_x(a,b)
			 =
			 B(a,b)-B_{1-x}(b,a).
			 \]
			  Since \(0<\theta_n<\frac{\pi}{2}\), the integrand in
			 \eqref{eq:In_def} is strictly positive, so \(I_n>0\).
			 Consequently, all factors in \eqref{eq:Bn_final} are positive for
			 \(n\ge5\), and therefore
			$B(n)>0$.
			 \end{proof}
		
		\begin{Lemma}[{\cite[Lemma~3.5]{APY}}]\label{lem:algebraic}
			Let $q>1$ and $L$ be a nonneg\-ative integer with $L\leq q$.
			Let $V$ and $\omega$ be measurable functions on $\Omega$ with
			$V\geq 0$, $V+\omega\geq 0$. Then
			\begin{align*}
				\int_{\Omega} (V+\omega)^q
				&= \sum_{i=0}^{L} \frac{q(q-1)\cdots(q-i+1)}{i!}
				\int_{\Omega} V^{q-i} \omega^i
				+ O\!\left( \int_{\Omega} V^{q-r} |\omega|^r + |\omega|^q \right),  \\
				\int_{\Omega} (V+\omega)^q \omega
				&= \sum_{i=0}^{L} \frac{q(q-1)\cdots(q-i+1)}{i!}
				\int_{\Omega} V^{q-i} \omega^{i+1}
				+ O\!\left( \int_{\Omega} V^{q-r} |\omega|^{r+1} + |\omega|^{q+1} \right),
			\end{align*}
			where $r = \min(L+1, q)$.
		\end{Lemma}

		\begin{Lemma}\label{lem:final_estimate}
			Assume that $n\ge 5$, $\alpha_{k}>0$, $\delta_{k}>0$, $P_{k}\in\partial\Omega$,
			$u_{k}\in H^{1}(\Omega)$ satisfy
$\alpha_{k}\to+\infty$, $\delta_{k}\to 0$, $P_{k}\to P_{0}$,
			$u_{k}\rightharpoonup 0$ in $H^{1}(\Omega)$, as $k\to\infty$, and
			\begin{equation}\label{eq:final_hyp}
				\lim_{k\to\infty}\|\nabla u_{k}-\nabla V_{\delta_{k},P_{k}}\|_{L^{2}(\Omega)}^{2}=0.
			\end{equation}
			Then there exist $\varepsilon_{k}>0$ and $\xi_{k}\in\partial\Omega$ such that,
			along a subsequence,
			$\varepsilon_{k}/\delta_{k}\to 1$, $\xi_{k}\to P_{0}$, as $k\to\infty$, and
			\begin{equation}\label{eq:decomp_u}
				u_k = b_kV_{\varepsilon_k,\xi_k}+\phi_{k}
			\end{equation}
			with $b_k\to 1$ and $\|\nabla\phi_k\|_{2,\Omega}\to 0$. Moreover,
			\begin{equation}\label{eq:final_estimate}
				I_{\alpha_{k}}(u_{k})
				\ge A-L(n)H(\xi_{k})\varepsilon_{k}+B(n)\alpha_{k}\varepsilon_{k}
				+O(\varepsilon_{k}^{2})
				+O(\alpha_{k}\varepsilon_{k}\|\phi_{k}\|)+O(\|\phi_{k}\|^{2}).
			\end{equation}
		\end{Lemma}
		
		\begin{proof}
			Since $\{u_k\}$ satisfies the hypotheses of
			Lemma~\ref{lem:best_projection}, the decomposition
			\eqref{eq:decomp_u} holds for suitable
			$b_k$, $\varepsilon_k$, $\xi_k\in\partial\Omega$, and
			$\phi_k\in H^1(\Omega)$.
			Moreover,
			\[
			d(u_k,\mathcal M)
			\le
			\|\nabla(u_k-V_{\delta_k,P_k})\|_{2,\Omega}
			\to0
			\qquad\text{as }k\to\infty,
			\]
			which implies
			\[
			\|\nabla\phi_k\|_{2,\Omega}
			\to0.
			\]
			Therefore, by the triangle inequality,
			\begin{equation}\label{eq:bubble_diff}
				\|\nabla(V_{\delta_k,P_k}-b_kV_{\varepsilon_k,\xi_k})\|_{2,\Omega}
				\to0
				\qquad\text{as }k\to\infty.
			\end{equation}
			In particular,
			\begin{equation*}
				\lim_{k\to\infty}\|\nabla V_{\delta_k,P_k}\|_{2,\Omega}
			=\lim_{k\to\infty}\|\nabla V_{\varepsilon_k,\xi_k}\|_{2,\Omega}
			=\|\nabla V_1\|_{2,\R_+^n}.
			\end{equation*}
				Dividing by the norm  in \eqref{eq:bubble_diff} implies $|1 - b_k| \to 0$ or $|1 + b_k| \to 0$ as $k\to\infty$. Since $u_k \ge 0$ and the bubbles are positive, we must have $b_k > 0$ for large $k$. Thus,
			\begin{equation}\label{eq:b_limit}
				\lim_{k\to\infty} b_k = 1.
			\end{equation}
			From~\eqref{eq:bubble_diff}--\eqref{eq:b_limit},
			\begin{equation}\label{eq:pure_bubble_conv}
				\lim_{k\to\infty}\|\nabla(V_{\delta_k,P_k}-V_{\varepsilon_k,\xi_k})\|_{2,\Omega}^2=0.
			\end{equation}
				Recall the classical Sobolev bubble
			\[
		U_{a,\lambda}(x)
			:=C_{n}\,\frac{\lambda^{\frac{n-2}{2}}}{\bigl(1+\lambda^{2}|x-a|^{2}\bigr)^{\frac{n-2}{2}}},
			\quad a\in\mathbb{R}^{n},\ \lambda>0.
			\]
			Setting $\widetilde P_{k}:=P_{k}-\delta_{k}\,x_{n}^{0}\,e_{n}(P_{k})$,
			since $\{e_{j}(P_{k})\}_{j=1}^{n}$ is an orthonormal frame we have
			\begin{equation*}
				|x-\widetilde P_{k}|^{2}
				=\sum_{j=1}^{n}\langle x-\widetilde P_{k},e_{j}(P_{k})\rangle^{2}
				=|y'|^{2}+\bigl(y_{n}+\delta_{k}\,x_{n}^{0}\bigr)^{2},
			\end{equation*}
			where $y_{j}:=\langle x-P_{k},e_{j}(P_{k})\rangle$.
			Substituting into definition~\eqref{equa2.9} of the bubble, we deduce
			\begin{equation}\label{eq:bdry_to_classical}
				V_{\delta_{k},P_{k}}(x)
				=C_{n}\,\frac{\delta_{k}^{-\frac{n-2}{2}}}
				{\left(1+\delta_{k}^{-2}|x-\widetilde P_{k}|^{2}\right)^{\frac{n-2}{2}}}
				=U_{\widetilde P_k,\delta_k^{-1}}(x).
			\end{equation}
			Similarly, $V_{\varepsilon_{k},\xi_{k}}(x)=U_{\widetilde\xi_{k},\varepsilon_{k}^{-1}}(x)$
			with $\widetilde\xi_{k}:=\xi_{k}-\varepsilon_{k}\,x_{n}^{0}\,e_{n}(\xi_{k})$.
			We now extract the parameter convergence from~\eqref{eq:pure_bubble_conv}
			via a rescaling argument.
			Define the rescaling operator $\mathcal{S}_{k}$ by
			\[
			(\mathcal{S}_{k}V)(z)
			:=\delta_{k}^{\frac{n-2}{2}}V\!\bigl(T_{P_{k}}(\delta_{k}z)\bigr),
			\qquad z\in\widehat\Omega_{k}
			:=\frac{T_{P_{k}}^{-1}(\Omega)}{\delta_{k}}\to\mathbb{R}^{n}_{+}.
			\]
			Since $T_{P_{k}}$ is an isometry and $\mathcal{S}_{k}$ preserves the
			$L^{2}$-norm of the gradient, a direct computation
			from~\eqref{eq:bdry_to_classical} gives
			\begin{equation}\label{eq:rescaled_bubbles}
				\mathcal{S}_{k}V_{\delta_{k},P_{k}}=V_{1}
				=U_{-x_{n}^{0}e_{n},1},
				\qquad
				\mathcal{S}_{k}V_{\varepsilon_{k},\xi_{k}}
				=U_{\hat a_{k},\hat\lambda_{k}},
			\end{equation}
			where $\hat\lambda_{k}:=\frac{\delta_{k}}{\varepsilon_{k}}$ and
			$\hat a_{k}:=\frac{T_{P_{k}}^{-1}(\widetilde\xi_{k})}{\delta_{k}}$.
			Applying $\mathcal{S}_{k}$ to~\eqref{eq:pure_bubble_conv} yields
			\begin{equation}\label{eq:rescaled_diff}
				\bigl\|\nabla\bigl(U_{-x_{n}^{0}e_{n},1}-U_{\hat a_{k},\hat\lambda_{k}}\bigr)
				\bigr\|^2_{2,\widehat\Omega_{k}}\to 0.
			\end{equation}
			Since $\widehat\Omega_{k}\nearrow\mathbb{R}^{n}_{+}$ and
			\eqref{eq:rescaled_diff} extends to all of $\mathbb{R}^{n}_{+}$,
			we conclude that
			\[
			\hat\lambda_k\to1,
			\qquad
			\hat a_k\to -x_n^0\mathbf e_n.
			\]
			By definition,
			\[
			\hat a_k=\frac{T_{P_k}^{-1}(\widetilde\xi_k)}{\delta_k},
			\qquad
			\widetilde\xi_k
			=\xi_k-\varepsilon_k x_n^0 e_n(\xi_k).
			\]
			Since
			\[
			T_{P_k}^{-1}(x)
			=
			R_{P_k}^{\mathsf T}(x-P_k),
			\qquad
			R_{P_k}^{\mathsf T}=(DT_{P_k})^{-1},
			\]
			we obtain
			\[
			\hat a_k
			=
			\frac{T_{P_k}^{-1}(\xi_k)}{\delta_k}
			-
			\frac{\varepsilon_k}{\delta_k}
			x_n^0
			R_{P_k}^{\mathsf T}e_n(\xi_k).
			\]
			Since $\hat a_k\to-x_n^0\mathbf e_n$, the sequence $\{\hat a_k\}$ is bounded.
			Moreover,
			\[
			|R_{P_k}^{\mathsf T}e_n(\xi_k)|=1,
			\qquad
			\frac{\varepsilon_k}{\delta_k}\to1.
			\]
			Hence,
			\[
			\frac{T_{P_k}^{-1}(\xi_k)}{\delta_k}
			=
			\hat a_k
			+\frac{\varepsilon_k}{\delta_k}
			x_n^0R_{P_k}^{\mathsf T}e_n(\xi_k)
			\]
			is bounded. Since
			\[
			T_{P_k}^{-1}(\xi_k)
			=
			R_{P_k}^{\mathsf T}(\xi_k-P_k),
			\]
			where $R_{P_k}^{\mathsf T}$ is an isometry, it follows that
			\[
			|P_k-\xi_k|
			=
			O(\delta_k)\to0
			\qquad\text{as }k\to\infty.
			\]
			
			Using Hölder's inequality together with the Sobolev embedding, we obtain
			\[
			\int_{\Omega}V_{\varepsilon_k,\xi_k}\phi_k\,dx
			=
			\begin{cases}
				O\left(\varepsilon_k^{3/2}\|\phi_k\|\right), & n=5,\\[2mm]
				O\left(\varepsilon_k^{2}|\log\varepsilon_k|^{2/3}\|\phi_k\|\right), & n=6,\\[2mm]
				O\left(\varepsilon_k^{2}\|\phi_k\|\right), & n\ge7.
			\end{cases}
			\]
			Similarly, Hölder's inequality together with the trace Sobolev embedding gives
			\[
			\int_{\partial\Omega}
			V_{\varepsilon_k,\xi_k}^{2^\sharp-1}\phi_k\,d\sigma
			\le
			\|\phi_k\|_{2^\sharp,\partial\Omega}
			\left(
			\int_{\partial\Omega}
			V_{\varepsilon_k,\xi_k}^{\frac{2(n-1)}{n-2}}
			\,d\sigma
			\right)^{\!\frac{n}{2(n-1)}}
			=
			O(\varepsilon_k\|\phi_k\|).
			\]
			Applying Lemma~\ref{lem:algebraic} to the decomposition
			\[
			u_k=b_kV_{\varepsilon_k,\xi_k}+\phi_k,
			\]
			we obtain
			\[
			I_{\alpha_k}(u_k)
			=
			I_{\alpha_k}(b_kV_{\varepsilon_k,\xi_k})
			+O(\alpha_k\varepsilon_k\|\phi_k\|)
			+O(\|\phi_k\|^2).
			\]
			Since $b_k\to1$, Lemma~\ref{lem:expansion} then yields
			\[
			I_{\alpha_k}(u_k)
			\ge
			A
			-L(n)H(\xi_k)\varepsilon_k
			+B(n)\alpha_k\varepsilon_k
			+O(\varepsilon_k^2)
			+O(\alpha_k\varepsilon_k\|\phi_k\|)
			+O(\|\phi_k\|^2).
			\qedhere
			\]
		\end{proof}
		
We remark that, if in addition $u_k$ is a solution of \eqref{eq:main},
the remainder terms in \eqref{eq:final_estimate} can be improved, as
shown in the next lemma.
	
		\begin{Lemma}\label{lem:phi_size}
		Under the hypotheses of Lemma~\ref{lem:final_estimate},
			assume in addition that
			$u_k$ is a solution of problem \eqref{eq:main}. If
			$\alpha_k\to+\infty$ and $\varepsilon_k\to0$, then the remainder satisfies
			\begin{equation}\label{eq4.8}
				\|\phi_k\|^2=O(\alpha_k\varepsilon_k\|\phi_k\|).
			\end{equation}
		\end{Lemma}
		
		\begin{proof}
	By Lemma~\ref{lem:final_estimate}, we have the decomposition
	$u_k=b_kV_{\varepsilon_k,\xi_k}+\phi_k$, where
	$\phi_k\perp T_{1,V_{\varepsilon_k,\xi_k}}(\mathcal M)$.
	Since $u_k$ is a solution of \eqref{eq:main}, testing the equation
	against $\phi_k$ gives
		\begin{equation*}
		\begin{aligned}
			&\int_{\Omega}(|\nabla\phi_k|^2+\lambda\phi_k^2)\,\d x
			+\lambda b_k\!\int_{\Omega}V_{\varepsilon_k,\xi_k}\phi_k\,\d x
			+\alpha_k\!\int_{\Omega}(b_kV_{\varepsilon_k,\xi_k}+\phi_k)^{2^\sharp-1}\phi_k\,\d x\\
			&=\int_{\Omega}(b_kV_{\varepsilon_k,\xi_k}+\phi_k)^{2^*-1}\phi_k\,\d x
			+\int_{\partial\Omega}(b_kV_{\varepsilon_k,\xi_k}+\phi_k)^{2^\sharp-1}\phi_k\,\d\sigma,
		\end{aligned}
	\end{equation*}
The mixed gradient term on the left-hand side vanishes due to the
orthogonality condition
\[
\int_{\Omega}
\nabla V_{\varepsilon_k,\xi_k}\cdot\nabla\phi_k\,dx=0.
\]
	Applying Lemma~\ref{lem:algebraic} and using $b_k\to1$, we obtain
			\begin{equation*}
				\begin{aligned}
					&\int_{\Omega}(|\nabla\phi_k|^2+\lambda\phi_k^2)\,\d x
					+\alpha_k(2^\sharp-1)\!\int_{\Omega}V_{\varepsilon_k,\xi_k}^{2^\sharp-2}\phi_k^2\,\d x
					=(2^*-1)\!\int_{\Omega}V_{\varepsilon_k,\xi_k}^{2^*-2}\phi_k^2\,\d x\\
					&+(2^\sharp-1)\!\int_{\partial\Omega}V_{\varepsilon_k,\xi_k}^{2^\sharp-2}\phi_k^2\,\d\sigma
					+O(\varepsilon_k\|\phi_k\|)
					+O(\alpha_k\varepsilon_k\|\phi_k\|)
					+O(\|\phi_k\|^{r+1}),
				\end{aligned}
			\end{equation*}
			where
			$r=\min(2,2^\sharp-1)>1$
			is given by Lemma~\ref{lem:algebraic} with
			$L=1$ and $q=2^\sharp-1$.
		Rearranging the above identity, we obtain
			\begin{equation*}
				Q_{\varepsilon_k}(\phi_k)+\lambda\!\int_{\Omega}\phi_{k}^2\,\d x
				= O(\varepsilon_k\|\phi_k\|)+O(\alpha_{k}\varepsilon_k\|\phi_k\|)+O(\|\phi_k\|^{r+1}).
			\end{equation*}
			By Lemma~\ref{lem:spectral_gap},
			\[
			Q_{\varepsilon_k}(\phi_k)
			\ge
			\gamma\|\nabla\phi_k\|^2.
			\]
			Hence,
			\[
			\|\phi_k\|^2
			\le
			C\varepsilon_k\|\phi_k\|
			+C\alpha_k\varepsilon_k\|\phi_k\|
			+C\|\phi_k\|^{r+1}.
			\]
			Since $\alpha_k\to+\infty$,
			the first term is absorbed into the second one. Moreover,
			$r+1>2$ and $\phi_k\to0$ in $H^1(\Omega)$ imply that
	$
			\|\phi_k\|^{r+1}
			=o(\|\phi_k\|^2).
		$
			Therefore,
			\[
			\|\phi_k\|^2
			\le
			C\alpha_k\varepsilon_k\|\phi_k\|
			+o(\|\phi_k\|^2).
			\]
			Absorbing the last term into the left-hand side yields
			\[
			\|\phi_k\|^2
			=
			O(\alpha_k\varepsilon_k\|\phi_k\|).
			\]
		\end{proof}
		
		\smallskip
		We next prove Theorem \ref{thm:main} \textup{(i)--(ii)}.
\begin{proof}[Proof of Theorem \ref{thm:main}\, \normalfont{(i)--(ii)}]
		We argue by contradiction. Suppose that $\alpha_{0}=+\infty$. Then
		$S_{\alpha}<A$ for every $\alpha\ge0$. By
		Lemma~\ref{lem3.1}, for any sequence
		$\alpha_k\to+\infty$, there exist least-energy solutions
		$u_k\in\mathcal N_{\alpha_k}$ satisfying
		\[
		I_{\alpha_k}(u_k)=S_{\alpha_k}\to A,
		\qquad
		u_k\rightharpoonup0
		\quad\text{in }H^1(\Omega)\qquad \text{as } k\to\infty.
		\]
		By Lemma~\ref{lem:concentration_locus}, the sequence
		$\{u_k\}$ concentrates at a boundary point
		$P_0\in\partial\Omega$. More precisely,
		\[
		\delta_k=\kappa_n^{2/(n-2)}M_k^{-2/(n-2)}
		\to0,
		\qquad
		\alpha_k\delta_k\to0\qquad \text{as } k\to\infty,
		\]
		and
		\[
		\|\nabla(u_k-V_{\delta_k,P_k})\|_{L^2(\Omega)}
		\to0\qquad \text{as } k\to\infty.
		\]
		
		Applying Lemma~\ref{lem:final_estimate}, we obtain
		\[
		I_{\alpha_k}(u_k)
		\ge
		A-L(n)H(\xi_k)\varepsilon_k
		+B(n)\alpha_k\varepsilon_k
		+O(\varepsilon_k^2)
		+O(\alpha_k\varepsilon_k\|\phi_k\|)
		+O(\|\phi_k\|^2),
		\]
		where
		$\xi_k\to P_0$ and
		$\varepsilon_k/\delta_k\to1$.
		
	Moreover, Lemma~\ref{lem:phi_size} yields
$
	\|\phi_k\|=O(\alpha_k\varepsilon_k),
$
	and therefore
	\[
	O(\alpha_k\varepsilon_k\|\phi_k\|)
	+
	O(\|\phi_k\|^2)
	=
	O(\alpha_k^2\varepsilon_k^2)
	=
	o(\alpha_k\varepsilon_k),
	\]
	where the last relation follows from
	$\alpha_k\varepsilon_k\to0$ as $k\to\infty$.
		Therefore,
		\[
		I_{\alpha_k}(u_k)
		\ge
		A+
		\left[
		B(n)\alpha_k
		-
		L(n)H(\xi_k)
		\right]\varepsilon_k
		+
		O(\varepsilon_k^2)
		+
		o(\alpha_k\varepsilon_k).
		\]
			Since
		$H(\xi_k)\to H(P_0)$
		and
		$B(n)>0$,
		there exists $k_0$ such that
		\[
		B(n)\alpha_k-L(n)H(\xi_k)
		\ge
		\frac12B(n)\alpha_k
		\]
		for all $k\ge k_0$.
		Consequently,
		\[
		I_{\alpha_k}(u_k)
		\ge
		A
		+\frac{B(n)}2\alpha_k\varepsilon_k
		+O(\varepsilon_k^2)
		+o(\alpha_k\varepsilon_k).
		\]
		
		Finally, since
		$\alpha_k\varepsilon_k\to0$ as $k\to\infty$ and
		$\varepsilon_k^2=o(\alpha_k\varepsilon_k)$,
		the error terms are negligible compared with
		$\alpha_k\varepsilon_k$. Hence, for all sufficiently large $k$,
		\[
		I_{\alpha_k}(u_k)
		>
		A,
		\]
		which contradicts
		\[
		I_{\alpha_k}(u_k)
		=
		S_{\alpha_k}
		\le
		A.
		\]
		Therefore,
		$\alpha_0<+\infty$.
		
		The existence and nonexistence assertions in
		Theorem~\ref{thm:main}\textup{(i)--(ii)}
		now follow from
		Corollary~\ref{cor:existence_nonexistence}.
	\end{proof}


\section{Lower bound of $\alpha_0$ and the existence for $\alpha =\alpha_0$
	}\label{sec:section5}
	In this section, we analyze the critical case $\alpha=\alpha_0$
	and complete the proof of Theorem~\ref{thm:main}(iii) by establishing
	the following three lemmas.
	
\begin{Lemma}\label{lem:alpha_0_lower_bound}
	For every
$
	\alpha<C(n)\max_{\partial\Omega}H,
$
	we have
	\[
	S_{\alpha}<A.
	\]
\end{Lemma}
\begin{proof}
	Let $P\in\partial\Omega$ satisfy
	\[
	H(P)=\max_{x\in\partial\Omega}H(x).
	\]
	Since $V_{\varepsilon,P}\not\equiv0$,
	Lemma~\ref{lem:nehari-fibring}\textup{(i)} yields a unique
	$t_\varepsilon:=t_\alpha(V_{\varepsilon,P})>0$
	such that
	$t_\varepsilon V_{\varepsilon,P}\in\mathcal N_\alpha$.
	Moreover,
	\[
	t_\varepsilon=1+O(\varepsilon)
	\qquad\text{as }\varepsilon\to0^+.
	\]
	
	Applying Lemma~\ref{lem:expansion}, we obtain
	\[
	\begin{aligned}
		S_\alpha
		&\le I_\alpha(t_\varepsilon V_{\varepsilon,P})  \\
		&=I_\alpha(V_{\varepsilon,P})+O(\varepsilon^2) \\
		&=A-L(n)H(P)\varepsilon
		+B(n)\alpha\varepsilon
		+O(\varepsilon^2+\alpha\varepsilon^2) \\
		&=A
		-B(n)\alpha\varepsilon
		\left[
		\frac{L(n)H(P)}{B(n)\alpha}
		-1
		+o(1)
		\right]
		\qquad\text{as }\varepsilon\to0^+.
	\end{aligned}
	\]
	
	Since
	$
	\alpha
	<
	C(n)
	\max_{\partial\Omega}H
	=
	\frac{L(n)}{B(n)}H(P),
$
	the quantity inside the brackets is strictly positive for
	$\varepsilon>0$ sufficiently small. Consequently,
$
	S_\alpha<A.
$
\end{proof}
\begin{Corollary}
	The lower bound of $\alpha_{0}$  is $ C(n) \underset{\partial \Omega}{\max} H $.
\end{Corollary}
\begin{proof}
This proof can be directly obtained based on  Theorem~\ref{thm:main}\, \normalfont(i), the definition of $\alpha_0$ and Lemma \ref{lem:alpha_0_lower_bound}.
\end{proof}
\begin{Lemma}\label{lem:existence_at_threshold}
	 Problem~\eqref{eq:main} with $\alpha=\alpha_0$
	admits a least energy solution if $\alpha_{0}>C(n) \underset{\partial \Omega}{\max} H $.
\end{Lemma}
\begin{proof}
	Let $\{\alpha_k\}$ be any sequence satisfying
	$\alpha_k\nearrow\alpha_0$ as $k\to\infty$.
	For each $k$, let $u_k$ be a least-energy solution of
	\eqref{eq:main}, so that
$
	I_{\alpha_k}(u_k)=S_{\alpha_k}.
$
	Arguing as in the proof of
	Lemma~\ref{lem:existence-below-threshold}, we see that
	$\{u_k\}$ is bounded in $H^1(\Omega)$.
	Passing to a subsequence if necessary, we may assume that
	\[
	u_k\rightharpoonup u
	\ \text{in }H^1(\Omega)
	\quad\text{as }k\to\infty.
	\]
	
	We claim that $u\not\equiv0$.
	Assume, by contradiction, that $u\equiv0$.
	If $\{u_k\}$ were uniformly bounded in $L^\infty(\Omega)$, then
	$|u_k|^{2^*}\le C^{2^*-2}|u_k|^2,$
	which implies
	$
	\|u_k\|_{2^*,\Omega}\to0
$ 
	since $u_k\to0$ in $H^1(\Omega)$ as $k\to\infty$.
	Similarly,
$
	\|u_k\|_{2^\sharp,\partial\Omega}\to0.
$
	Using the Nehari identity,
	\[
	\|u_k\|^2+\alpha_k\|u_k\|_{2^\sharp,\Omega}^{2^\sharp}
	=\|u_k\|_{2^*,\Omega}^{2^*}
	+\|u_k\|_{2^\sharp,\partial\Omega}^{2^\sharp},
	\]
	we conclude that $\|u_k\|\to0$ as $k\to\infty$, contradicting the uniform lower bound
	$\|u_k\|\ge\delta$ on $\mathcal N_{\alpha_k}$.
	Therefore,
	\[
	\|u_k\|_{L^\infty(\Omega)}\to\infty\qquad\text{as }k\to\infty
	\]
	along a subsequence.
	
	Consequently, Lemmas~\ref{lem:final_estimate}
	and~\ref{lem:phi_size} apply and yield
	\[
	I_{\alpha_k}(u_k)
	\ge
	A+\bigl[B(n)\alpha_k-L(n)H(P_k)\bigr]\varepsilon_k
	+O(\varepsilon_k^2)
	+o(\alpha_k\varepsilon_k).
	\]
	Since
$
	H(P_k)\le\max_{\partial\Omega}H
$
	and
$
	\alpha_k\to\alpha_0>C(n)\max_{\partial\Omega}H$
as $k\to\infty,
$
	there exists $k_0$ such that
	\[
	B(n)\alpha_k-L(n)H(P_k)>0
	\qquad\text{for all }k\ge k_0.
	\]
	Hence,
	\[
	I_{\alpha_k}(u_k)>A
	\]
	for all sufficiently large $k$, contradicting
	\[
	S_{\alpha_k}=I_{\alpha_k}(u_k)<A.
	\]
	Therefore, $u\not\equiv0$.
	
	 It remains to prove that $u$ is a least energy solution at the threshold $\alpha_0$.
	 Set $v_k:=u_k-u$. Then $v_k\rightharpoonup0$ weakly in $H^1(\Omega).$
	 Passing to the limit in
	 \[
	 \langle I'_{\alpha_k}(u_k),\varphi\rangle=o(1),
	 \qquad
	 \varphi\in H^1(\Omega),
	 \]
	 yields
	 \[
	 \langle I'_{\alpha_0}(u),\varphi\rangle=0
	 \qquad
	 \text{for every }\varphi\in H^1(\Omega).
	 \]
	 Since $u\not\equiv0$, it follows that
$
	 u\in\mathcal N_{\alpha_0}.
$
	 The Nehari identity then yields
	 \[
	 I_{\alpha_0}(u)
	 =
	 \frac{1}{2(n-1)}\|u\|^{2}
	 +
	 \frac{n-2}{2n(n-1)}
	 \|u\|_{2^*,\Omega}^{2^*}
	 >0.
	 \]
	 	By Lemma~\ref{lem:monotone_continuous} and
	 	$\alpha_k\nearrow\alpha_0$, we have
	 	\[
	 	S_{\alpha_k}\to S_{\alpha_0}=A\qquad\text{as }k\to\infty.
	 	\]
	 	Moreover, the Brezis--Lieb lemma together with
	 $
	 	\|v_k\|_{2^\sharp,\Omega}^{2^\sharp}=o(1)
	 $
	 	yields
	 	\begin{equation}\label{eq:I-split-prop51}
	 		I_{\alpha_k}(u_k)
	 		=
	 		I_{\alpha_0}(u)
	 		+
	 		\Psi_\Omega(v_k)
	 		+
	 		o(1).
	 	\end{equation}
	 	Since
	 $I_{\alpha_k}(u_k)=S_{\alpha_k}\to A$ as $k\to\infty$
	 	and
	 	$
	 	I_{\alpha_0}(u)>0,
	 $
	 	it follows from \eqref{eq:I-split-prop51} that
\[
\Psi_\Omega(v_k)
\to
A-I_{\alpha_0}(u)
<A.
\]
	 	From $J_{\alpha_k}(u_k)=0$, $J_{\alpha_0}(u)=0$, and the
	 	Brezis--Lieb identities, we obtain
	 	\[
	 	\langle\Psi'_\Omega(v_k),v_k\rangle
	 	=\|\nabla v_k\|_2^2
	 	-\|v_k\|_{2^*}^{2^*}
	 	-\|v_k\|_{2^\sharp,\partial\Omega}^{2^\sharp}
	 	=o(1).
	 	\]
	 	Let $t_k>0$ be the unique critical point of
	 	$t\mapsto\Psi_\Omega(tv_k)$.
	 	Arguing exactly as in the proof of
	 	Lemma~\ref{lem:existence-below-threshold}, the critical point equation,
	 	together with
	 	\[
	 	\langle\Psi'_\Omega(v_k),v_k\rangle=o(1),
	 	\]
	 	implies that
	 	$t_k\to1$ as $k\to\infty$. Therefore,
	 	\[
	 	\liminf_{k\to\infty}\sup_{t>0}\Psi_\Omega(tv_k)
	 	=
	 	\liminf_{k\to\infty}\Psi_\Omega(t_kv_k)
	 	=
	 	\liminf_{k\to\infty}\Psi_\Omega(v_k)
	 	=
	 	A-I_{\alpha_0}(u)
	 	<
	 	A.
	 	\]
	Lemma~\ref{lem:local_compactness} then implies that
	\[
	v_k\to0
	\qquad\text{strongly in }H^1(\Omega).
	\]
	 	Therefore
	 	\[
	 	I_{\alpha_0}(u)
	 	=\lim_{k\to\infty}I_{\alpha_0}(u_k)
	 	=\lim_{k\to\infty}\Bigl[S_{\alpha_k}
	 	+\tfrac{\alpha_0-\alpha_k}{2^\sharp}
	 	\|u_k\|_{2^\sharp,\Omega}^{2^\sharp}\Bigr]
	 	=A=S_{\alpha_0},
	 	\]
	 	so $u$ is a least energy solution of~\eqref{eq:main} with $\alpha=\alpha_0$.
	
\end{proof}
As a result, we obtained the proof of Theorem \ref{thm:main}\,(iii).
\begin{proof}[Proof of Theorem \ref{thm:main}
  \normalfont(iii)]
  The proof can be directly derived from Lemma \ref{lem:existence_at_threshold}.
  \end{proof}
\begin{Lemma}
	Suppose $\alpha_{0}=C(n) \max\limits_{\partial\Omega}H.$
	Let $\{u_k\}$ be any sequence in which
		$u_k$ is a least energy solution of~\eqref{eq:main}
		with $\alpha=\alpha_k\nearrow\alpha_0$ as $k\to\infty$.
		Then at least one of the following holds:
	\begin{enumerate}
		\item[\textup{(i)}]
		There exists a least energy solution of~\eqref{eq:main}
		with $\alpha=\alpha_0$.
		\item[\textup{(ii)}]
		The sequence $\{u_k\}$ has a subsequence
			$\{u_{k_j}\}$ with
		$u_{k_j}\rightharpoonup 0$,\,
		$\|u_{k_j}\|_{L^{\infty}(\Omega)}\to\infty,$
		and every limit point of the maximum points
		of $u_{k_j}$
		lies in
		$\{\xi\in\partial\Omega:
		H(\xi)=\max\limits_{\partial\Omega}H\}.$
	\end{enumerate}
\end{Lemma}
	\begin{proof}
	Since $\alpha_k<\alpha_0$ and
	$\alpha_k\to\alpha_0$ as $k\to\infty$,
	Remark~\ref{rmk:alpha0_positive} yields
	$S_{\alpha_k}<A$.
	Hence, by Lemma~\ref{lem:existence-below-threshold}, there exists
	$u_k\in\mathcal N_{\alpha_k}$ satisfying
	\[
	I_{\alpha_k}(u_k)=S_{\alpha_k}.
	\]
	Arguing as in the proof of
	Lemma~\ref{lem:existence_at_threshold}, we conclude that
	$\{u_k\}$ is bounded in $H^1(\Omega)$.
	Passing to a subsequence if necessary, we may assume that $u_k\rightharpoonup u$
weakly in $H^1(\Omega).$

	\medskip
	\noindent
	\textit{Case A: $u\neq0$.}
	Since $\alpha_k\to\alpha_0$ and
	$u_k\rightharpoonup u\neq0$, the compactness argument used in
	Lemma~\ref{lem:existence_at_threshold} applies. Hence
	\[
	u_k\to u\quad\text{strongly in }H^1(\Omega),
	\]
	and passing to the limit in the Euler--Lagrange equation gives
$
	u\in\mathcal N_{\alpha_0}.
$
	Moreover,
	\[
	I_{\alpha_0}(u)=\lim_{k\to\infty}I_{\alpha_k}(u_k)=A.
	\]
	Hence $u$ is a least energy solution at $\alpha_0$,
	which proves~(i).
	
	\medskip
	\noindent
	\textit{Case B: $u=0$.}
	If $\{u_k\}$ were uniformly bounded in $L^\infty(\Omega)$,
	then the argument in the proof of
	Lemma~\ref{lem:existence_at_threshold},
	together with the Nehari identity,
	would imply that
	$\|u_k\|\to0$ as $k\to\infty$,
	contradicting the uniform lower bound
	$\|u_k\|\ge\delta>0$.
	Hence,
	\[
	\|u_k\|_{L^\infty(\Omega)}\to\infty
	\qquad\text{as }k\to\infty.
	\]
	Passing to a further subsequence,
	still denoted by $\{u_k\}$,
	Lemmas~\ref{lem:concentration_locus}
	and~\ref{lem:bubble_profile}
	yield points
	$P_k\in\partial\Omega$
	and scales
	$\delta_k\to0$
	such that
	\[
	\|\nabla(u_k-V_{\delta_k,P_k})\|_{2,\Omega}\to0
	\qquad\text{as }k\to\infty.
	\]
	Since
	$\alpha_k\to\alpha_0<+\infty$
	and
	$\delta_k\to0$,
	we have
	\[
	\alpha_k\delta_k\to0
	\qquad\text{as }k\to\infty.
	\]
	Applying Lemma~\ref{lem:final_estimate},
	we obtain
	$\varepsilon_k>0$ and
	$\xi_k\in\partial\Omega$
	with
	\[
	\frac{\varepsilon_k}{\delta_k}\to1
	\qquad\text{as }k\to\infty,
	\]
	and
	\[
	S_{\alpha_k}
	=
	I_{\alpha_k}(u_k)
	\ge
	A+
	\bigl[
	B(n)\alpha_k
	-
	L(n)H(P_k)
	\bigr]
	\varepsilon_k
	+
	O(\varepsilon_k^2)
	+
	o(\alpha_k\varepsilon_k).
	\]
	
		We claim that every accumulation point of
	$\{P_k\}$ belongs to $\left\{
\xi\in\partial\Omega:
H(\xi)=\max\limits_{\partial\Omega}H
\right\}.$
	Indeed, suppose that, along a subsequence,
	$P_k\to P^*\in\partial\Omega$
	with
$H(P^*)<\max\limits_{\partial\Omega}H.$	Since
\[
\alpha_k\to\alpha_0
=
C(n)\max_{\partial\Omega}H
\qquad\text{as }k\to\infty,
\]
and $P_k\to P^*$, we have
	\[
	B(n)\alpha_k
	-
	L(n)H(P_k)
	\to
	L(n)
	\bigl(
	\max_{\partial\Omega}H-H(P^*)
	\bigr)
	>0.
	\]
	Hence,
	\[
	S_{\alpha_k}
	\ge
	A+c\,\varepsilon_k
	>A
	\]
	for all sufficiently large $k$,
	contradicting
	$S_{\alpha_k}<A$. Therefore, every accumulation point of $\{P_k\}$ is a maximum point of
	the mean curvature on $\partial\Omega$.
\end{proof}

\noindent {\bf Data Availability} Data sharing is not applicable to this article as no new data were created or analyzed in this study.

\section*{Declarations}

\vspace{0.5em} 
\noindent {\bf Conflict of interest} The authors declare that they have no Conflict of interest.

\end{document}